\documentclass[reqno,11pt]{amsart}
\usepackage{amsmath,amsfonts,amssymb,amsxtra,latexsym,amscd,enumerate,amsthm,verbatim}

\usepackage{graphicx}

\usepackage[margin=1.40in]{geometry}
\numberwithin{equation}{section}

\newcommand{\R}{\mathbb{R}}

\newcommand{\E}{\mathcal{E}}

\newcommand{\G}{\mathcal{G}}
\newcommand{\T}{\mathbb{T}}
\newcommand{\C}{\mathbb{C}}
\newcommand{\Z}{\mathbb{Z}}
\newcommand{\eps}{\epsilon}

\newcommand{\va}{\vartheta}

\numberwithin{equation}{section} 

\newcommand{\lv}{b(0)}
\newcommand{\uv}{b(1)}

\newtheorem{theorem}{Theorem}[section]
\newtheorem{lemma}[theorem]{Lemma}
\newtheorem{proposition}[theorem]{Proposition}

\newtheorem{remark}[theorem]{Remark}

\begin{document}

\title{Nonlinear inviscid damping near monotonic shear flows}

\author{Alexandru D. Ionescu}
\address{Princeton University}
\email{aionescu@math.princeton.edu}

\author{Hao Jia}
\address{University of Minnesota}
\email{jia@umn.edu}

\thanks{The first author was supported in part by NSF grant DMS-1600028.  The second author was supported in part by DMS-1600779 .}

\begin{abstract}
We prove nonlinear asymptotic stability of a large class of monotonic shear flows among solutions of the 2D Euler equations in the channel $\mathbb{T}\times[0,1]$. More precisely, we consider shear flows $(b(y),0)$ given by a function $b$ which is Gevrey smooth, strictly increasing, and linear outside a compact subset of the interval $(0,1)$ (to avoid boundary contributions which are incompatible with inviscid damping). We also assume that the associated linearized operator satisfies a suitable spectral condition, which is needed to prove linear inviscid damping. 

Under these assumptions, we show that if $u$ is a solution which is a small and Gevrey smooth perturbation of such a shear flow $(b(y),0)$ at time $t=0$, then the velocity field $u$ converges strongly to a nearby shear flow as the time goes to infinity. This is the first nonlinear asymptotic stability result for Euler equations around general steady solutions for which the linearized flow cannot be explicitly solved.
\end{abstract}

\maketitle

\setcounter{tocdepth}{1}

\tableofcontents

\section{Introduction}

In this paper we continue our investigation of asymptotic stability of solutions of the two dimensional incompressible Euler equation in a channel. More precisely we consider solutions $u:[0,\infty)\times\T\times[0,1]\to\mathbb{R}^2$ of the equation
\begin{equation}\label{euler}
\partial_tu+u\cdot\nabla u+\nabla p=0,\qquad {\rm div}\,u=0,
\end{equation}
with the boundary condition $u^y|_{y=0,\,1}\equiv 0$. Letting
$\omega:=-\partial_yu^x+\partial_xu^y$
be the vorticity field, the equation (\ref{euler}) can be written in vorticity form as
\begin{equation}\label{Euler}
  \partial_t\omega+u\cdot\nabla \omega=0,\qquad u=\nabla^{\perp}\psi=(-\partial_y\psi, \partial_x\psi),
\end{equation}
for $(x,y)\in \mathbb{T}\times [0,1], t\ge0$, where the stream function $\psi$ is determined through 
\begin{equation}\label{stream}
 \Delta \psi=\omega,\quad{\rm on}\,\,\mathbb{T}\times[0,1],\qquad \psi(x,0)\equiv 0,\,\psi(x,1)\equiv C_0,
\end{equation}
where $C_0$ is a constant preserved by the flow. 

The two dimensional incompressible Euler equation is globally well-posed for smooth initial data, by the classical result of Wolibner \cite{Wolibner}. See also \cite{Yudovich1, Yudovich2} for global well-posedness results with rough initial data, such as $L^\infty$ vorticity. The long time behavior of general solutions is however very difficult to understand, due to the lack of a global relaxation mechanism.

A more realistic goal is to study the global nonlinear dynamics of solutions that are close to steady states of the 2D Euler equation. Coherent structures, such as shear flows and vortices, are particularly important in the study of the 2D Euler equation, since precise numerical simulations and physical experiments show that they tend to form dynamically and become the dominant feature of the solution for a long time.

The study of stability property of these steady states is a classical subject and a fundamental problem in hydrodynamics. Early investigations were started by Kelvin \cite{Kelvin}, Orr \cite{Orr}, Rayleigh \cite{Ray}, Taylor \cite{Taylor}, among many others, with a focus on mode stability. Later, more detailed understanding of the general spectral properties and suitable linear decay estimates were also obtained, see \cite{Faddeev,Stepin}. In the direction of nonlinear results, Arnold \cite{Arnold} proved a general stability criteria, using the energy Casimir method, but this method does not give asymptotic information on the global dynamics.

The full nonlinear asymptotic stability problem has only been investigated in recent years, starting with the remarkable work of Bedrossian--Masmoudi \cite{BM}, who proved inviscid damping and nonlinear stability in the simplest case of perturbations of the Couette flow on $\T\times\R$. 

Motivated by this result, the linearized equations around other stationary solutions were investigated intensely in the last few years, and linear inviscid damping and decay was proved in many cases of physical interest, see for example \cite{Bed2,Zillinger3,Grenier,JiaL,dongyi,Dongyi2, Dongyi3,Zillinger1,Zillinger2}. However, it also became clear that there are major difficulties in passing from linear to nonlinear stability, such as the presence of ``resonant times" in the nonlinear problem, which require refined Fourier ana\-lysis techniques, and the fact that the final state of the flow is determined dynamically by the global evolution and cannot be described in terms of the initial data.

In this paper we close this gap and establish inviscid damping and full nonlinear asymptotic stability for a general class of monotone shear flows, which are not close to the Couette flow. We hope that the general framework we develop here can be adapted to establish nonlinear asymptotic stability in other outstanding open problems involving 2D or 3D Euler and Navier-Stokes equations, such as the stability of smooth radially decreasing vortices in 2D.

\subsection{The main theorem} We consider a perturbative regime for the Euler equation (\ref{euler}), with velocity field given by $(b(y),0))+u(x,y)$ and vorticity given by $-b'(y)+\omega$. 

To state our main theorem we define the Gevrey spaces $\mathcal{G}^{\lambda,s}\big(\mathbb{T}\times \R\big)$ as the spaces of $L^2$ functions $f$ on 
$\mathbb{T}\times \R$ defined by the norm
\begin{equation}\label{Gev}
\|f\|_{\G^{\lambda,s}(\mathbb{T}\times \R)}:=\big\|e^{\lambda \langle k,\xi\rangle^s}\widetilde{f}(k,\xi)\big\|_{L^2_{k,\xi}}<\infty,\qquad s\in(0,1], \lambda>0.
\end{equation}
In the above $(k,\xi)\in \mathbb{Z}\times\mathbb{R}$ and $\widetilde{f}$ denotes the Fourier transform of $f$ in $(x,y)$. More generally, for any interval $I\subseteq\R$ we define the Gevrey spaces $\mathcal{G}^{\lambda,s}\big(\mathbb{T}\times I\big)$ by
\begin{equation}\label{Gev2}
\|f\|_{\G^{\lambda,s}(\mathbb{T}\times I)}:=\|Ef\|_{\G^{\lambda,s}(\mathbb{T}\times \R)},
\end{equation}
where $Ef(x):=f(x)$ if $x\in I$ and $Ef(x):=0$ if $x\notin I$.

Concerning the background shear flow $b\in C^{\infty}(\mathbb{R})$, our main assumptions are the following:
\medskip

$(A)$ For some $\vartheta_0\in(0,1/10]$ and $\beta_0>0$
\begin{equation}\label{IntB}
\vartheta_0\leq b'(y)\leq 1/\vartheta_0\text{ for }y\in [0,1]\qquad {\rm and}\qquad b''(y)\equiv 0\,\,{\rm for}\,\,y\notin [2\vartheta_0,1-2\vartheta_0],
\end{equation}
and 
\begin{equation}\label{IntB1}
\|b\|_{L^{\infty}(0,1)}+\|b''\|_{\mathcal{G}^{\beta_0,1/2}}\leq 1/\vartheta_0.
\end{equation}

$(B)$ The associated linearized operator $L_k:L^2(0,1)\to L^2(0,1), k\in\mathbb{Z}\backslash\{0\}$, given by
\begin{equation}\label{IntB2}
L_kf= b(y)f-b''(y)\varphi_k, \qquad\text{ where }\,\,\,\partial_y^2\varphi_k-k^2\varphi_k=f,\,\,\,\varphi_k(0)=\varphi_k(1)=0,
\end{equation}
has no discrete eigenvalues and therefore, by the general theory of Fredholm operators, the spectrum of $L_k$ is purely continuous spectrum $[b(0),b(1)]$ for all $k\in\mathbb{Z}\backslash\{0\}$.
\medskip

The spectral condition $(B)$ is a qualitative condition, and we need to make it quantitative in order to link it to the perturbation theory. For this we define, for any $k\in\mathbb{Z}\setminus\{0\}$,
\begin{equation}\label{T1.1}
\|f\|_{H^1_k(\R)}:=\|f\|_{L^2(\R)}+|k|^{-1}\|f'\|_{L^2(\R)}.
\end{equation}
The following quantitative bounds were proved in \cite[Lemmas 3.1 and 3.2]{JiaG}.

\begin{lemma}\label{T5}
Assume $\varphi\in H^{10}$ is supported in $[\vartheta_0/4,1-\vartheta_0/4]$. For $k\in\mathbb{Z}\backslash\{0\},y_0\in[0,1],\epsilon\in[-1/4,1/4]\backslash\{0\}$ and any $f\in L^2(0,1)$ we define the operator
\begin{equation}\label{T1}
T_{k,y_0,\epsilon}f(y):=\int_{\R}\varphi(y)G_k(y,z)\frac{b''(z)f(z)}{b(z)-b(y_0)+i\epsilon}\,dz,
\end{equation}
where $G_k$ is the Green function associated to the operator $-\partial_y^2+k^2$ on $[0,1]$ (see \eqref{elli11} for explicit formulas).
Then there is $\kappa>0$ such that, for any $f\in H^1_k(\mathbb{R})$,
\begin{equation}\label{T6}
\|T_{k,y_0,\epsilon}f\|_{H^1_k(\R)}\lesssim |k|^{-1/3}\|f\|_{H^1_k(\R)},\qquad \|f+T_{k,y_0,\epsilon}f\|_{H^1_k(\R)}\ge \kappa \|f\|_{H^1_k(\R)},
\end{equation}
uniformly in $y_0\in[0,1]$, $ k\in\mathbb{Z}\backslash\{0\}$, and $\epsilon$ sufficiently small. 
\end{lemma}

In our case, the function $\varphi$ will be a fixed Gevrey cutoff function, $\varphi(y)=\Psi(b(y))$, where $\Psi$ is defined in \eqref{rec0}. The parameter $\kappa>0$ in \eqref{T6} will be one of the parameters that determine the smallness of the perturbation in our main theorem.

 For any function $H(x,y)$ let $\langle H\rangle(y)$ denote the average of $H$ in $x$. Our main result in this paper is the following theorem:

\begin{theorem}\label{maintheoremINTRO}
Assume that $\beta_0,\vartheta_0,\kappa>0$ are constants as defined in \eqref{IntB}, \eqref{IntB1}, and \eqref{T6}. Then there are constants $\beta_1=\beta_1(\beta_0,\vartheta_0,\kappa)>0$  and $\overline{\eps}=\overline{\eps}(\beta_0,\vartheta_0,\kappa)>0$ such that the following statement is true:

Assume that the initial data $\omega_0$ has compact support in $\T\times [2\vartheta_0,1-2\vartheta_0]$, and satisfies 
\begin{equation}\label{Eur0}
\|\omega_0\|_{\G^{\beta_0,1/2}(\mathbb{T}\times \R)}=\epsilon\leq\overline{\epsilon},\qquad \int_{\T} \omega_0(x,y)\,dx=0\,\,\text{ for any }y\in[0,1].
\end{equation}
Let $\omega:[0,\infty)\times\mathbb{T}\times [0,1]\to\mathbb{R}$ denote the global smooth solution to the Euler equation
 \begin{equation}\label{IEur1}
\begin{cases}
 &\partial_t\omega+b(y)\partial_x\omega-b''(y)\partial_x\psi+u\cdot\nabla\omega=0,\\
 &u=(u^x,u^y)=(-\partial_y\psi,\partial_x\psi),\qquad \Delta\psi=\omega,\qquad \psi(t,x,0)=\psi(t,x,1)=0.
\end{cases}
 \end{equation}
Then we have the following conclusions:

(i) For all $t\ge 0$, ${\rm supp}\,\omega(t)\subseteq \mathbb{T}\times[\vartheta_0,1-\vartheta_0]$.

(ii) There exists $F_{\infty}(x,y) \in \G^{\beta_1,1/2}$ with ${\rm supp}\,F_{\infty}\subseteq \mathbb{T}\times [\vartheta_0,1-\vartheta_0]$ such that for all $t\ge 0$,
\begin{equation}\label{convergence}
\left\|\omega(t,x+tb(y)+\Phi(t,y),y)-F_{\infty}(x,y)\right\|_{\G^{\beta_1,1/2}(\mathbb{T}\times[0,1])} \lesssim_{\beta_0,\vartheta_0,\kappa}\frac{\epsilon}{\langle t\rangle},
\end{equation}
where
\begin{equation}\label{DefPhi}
\Phi(t,y):=\int_0^t\langle u^x\rangle(\tau,y)\,d\tau.
\end{equation}

(iii) We define the smooth functions $\psi_\infty,u_\infty:[0,1]\to\mathbb{R}$ by
\begin{equation}\label{AsymPhi2}
\partial_y^2\psi_\infty=\langle F_\infty\rangle,\qquad \psi_\infty(0)=\psi_\infty(1)=1,\qquad u_{\infty}(y):=-\partial_y\psi_\infty.
\end{equation}
Then the velocity field $u=(u^x,u^y)$ satisfies
\begin{equation}\label{convergenceofux}
\left\|\langle u^x\rangle(t,y)-u_{\infty}(y)\right\|_{\G^{\beta_1,1/2}(\mathbb{T}\times [0,1])}\lesssim_{\beta_0,\vartheta_0,\kappa}\frac{\epsilon}{\langle t\rangle^2},
\end{equation}

\begin{equation}\label{convergencetomean}
\left\|u^x(t,x,y)-\langle u^x\rangle(t,y)\right\|_{L^{\infty}(\mathbb{T}\times [0,1])}\lesssim_{\beta_0,\vartheta_0,\kappa}\frac{\epsilon}{\langle t\rangle},
\end{equation}

\begin{equation}\label{convergenceuy}
\left\|u^y(t,x,y)\right\|_{L^{\infty}(\mathbb{T}\times [0,1])}\lesssim_{\beta_0,\vartheta_0,\kappa}\frac{\epsilon}{\langle t\rangle^2}.
\end{equation}
\end{theorem}

\subsection{Remarks}\label{remarksintro} We discuss now some of the assumptions and the conclusions of Theorem \ref{maintheoremINTRO}.

(1) The equation \eqref{IEur1} for the vorticity deviation is equivalent to the original Euler equations \eqref{euler}--\eqref{stream}. The condition $\int_{\T}\omega_0(x,y)\,dx=0$ can be imposed without loss of generality, because we may replace the shear flow $b(y)$ by the nearby shear flow $b(y)+\langle u_0^x\rangle(y)$. In fact, since $\partial_y\langle\partial_y\psi\rangle=\langle\omega\rangle$, this condition is equivalent to 
\begin{equation}\label{Intro7}
\langle u^x_0\rangle (y)=0\qquad\text{ for any }y\in[0,1].
\end{equation}
These identities only hold for the initial data, and are not propagated by the flow \eqref{IEur1}. However, as we show in (\ref{rea8}) below, we have 
\begin{equation*}
\langle u^x\rangle(t,y)\equiv 0\qquad\text{ for }\qquad y\in [0,1]\setminus[\vartheta_0,1-\vartheta_0]\text{ and }t\in[0,T],
\end{equation*}
as long as the vorticity $\omega$ is supported in $[0,T]\times\mathbb{T}\times[\vartheta_0,1-\vartheta_0]$. In particular, $\langle u^x\rangle(t,y)-u_{\infty}(y)$ is compactly supported in $[\vartheta_0,1-\vartheta_0]$.

(2) The assumption on the compact support of $\omega_0$ is likely necessary to prove scattering in Gevrey spaces. Indeed, Zillinger  \cite{Zillinger2} showed that scattering does not hold in high Sobolev spaces unless one assumes that the vorticity vanishes at high order at the boundary. This is due to what is called ``boundary effect", which is not consistent with inviscid damping. This boundary effect can also be seen clearly in \cite{JiaL} as the main asymptotic term for the stream function. 

Understanding quantitatively the boundary effect in the context of asymptotic stability of Euler or Navier-Stokes equations is an interesting topic by itself, but we will not address it here. 

The assumption on the support of $b''$ is necessary to preserve the compact support of $\omega(t)$ in $\mathbb{T}\times[\vartheta_0,1-\vartheta_0]$, due to the nonlocal term $b''(y)\partial_x\psi$ in \eqref{IEur1}. In principle, one could hope to remove this strong assumption (and replace it with a milder decay assumption) by working in the infinite cylinder $\mathbb{T}\times\R$ domain instead of the finite channel $\mathbb{T}\times [0,1]$, but this would be at the expense of considering solutions of infinite energy.

(3) There are a large class of shear flows $b$ satisfying our assumptions. For instance,  for any $b(y)$ which satisfies $|b'|\ge 1$ and $|b'''|<1$, then the spectrum of $L_k$ consists entirely of the continuous spectrum $[b(0),b(1)]$ for $k\in\mathbb{Z}\backslash\{0\}$.

(4) The Gevrey regularity assumption \eqref{Eur0} on the initial data $\omega_0$ is likely sharp. See the recent construction of nonlinear instability of Deng--Masmoudi \cite{Deng} for the Couette flow in slightly larger Gevrey spaces, and the more definitive counter-examples to inviscid damping in low Sobolev spaces by Lin--Zeng \cite{ZhiWu}. 

(5) At the qualitative level, our main conclusion \eqref{convergence} shows that the vorticity $\omega$ converges weakly to the function $\langle F_\infty\rangle(y)$. This is consistent with a far-reaching conjecture regarding the long time behavior of the $2D$ Euler equation, see \cite{SverakNotes}, which predicts that for general generic solutions the vorticity field converges, as $t\to\infty$, {\it weakly but not strongly} in $L^2_{{\rm loc}}$ to a steady state. Proving such a conjecture for general solutions is, of course, well beyond the current PDE techniques, but the nonlinear asymptotic stability results we have so far in \cite{BM, IOJI, IOJI2} are consistent with this conjecture.

(6) There are several parameters in our proof, and we summarize their roles here. The parameters $\beta_0,\vartheta_0,\kappa>0$ (the structural constants of the problem) are assumed fixed, and implicit constants in inequalities like $A\lesssim B$ are allowed to depend on these parameters. We will later fix a constant $\delta_0>0$ sufficiently small depending on these parameters, as part of the construction of our main weights, see \eqref{reb10.5}. 

The weights will also depend on a small parameter $\delta>0$, much smaller than $\delta_0$, which is needed at many places, such as in commutator estimates using inequalities like \eqref{S1intro}. We will use the general notation $A\lesssim_\delta B$ to indicate inequalities where the implicit constants may depend on $\delta$. Finally, the parameters $\eps$ and $\eps_1=\eps^{2/3}$, which bound the size of the perturbation, are assumed to be much smaller than $\delta$.

\subsection{Previous work and related results}

The study of stability properties of shear flows and vortices is one of the most important problems in hydrodynamics, and has a long history starting with work of Kelvin \cite{Kelvin}, Rayleigh \cite{Ray}, and Orr \cite{Orr}. The problem is well motivated physically, since numerical simulations and physical experiments, such as those of \cite{Ba1,Ba2,Benzi,Brachet,McW,McW1,Santanqelo}, show that coherent structures tend to form and become the dominant feature of incompressible 2D Euler evolutions. This indicates a reverse cascade of energy from high to low frequencies, which is in sharp contrast to the 3D situation, where it is expected that energy flow from small frequencies to high frequencies until the dissipation scale.  

We refer also to the recent papers \cite{KiselevSverak,ZhiWu2} for other interesting results concerning the dynamics of solutions of the 2D Euler equations.


Our main topic in this paper is asymptotic stability. Nonlinear asymptotic stability results are difficult for the $2D$ incompressible Euler equation, because the rate of stabilization is slow, the convergence of the vorticity field holds only in the weak sense, and the nonlinear effect is strong. In a recent remarkable paper Bedrossian--Masmoudi \cite{BM} proved the first nonlinear asymptotic stability result, showing that small perturbations of the Couette flow on the infinite cylinder $\T\times\R$ converge weakly to nearby shear flows. This result was extended by the authors \cite{IOJI} to the finite channel $\T\times[0,1]$, in order to be able to consider solutions with finite energy. In \cite{IOJI2} the authors also proved asymptotic stability of point vortex solutions in $\mathbb{R}^2$, showing that small and Gevrey smooth perturbations converge to a smooth radial profile, and the position of the point vortex stabilizes rapidly and forms the center of the final radial profile. These three results appear to be the only known results on nonlinear asymptotic stability of stationary solutions for the Euler equations. 

A key common feature of these stability results is that the steady states are simple explicit functions, and, more importantly, the associated linearized flow can be solved explicitly. 

To expand the stability theory to more general steady states, one can first consider the linearized equation and prove inviscid damping of linear solutions. The linear evolution problem has been investigated intensely in the last few years, in particular around general shear flows and vortices, see for example \cite{Grenier,JiaL,Zillinger1,Zillinger2,dongyi}. In particular, Wei-Zhang-Zhao \cite{dongyi} proved optimal decay rate of the stream function for the linearized problem near monotone shear flows, and Bedrossian-Coti Zelati-Vicol \cite{Bed2} obtained sharp decay estimates for general vortices with decreasing profile.  We also refer the reader to important developments for the linear inviscid damping in the case of non-monotone shear flows \cite{Dongyi2, Dongyi3} and circular flows \cite{Bed2,Zillinger3}.

There is a large gap, however, between linear and nonlinear theory. As we know, even in the simplest case of the Couette flow, to prove nonlinear stability one needs to bound the contribution of the so-called ``resonant times", which can only be detected by working in the Fourier space, in a specific coordinate system. This requires refined Fourier analysis techniques, which are not compatible with the natural spectral theory of the variable-coefficient linearized problems associated to general shear flows and vortices. In addition, nonlinear decay comes at the expense of loss of regularity, and one needs a subtle interplay of energy functionals with suitable weights (in the Fourier space) to successfully close the argument.

This gap was bridged in part by the second author in \cite{JiaG}, who proved a precise linear result, which combined Fourier analysis and spectral analysis, and provided accurate estimates that are compatible with nonlinear analysis. In this paper we close this gap completely in one important case, namely the case of monotone shear flows satisfying a suitable spectral assumption.

The problem of nonlinear inviscid damping we consider here is connected to the well-known Landau damping effect for Vlasov-Poisson equations, and we refer to the celebrated work of Mouhot--Villani \cite{Villani} for the physical background and more references. Inviscid damping is a very subtle mechanism of stability, and has only been proved rigorously in 2D for Euler-type equations. It can also be viewed as the limiting case of the Navier-Stokes equation with small viscosity $\nu>0$. In the presence of viscosity, one can have more robust stability results for initial data that is sufficiently small relative to $\nu$, which exploit the enhanced dissipation due to the mixing of the fluid. See \cite{Bed4,Bed7,Dongyi4} and references therein. Moreover, in the limit $\nu\to 0$ and if there is boundary then the boundary layer becomes an important issue, and there are significant additional difficulties. We refer the interested reader to \cite{Bed11,Qchen} for more details and further references.

\subsection{Main ideas}\label{subsection:reviewBedMas} We describe now some of the main ideas involved in the proof.

\subsubsection{Renormalization and time-dependent energy functionals} These are two key ideas introduced by Bedrossian--Masmoudi \cite{BM} in the case of Couette flow. 

As in \cite{BM} and \cite{IOJI,IOJI2}, we make a nonlinear change of variable, and define $z,v$ by 
\begin{equation}\label{newCoordInt}
v(t,y):=b(y)+\frac{1}{t}\int_0^t\langle u^x\rangle(s,y)\,ds,\qquad z(t,y):=x-tv(t,y).
\end{equation}
The main point is to remove the terms containing the non-decaying components $b(y)\partial_x\omega$ and  $\langle u^x\rangle\partial_x\omega$ from the evolution equation satisfied by the renormalized vorticity. 
Denote
\begin{equation}\label{Irea1}
F(t,z,v):=\omega(t,x,y),\qquad \phi(t,z,v):=\psi(t,x,y).
\end{equation}
Under this change of variable, the equation (\ref{IEur1}) becomes
\begin{equation}\label{Imain}
 \partial_tF-B''\partial_z\phi-V'\partial_vP_{\neq 0}\phi\,\partial_zF+(\dot{V}+V'\partial_z\phi)\,\partial_vF=0,
\end{equation}
where $P_{\neq 0}$ is projection off the zero mode. The coefficients $B'', V', \dot{V}, V'$ are suitable coordinate functions, connected to the change of variable \eqref{newCoordInt}, see \eqref{rea1}-\eqref{rea3'} for the precise definitions.

The main idea is to control the regularity of $F$ for all $t\ge0$, as well as all the other quantities such as $V',V'',B'',\dot{V},\phi$, using a bootstrap argument involving nine time dependent energy functionals and space-time norms. These norms depend on a family of weights 
\begin{equation}\label{weiIntro}
A_k(t,\xi),\,\,A_{NR}(t,\xi),\,\,A_R(t,\xi),\qquad k\in\mathbb{Z},\xi\in\R,
\end{equation}
which have to be designed carefully to control the nonlinearities (in particular, the difficult ``reaction term" $\partial_v\mathbb{P}_{\neq0}\phi\cdot\partial_zF$ in \eqref{Imain}, which cannot be estimated using standard weights due to loss of derivatives around certain ``resonant times"). See the longer discussion in \cite{BM} and \cite[Section 1.3]{IOJI}. 

The special weights we use here are the same as the weights we used in our earlier work \cite{IOJI,IOJI2}, and we rely on many estimates proved in these papers. Our weights are refinements of the weights of \cite{BM}, but depend on an additional small parameter $\delta$ which gives critical flexibility at several stages of the argument (such as inequalities like \eqref{S1intro} below which are needed for commutator estimates).
  
\subsubsection{The auxiliary nonlinear profile} In the case of general shear flows, an essential new difficulty that is not present in the Couette case or the point vortex case, is the additional linear term $B''(t,v)\partial_z\phi$ in \eqref{Imain}. This extra linear term can not be treated as a perturbation if $b''$ is not assumed small. We deal with this basic issue in two steps: first we define an auxiliary nonlinear profile $F^{\ast}(t)$ given by 
\begin{equation}\label{InF-F}
F^{\ast}(t,z,v)= F(t,z,v)-\int_0^tB''(0,v)\partial_z\phi'(s,z,v)\,ds.
\end{equation}
Thus $F^{\ast}$ takes into account the linear effect accumulated up to time $t$ and can be bounded perturbatively. The function $\phi'$ is a small but crucial modification of $\phi$, obtained by freezing the coefficients of the elliptic equation defining stream functions at time $t=0$, in order to keep these coefficients very smooth. See \eqref{Ph1}-\eqref{reb14} for the precise definitions. 

The modified profile $F^{\ast}$ now evolves in a perturbative fashion, and can be bounded using the method in \cite{IOJI,IOJI2}. However, this construction leads to loss of symmetry in the transport terms $V'\partial_vP_{\neq 0}\phi\,\partial_zF$ and $(\dot{V}+V'\partial_z\phi)\,\partial_vF$, since the main perturbative variable is now $F^{\ast}$. This loss of symmetry causes a derivative loss, so we need to prove stronger bounds on $F-F^{\ast}$ than on the variables $F, F^{\ast}$, as described in \eqref{rec1''}.

\subsubsection{Control of the full profile} We still need to recover the bounds on $F$ and the improved bounds on $F-F^{\ast}$. Since the bounds on $F^{\ast}$ are already proved, it suffices to prove the improved bounds \eqref{boot3} for $F-F^{\ast}$. 

This is a critical step where we need to use our main spectral assumption and the precise estimates on the linearized flow. To link $F-F^{\ast}$ with the linearized flow, we define an auxiliary function $\phi^{\ast}$, which can be approximately viewed as a stream function associated with $F^{\ast}$, see \eqref{B27} for the precise definitions. Now setting $g=F-F^{\ast}$, $\varphi:=\phi'-\phi^{\ast}$, the functions $g$ and $\varphi$ satisfy the inhomogeneous linear system with trivial initial data
\begin{equation}\label{Intg}
\begin{split}
&\partial_tg-B''_0(v)\partial_z\varphi=H,\qquad g(0,z,v)=0,\\
&B'_0(v)^2(\partial_v-t\partial_z)^2\varphi+B''_0(v)(\partial_v-t\partial_z)\varphi+\partial_z^2\varphi=g(t,z,v),
\end{split}
\end{equation}
where $(t,z,v)\in[0,\infty)\times\mathbb{T}\times[b(0),b(1)]$. The functions $B'_0(v)=B'(0,v)$ and $B''_0(v)=B''(0,v)$ are time-independent, very smooth, and can be expressed in terms of the original shear flow $b$. The source term $H$ is given by $H=B''_0(v)\partial_z\phi^{\ast}$.

The function $\phi^{\ast}$ is determined by the auxiliary profile $F^\ast$. Since we have already proved quadratic bounds on the profile $F^\ast$, we can use elliptic estimates to prove quadratic bounds on $\phi^{\ast}$, and then on the source term $H$. Therefore, we can think of \eqref{Intg} as a linear inhomogeneous system with trivial initial data, and attempt to adapt the linear theory to our situation. 

Decomposing in modes, conjugating by $e^{-ikvt}$, and using Duhamel's formula, we can further reduce to the study of the homogeneous initial-value problem 
\begin{equation}\label{B1intro}
\begin{split}
&\partial_tg_k+ikvg_k-ikB''_0\varphi_k=0,\qquad g_k(0,v)=X_k(v)e^{-ikav},\\
&(B'_0)^2\partial_v^2\varphi_k+B_0''(v)\partial_v\varphi_k-k^2\varphi_k=g_k,\qquad \varphi_k(b(0))=\varphi_k(b(1))=0.
\end{split}
\end{equation}
for $(t,v)\in[0,\infty)\times[b(0),b(1)]$, where $k\in\mathbb{Z}\setminus\{0\}$ and $a\in\R$.

\subsubsection{Analysis of the  linearized flow}   The equation \eqref{B1intro} was analyzed, at least when $a=0$, by Wei--Zhang--Zhao in \cite{dongyi} and by the second author in \cite{JiaG}. We follow the approach in \cite{JiaG}. The main idea is to use the spectral representation formula and reduce the analysis of the linearized flow to the analysis of generalized eigenfunctions corresponding to the continuous spectrum. More precisely, given data $X_k$ smooth and satisfying ${\rm supp}\,X_k\subseteq[b(\vartheta_0),b(1-\vartheta_0)]$ we find a representation formula
\begin{equation}\label{B7intro}
\widetilde{g_k}(t,\xi)=\widetilde{X_k}(\xi+kt+ka)+ik\int_0^t\int_{\R}\widetilde{B_0''}(\zeta)\widetilde{\,\,\Pi_k'}(\xi+kt-\zeta-k\tau,\xi+kt-\zeta,a)\,d\zeta\,d\tau
\end{equation}
for the solution $g_k$ of the linear evolution equation \eqref{B1intro}, where $\Pi_k'(\xi,\eta,a)$ can be expressed in terms of a family of generalized eigenfunctions. As proved in \cite{JiaG}, these eigenfunctions cannot be calculated explicitly, but can be estimated very precisely in the Fourier space,
\begin{equation}\label{B5intro}
\left\|(|k|+|\xi|)W_k(\eta+ka)\widetilde{\Pi'_k}(\xi,\eta,a)\right\|_{L^2_{\xi,\eta}}\lesssim_{\delta} \big\|W_k(\eta)\widetilde{X_k}(\eta)\big\|_{L^2_\eta},
\end{equation}
for any $a\in\mathbb{R}$, where $W_k$ is a family of weights satisfying smoothness properties of the type
\begin{equation}\label{S1intro}
\left|W_k(\xi)-W_k(\eta)\right|\lesssim e^{2\delta_0\langle\xi-\eta\rangle^{1/2}}W_k(\eta)\Big[\frac{C(\delta)}{\langle k,\eta\rangle^{1/8}}+\sqrt{\delta}\Big]\qquad \text{ for any }\xi,\eta\in\mathbb{R}.
\end{equation}

The inequality \eqref{S1intro} holds for standard weights, like polynomial weights $W_k(\xi)=(1+|\xi|^2)^{N/2}$, which correspond to Sobolev spaces, or exponential weights $W_k(\xi)=e^{\lambda\langle\xi\rangle^s}$, $s<1/2$, which correspond to Gevrey spaces. More importantly, it also holds for our carefully designed weights $A_k(t,\xi)$, as we have already seen in \cite{IOJI2}. This allows us to adapt and incorporate the linear theory, and close the argument.

\subsection{Organization} The rest of the paper is organized as follows. In section 2 we renormalize the variables using a nonlinear change of coordinates and set up the main bootstrap Proposition \ref{MainBootstrap}.  In section 3 we collect some lemmas concerning Gevrey spaces and describe in detail our main weights $A_k$, $A_R$, and $A_{NR}$. In section 4 we prove several bilinear estimates and, more importantly, an elliptic estimate that can be applied many times to control stream-like functions. In sections 5-7 we prove the main bootstrap Proposition \ref{MainBootstrap}. In section 8 we prove the main estimates on the linear flow, by adapting the analysis in \cite{JiaG}. Finally, in section 9 we use the main bootstrap proposition to complete the proof of Theorem \ref{maintheoremINTRO}.

\section{The main bootstrap proposition}

\subsection{Renormalization and the new equations}\label{sec:variables}

Assume that $\omega:[0,T]\times\mathbb{T}\times[0,1]$ is a sufficiently smooth solution of the system
\begin{equation}\label{Eur1}
\begin{split}
&\partial_t\omega+b(y)\partial_x\omega-b''(y)\partial_x\psi+u\cdot\nabla\omega=0,\\
& (u^x,u^y)=(-\partial_y\psi,\partial_x\psi),\qquad\Delta\psi=\omega,\qquad \psi(t,x,1)=\psi(t,x,0)=0,
\end{split}
\end{equation} 
which is supported in $\mathbb{T}\times[\va_0,1-\va_0]$ at all times $t\in[0,T]$, satisfying $\|\langle\omega\rangle(t)\|_{H^{10}}\ll 1$ and
\begin{equation}\label{Eur1.1}
\int_{\mathbb{T}}u^x(0,x,y)\,dx=0\qquad\text{ for any }y\in[0,1].
\end{equation}
Using \eqref{Eur1}--\eqref{Eur1.1} it is easy to show that 
\begin{equation}\label{rea8}
\int_{\mathbb{T}}u^x(t,x,y)\,dx\equiv0\qquad\text{ for any }t\in[0,T]\text{ and }y\in[0,\vartheta_0]\cup [1-\vartheta_0, 1].
\end{equation}
Indeed, since $u^x=-\partial_y\psi$ and $\Delta\psi=\omega$, we have $\partial_y\langle u^x\rangle=-\langle\omega\rangle$. We also have $\partial_t\langle u^x\rangle=\langle\omega\partial_x\psi\rangle$ (see the proof of \eqref{conv2.3} below), and the desired identities \eqref{rea8} follow using the support assumption on $\omega$. 

As in \cite{BM, IOJI, IOJI2}, we make the nonlinear change of variables 
\begin{equation}\label{changeofvariables1}
 v=b(y)+\frac{1}{t}\int_0^t\big<u^x\big>(\tau,y)\,d\tau,\qquad z=x-tv.
\end{equation}
The point of this change of variables is to eliminate two of the non-decaying terms in the evolution equation in \eqref{Eur1}, namely the terms $b(y)\partial_x\omega$ and $\langle u^x\rangle \partial_x\omega$. 

Then we define the functions
\begin{equation}\label{rea1}
F(t,z,v):=\omega(t,x,y),\qquad \phi(t,z,v):=\psi(t,x,y),
\end{equation}
\begin{equation}\label{rea3}
V'(t,v):=\partial_yv(t,y),\qquad V''(t,v):=\partial_{yy}v(t,y),\qquad \dot{V}(t,v):=\partial_tv(t,y), 
\end{equation}
\begin{equation}\label{rea3'}
B'(t,v):=\partial_yb(y),\qquad B''(t,v):=\partial_{yy}b(y).
\end{equation}
Using \eqref{rea8}, we have
\begin{equation}\label{rea10}
v\in[b(0),b(1)]\,\,\text{ and }\,\,{\rm supp}\,F(t)\subset \mathbb{T}\times[b(\vartheta_0),\,b(1-\vartheta_0)]\,\,\text{ for any }\,\,t\in[0,T].
\end{equation}
The evolution equation in (\ref{Eur1}) becomes
\begin{equation}\label{main}
 \partial_tF-B''\partial_z\phi-V'\partial_vP_{\neq 0}\phi\,\partial_zF+(\dot{V}+V'\partial_z\phi)\,\partial_vF=0,
\end{equation}
where $P_{\neq 0}$ is projection off the zero mode, i.e., for any function $H(t,z,v)$
\begin{equation}\label{Pneq0}
P_{\neq 0}H(t,z,v)=H(t,z,v)-\langle H\rangle(t,v).
\end{equation}
Moreover, we have
\begin{equation}\label{rea12}
\partial_x\psi=\partial_z\phi,\qquad \partial_y\psi=V'(\partial_v\phi-t\partial_z\phi)=V'(\partial_v-t\partial_z)\phi,
\end{equation}
therefore
\begin{equation}\label{rea12.1}
\partial_{xx}\psi=\partial_{zz}\phi;\quad \partial_{yy}\psi=(V')^2(\partial_v-t\partial_z)^2\phi+V''(\partial_v-t\partial_z)\phi.
\end{equation}
Recalling the equation $\Delta\psi=\omega$, we see that $\phi$ satisfies
\begin{equation}\label{eq:Delta_t}
\partial_z^2\phi+(V')^2(\partial_v-t\partial_z)^2\phi+V''(\partial_v-t\partial_z)\phi=F,
\end{equation}
with $\phi(t,x,b(0))=\phi(t,x,b(1))=0$ for any $t\in[0,T]$ and $x\in\mathbb{T}$. 

We also need to establish equations for the functions $V',V'',\dot{V},B',B''$ associated to the change of variables. Using \eqref{changeofvariables1} and the observation $-\partial_y\big<u^x\big>=\langle\omega\rangle$, we have
\begin{equation}\label{rea15}
\begin{split}
\partial_yv(t,y)&=b'(y)-\frac{1}{t}\int_0^t\big<\omega\big>(\tau,y)\,d\tau,\\
\partial_tv(t,y)&=\frac{1}{t}\Big[-\frac{1}{t}\int_0^t\big<u^x\big>(\tau,y)\,d\tau+\big<u^x\big>(t,y)\Big],\\
\partial_y\partial_tv(t,y)&=\frac{1}{t}\Big[\frac{1}{t}\int_0^t\big<\omega\big>(\tau,y)\,d\tau-\big<\omega\big>(t,y)\Big].
\end{split}
\end{equation}
Thus
\begin{equation}\label{rea15.5}
-\frac{1}{t}\int_0^t\big<\omega\big>(\tau,y)d\tau=V'(t,v(t,y))-b'(y).
\end{equation}
By the chain rule it follows that
\begin{equation}\label{rea16}
\partial_t\big[t(V'(t,v)-B'(t,v))\big]+t\dot{V}(t,v)\partial_v\big[V'(t,v)-B'(t,v)\big]=-\big<F\big>(t,v):=-\frac{1}{2\pi}\int_{\mathbb{T}}F(t,z,v)\,dz.
\end{equation}
We notice that
\begin{equation}\label{rea17}
\partial_y(\partial_tv(t,y))=\partial_y\big[\dot{V}(t,v(t,y))\big]=V'(t,v(t,y))\partial_v\dot{V}(t,v(t,y)).
\end{equation}
Hence, using the last identity in \eqref{rea15} and the identities \eqref{rea15.5} and \eqref{rea17}, we have
\begin{equation}\label{rea18}
tV'(t,v)\partial_v\dot{V}(t,v)=B'(t,v)-V'(t,v)-\big<F\big>(t,v).
\end{equation}

We derive now our main evolution equations. It follows from \eqref{rea16} and \eqref{rea18} that
\begin{equation}\label{rea19}
\partial_t(V'-B')=V'\partial_v\dot{V}-\dot{V}\partial_v(V'-B').
\end{equation}
Set
\begin{equation}\label{rea19.5}
\mathcal{H}:=tV'\partial_v\dot{V}=B'-V'-\big<F\big>.
\end{equation}
Using \eqref{rea19} and \eqref{main} we calculate
\begin{equation*}
\begin{split}
\partial_t\mathcal{H}&=-\partial_t(V'-B')-\big<\partial_tF\big>\\
&=-V'\partial_v\dot{V}+\dot{V}\partial_v(V'-B')-V'\big<\partial_vP_{\neq 0}\phi\,\partial_zF\big>+\big<(\dot{V}+V'\partial_z\phi)\,\partial_vF\big>.
\end{split}
\end{equation*}
Using again \eqref{rea19.5} and simplifying, we get 
\begin{equation*}
\partial_t\mathcal{H}=-\frac{\mathcal{H}}{t}-\dot{V}\partial_v\mathcal{H}-V'\big<\partial_vP_{\neq0}\phi\,\partial_zF\big>+V'\big<\partial_z\phi\,\partial_vF\big>.
\end{equation*}
Finally, using \eqref{rea3'} we have
\begin{equation}\label{rea19.8}
\partial_tB'(t,v)+\dot{V}\partial_vB'(t,v)=\partial_tB''(t,v)+\dot{V}\partial_vB''(t,v)=0.
\end{equation}

We summarize our calculations so far in the following:

\begin{proposition}\label{ChangedEquations} Assume $\omega:[0,T]\times\mathbb{T}\times[0,1]\to\mathbb{R}$ is a sufficiently smooth solution of the system \eqref{Eur1}--\eqref{Eur1.1} on some time interval $[0,T]$. Assume that $\omega(t)$ is supported in $\mathbb{T}\times[\va_0,1-\va_0]$ and that $\|\langle\omega\rangle(t)\|_{H^{10}}\ll 1$  for all $t\in[0,T]$. Then 
\begin{equation}\label{rea20.5}
\big<u^x\big>(t,y)=0\qquad\text{ for any }t\in[0,T]\text{ and }y\in [0,\vartheta_0]\cup [1-\vartheta_0, 1].
\end{equation}

We define the change-of-coordinates functions $(z,v):\mathbb{T}\times[0,1]\to\mathbb{T}\times[b(0),b(1)]$, 
 \begin{equation}\label{rea20}
 v:=b(y)+\frac{1}{t}\int_0^t\big<u^x\big>(\tau,y)\,d\tau,\qquad z:=x-tv,
 \end{equation}
and the new variables $F,\phi:[0,T]\times\mathbb{T}\times[b(0),b(1)]\to\mathbb{R}$ and $V',V'',\dot{V},B',B'',\mathcal{H}:[0,T]\times[b(0),b(1)]\to\mathbb{R}$, 
\begin{equation}\label{rea21}
F(t,z,v):=\omega(t,x,y),\qquad \phi(t,z,v):=\psi(t,x,y),
\end{equation}
\begin{equation}\label{rea22}
V'(t,v):=\partial_yv(t,y),\qquad V''(t,v)=\partial_{yy}v(t,y),\qquad \dot{V}(t,v)=\partial_tv(t,y), 
\end{equation}
\begin{equation}\label{rea22'}
B'(t,v):=\partial_yb(y),\qquad B''(t,v):=\partial_{yy}b(y),
\end{equation}
\begin{equation}\label{rea23}
\mathcal{H}(t,v):=tV'(t,v)\partial_v\dot{V}(t,v)=B'(t,v)-V'(t,v)-\langle F\rangle(t,v).
\end{equation}
Then $V'(t,v)\geq \vartheta_0/2$. Moreover, the new variables $F$, $V'-B'$, $\dot{V}$, and $\mathcal{H}$ are supported in $[0,T]\times\mathbb{T}\times[b(\vartheta_0),b(1-\vartheta_0)]$ and satisfy the evolution equations
\begin{equation}\label{rea23.1}
 \partial_tF-B''\partial_z\phi=V'\partial_vP_{\neq 0}\phi\,\partial_zF-(\dot{V}+V'\partial_z\phi)\,\partial_vF,
\end{equation}
\begin{equation}\label{rea23.8}
\partial_tB'(t,v)+\dot{V}\partial_vB'(t,v)=\partial_tB''(t,v)+\dot{V}\partial_vB''(t,v)=0,
\end{equation}
\begin{equation}\label{rea24}
\partial_t(V'-B')+\dot{V}\partial_v(V'-B')=\mathcal{H}/t,
\end{equation}
\begin{equation}\label{rea25}
\partial_t\mathcal{H}+\dot{V}\partial_v\mathcal{H}=-\mathcal{H}/t-V'\big<\partial_vP_{\neq0}\phi\,\partial_zF\big>+V'\big<\partial_z\phi\,\partial_vF\big>.
\end{equation}
The variables $\phi$, $V''$, and $\dot{V}$ satisfy the elliptic-type identities
\begin{equation}\label{rea26}
\partial_z^2\phi+(V')^2(\partial_v-t\partial_z)^2\phi+V''(\partial_v-t\partial_z)\phi=F,
\end{equation}
\begin{equation}\label{rea27}
\partial_v\dot{V}=\mathcal{H}/(tV'),\qquad \dot{V}(t,b(0))=\dot{V}(t,b(1))=0,\qquad V''=V'\partial_vV'.
\end{equation}
\end{proposition}

\subsection{Energy functionals and the bootstrap proposition}\label{weightsdef} The main idea of the proof is to estimate the increment of suitable energy functionals, which are defined using special weights. For simplicity, we use exactly the same weights $A_{NR}$, $A_R$, $A_k$ as in our earlier papers \cite{IOJI,IOJI2}, so we can use some of their properties proved there. These weights are defined by 
\begin{equation}\label{reb11}
A_{NR}(t,\xi):=\frac{e^{\lambda(t)\langle\xi\rangle^{1/2}}}{b_{NR}(t,\xi)}e^{\sqrt{\delta}\langle\xi\rangle^{1/2}},\qquad A_R(t,\xi):=\frac{e^{\lambda(t)\langle\xi\rangle^{1/2}}}{b_R(t,\xi)}e^{\sqrt{\delta}\langle\xi\rangle^{1/2}},
\end{equation}
and
\begin{equation}\label{reb12}
A_k(t,\xi):=e^{\lambda(t)\langle k,\xi\rangle^{1/2}}\Big(\frac{e^{\sqrt{\delta}\langle\xi\rangle^{1/2}}}{b_k(t,\xi)}+e^{\sqrt{\delta}|k|^{1/2}}\Big),
\end{equation}
where $k\in\mathbb{Z}$, $t\in[0,\infty)$, $\xi\in\mathbb{R}$. The function $\lambda:[0,\infty)\to[\delta_0,3\delta_0/2]$ is defined by
\begin{equation}\label{reb10.5}
\lambda(0)=\frac{3}{2}\delta_0,\,\,\,\,\lambda'(t)=-\frac{\delta_0\sigma_0^2}{\langle t\rangle^{1+\sigma_0}},
\end{equation}
where $\delta_0>0$ is a fixed parameter and $\sigma_0=0.01$. In particular, $\lambda$ is decreasing on $[0,\infty)$, and the functions $A_{NR},\,A_R,\,A_k$ are also decreasing in $t$. The parameter $\delta>0$, which appears also in the weights $b_R,\,b_{NR},\,b_k$, is to be taken sufficiently small, depending only on the structural parameters $\delta_0$, $\vartheta_0$, and $\kappa$.

The precise definitions of the weights $b_{NR},\,b_R,\,b_k$ are very important; the details are provided in section \ref{weights}. For now we note that these functions are essentially increasing in $t$ and satisfy 
\begin{equation}\label{reb13}
e^{-\delta\sqrt{|\xi|}}\leq b_R(t,\xi)\leq b_k(t,\xi)\leq b_{NR}(t,\xi)\leq 1,\qquad\text{ for any }t,\xi,k.
\end{equation}
In other words, the weights $1/b_{NR},\,1/b_R,\,1/b_k$ are small when compared to the main factors $e^{\lambda(t)\langle\xi\rangle^{1/2}}$ and $e^{\lambda(t)\langle k,\xi\rangle^{1/2}}$ in \eqref{reb11}--\eqref{reb12}. However, their relative contributions are important as they are used to distinguish between resonant and non-resonant times. 

Assume that $\omega:[0,T]\times\mathbb{T}\times[0,1]\to\mathbb{R}$ is as in Proposition \ref{ChangedEquations} and define the functions $F, \phi, V', V'', \dot{V}, B', B'', \mathcal{H}$ as in \eqref{rea21}--\eqref{rea23}. To construct useful energy functionals we need to modify the functions $V', B', B''$ which are not ``small", so we define the new variables
\begin{equation}\label{VB5}
\begin{split}
&B'_0(v):=B'(0,v)=(\partial_yb)(b^{-1}(v)),\qquad B''_0(v):=B''(0,v)=(\partial_y^2b)(b^{-1}(v)),\\
&V'_{\ast}:=V'-B'_0,\qquad B'_{\ast}:=B'-B'_0,\qquad B''_{\ast}:=B''-B''_0.
\end{split}
\end{equation}

Our main goal is to control the functions $F$ and $\phi$. For this we need to consider two auxiliary functions $F^\ast$ and $\phi'$. We define first the function $\phi'(t,z,v):[0,T]\times\mathbb{T}\times[b(0),b(1)]\to\mathbb{R}$ as the unique solution to the equation (see Lemma \ref{lm:elli3} for existence and uniqueness)
\begin{equation}\label{Ph1}
\partial_z^2\phi'+(B'_0)^2(\partial_v-t\partial_z)^2\phi'+B''_0(\partial_v-t\partial_z)\phi'=F,\qquad \phi'(t,b(0))=\phi'(t,b(1))=0,
\end{equation}
on $\mathbb{T}\times[b(0),b(1)]$. Then we define the modified profile
\begin{equation}\label{reb14}
F^{\ast}(t,z,v):=F(t,z,v)-B_0''(v)\int_0^t\partial_z\phi'(\tau,z,v)\,d\tau,
\end{equation}
and the renormalized stream functions
\begin{equation}\label{defgellip} 
\begin{split}
\Theta(t,z,v)&:=(\partial_z^2+(\partial_v-t\partial_z)^2)\left(\Psi(v)\,\phi(t,z,v)\right),\\
\Theta^{\ast}(t,z,v)&:=(\partial_z^2+(\partial_v-t\partial_z)^2)\left(\Psi(v)\,(\phi(t,z,v)-\phi'(t,z,v))\right),
\end{split}
\end{equation}
 where $\Psi:\mathbb{R}\to[0,1]$ is a Gevrey class cut-off function, satisfying
\begin{equation}\label{rec0}
\begin{split}
&\big\|e^{\langle\xi\rangle^{3/4}}\widetilde{\Psi}(\xi)\big\|_{L^\infty}\lesssim 1,\\
&{\rm supp}\,\Psi\subseteq \big[b(\va_0/4),b(1-\va_0/4)\big], \quad\Psi\equiv 1\text{ in }\big[b(\va_0/3),b(1-\va_0/3)\big].
 \end{split}
\end{equation}

Our bootstrap argument is based on controlling simultaneously energy functionals and space-time integrals. Let $\dot{A}_Y(t,\xi):=(\partial_t A_Y)(t,\xi)\leq 0$, $Y\in\{NR,R,k\}$, and define, for any $t\in[0,T]$,
\begin{equation}\label{rec1}
\begin{split}
\mathcal{E}_f(t)&:=\sum_{k\in \mathbb{Z}}\int_{\R}A_k^2(t,\xi)\big|\widetilde{f}(t,k,\xi)\big|^2\,d\xi,\qquad f\in\{F,F^\ast\},\\
\mathcal{B}_f(t)&:=\int_1^t\sum_{k\in \mathbb{Z}}\int_{\R}|\dot{A}_k(s,\xi)|A_k(s,\xi)\big|\widetilde{f}(s,k,\xi)\big|^2\,d\xi ds,
\end{split}
\end{equation}
\begin{equation}\label{rec1''}
\begin{split}
\mathcal{E}_{F-F^{\ast}}(t)&:=\sum_{k\in \mathbb{Z}^\ast}\int_{\R}(1+\langle k,\xi\rangle/\langle t\rangle) A_k^2(t,\xi)\big|\widetilde{(F-F^{\ast})}(t,k,\xi)\big|^2\,d\xi,\\
\mathcal{B}_{F-F^{\ast}}(t)&:=\int_1^t\sum_{k\in \mathbb{Z}^\ast}\int_{\R}(1+\langle k,\xi\rangle/\langle s\rangle)|\dot{A}_k(s,\xi)|A_k(s,\xi)\big|\widetilde{(F-F^{\ast})}(s,k,\xi)\big|^2\,d\xi ds,
\end{split}
\end{equation}
\begin{equation}\label{rec3.5}
\begin{split}
\mathcal{E}_{\Phi}(t)&:=\sum_{k\in \mathbb{Z}^\ast}\int_{\R}A_k^2(t,\xi)\frac{|k|^2\langle t\rangle^2}{|\xi|^2+|k|^2\langle t\rangle^2}\big|\widetilde{\Phi}(t,k,\xi)\big|^2\,d\xi,\qquad\Phi\in\{\Theta,\Theta^\ast\},\\
\mathcal{B}_{\Phi}(t)&:=\int_1^t\sum_{k\in \mathbb{Z}^\ast}\int_{\R}|\dot{A}_k(s,\xi)|A_k(s,\xi)\frac{|k|^2\langle s\rangle^2}{|\xi|^2+|k|^2\langle s\rangle^2}\big|\widetilde{\Phi}(s,k,\xi)\big|^2\,d\xi ds,
\end{split}
\end{equation}
\begin{equation}\label{rec2}
\begin{split}
\mathcal{E}_{g}(t)&:=\int_{\R}A_R^2(t,\xi)\big|\widetilde{g}(t,\xi)\big|^2\,d\xi,\qquad g\in\{V'_{\ast}, B'_{\ast}, B''_{\ast}\},\\
\mathcal{B}_{g}(t)&:=\int_1^t\int_{\R}|\dot{A}_R(s,\xi)|A_R(s,\xi)\big|\widetilde{g}(s,\xi)\big|^2\,d\xi ds,
\end{split}
\end{equation}
\begin{equation}\label{rec3}
\begin{split}
\mathcal{E}_{\mathcal{H}}(t)&:=\mathcal{K}^{2}\int_{\R}A_{NR}^2(t,\xi)\big(\langle t\rangle/\langle\xi\rangle\big)^{3/2}\big|\widetilde{\mathcal{H}}(t,\xi)\big|^2\,d\xi,\\
\mathcal{B}_{\mathcal{H}}(t)&:=\mathcal{K}^{2}\int_1^t\int_{\R}|\dot{A}_{NR}(s,\xi)|A_{NR}(s,\xi)\big(\langle s\rangle/\langle\xi\rangle\big)^{3/2}\big|\widetilde{\mathcal{H}}(s,\xi)\big|^2\,d\xi ds,
\end{split}
\end{equation}
where $\mathbb{Z}^\ast:=\mathbb{Z}\setminus \{0\}$ and $\mathcal{K}\geq 1$ is a large constant that depends only on $\delta$.

Our main bootstrap proposition is the following: 

\begin{proposition}\label{MainBootstrap}
Assume $T\geq 1$ and $\omega\in C([0,T]:\G^{2\delta_0,1/2})$ is a sufficiently smooth solution of the system \eqref{Eur1}--\eqref{Eur1.1}, with the property that $\omega(t)$ is supported in $\mathbb{T}\times[\va_0,1-\va_0]$ and that $\|\langle\omega\rangle(t)\|_{H^{10}}\ll 1$  for all $t\in[0,T]$. Define $F,F^{\ast},\Theta, \Theta^{\ast}\,B'_{\ast},B''_{\ast},V'_{\ast},\mathcal{H}$ as above. Assume that $\eps_1$ is sufficiently small (depending on $\delta$),
\begin{equation}\label{boot1}
\sum_{g\in\{F,\,F^{\ast},\,F-F^{\ast},\,\Theta,\,\Theta^{\ast},\,V'_{\ast},\,B'_{\ast},\,B''_{\ast},\,\mathcal{H}\}}\mathcal{E}_g(t)\leq\eps_1^3\qquad\text{ for any }t\in[0,1],
\end{equation}
and
\begin{equation}\label{boot2}
\sum_{g\in\{F,\,F^{\ast},\,F-F^{\ast},\,\Theta,\,\Theta^{\ast},\,V'_{\ast},\,B'_{\ast},\,B''_{\ast},\,\mathcal{H}\}}\big[\mathcal{E}_g(t)+\mathcal{B}_g(t)\big]\leq\eps_1^2\qquad\text{ for any }t\in[1,T].
\end{equation}
Then for any $t\in[1,T]$ we have the improved bounds
\begin{equation}\label{boot3}
\sum_{g\in\{F,\,F^{\ast},\,F-F^{\ast},\,\Theta,\,\Theta^{\ast},\,V'_{\ast},\,B'_{\ast},\,B''_{\ast},\,\mathcal{H}\}}\big[\mathcal{E}_g(t)+\mathcal{B}_g(t)\big]\leq\eps_1^2/2.
\end{equation}
Moreover, we also have the stronger bounds for $t\in[1,T]$
\begin{equation}\label{boot3'}
\sum_{g\in\{F,\,\Theta\}}\big[\mathcal{E}_g(t)+\mathcal{B}_g(t)\big]\lesssim_{\delta}\eps_1^3.
\end{equation}
\end{proposition}

The proof of Proposition \ref{MainBootstrap} is the main part of this paper, and covers sections \ref{SecLem}--\ref{PropD}.  In section \ref{ProofMainThm} we then show how to use this proposition to prove the main theorem.

\subsection{The variables of the bootstrap argument}\label{RemBoots}
Our argument outlined in Proposition \ref{MainBootstrap} involves control of nine quantities. We explain now the roles of these quantities:

(1) The main variables are the  vorticity profile $F$ and the differentiated stream function $\Theta$. Our primary goal is to prove global bounds on these quantities.

(2) The functions $F^\ast$ and $\Theta^\ast$ are auxiliary variables, and we analyze them as an intermediate step to controlling the main variables $F$ and $\Theta$. The function $F^\ast$ satisfies a better transport equation than $F$, without any other linear terms, while the function $\Theta^\ast$ satisfies a better elliptic equation than $\Theta$, again without linear terms in the right-hand side.

(3) A significant component of the proof is to control the function $F-F^\ast$, which allows us to pass from the modified profile $F^\ast$ to the true profile $F$. This is based on the theory of the linearized equation in Gevrey spaces, as developed by the second author in \cite{JiaG}, and requires the spectral assumption $(B)$ on the shear flow. We remark that the bootstrap control on the variable $F-F^\ast$ is slightly stronger than on the variables $F$ and $F^\ast$ separately, which is needed to compensate for the lack of symmetry in some of the transport terms.

(4) The functions $V'_\ast$, $B'_\ast$, and $B''_\ast$ are connected to the change of variables $y\to v$.  These functions appear in many of the nonlinear terms in the equations, so it is important to control their smoothness precisely, as part of a combined bootstrap argument, in a way that is consistent with the smoothness of the functions $F$ and $\Theta$.

(5) Finally, the function $\mathcal{H}$, which decays in time, encodes the convergence of the system as $t\to\infty$. This function decays at a rate of $\langle t\rangle^{-3/4}$, in a weaker topology, which shows that the function $\partial_v\dot{V}$ decays fast at an integrable rate of $\langle t\rangle^{-7/4}$, again in a weaker topology.

\section{Gevrey spaces and the weights $A_k$, $A_R$, and $A_{NR}$}\label{SecLem}

\subsection{Gevrey spaces}\label{appendix} We summarize here some general properties of the Gevrey spaces of functions. To perform certain algebraic operations, it is very useful to have a related definition in the physical space. For any domain $D\subseteq\T\times\R$ (or $D\subseteq\R$) and parameters $s\in(0,1)$ and $M\geq 1$ we define the spaces
\begin{equation}\label{Gevr2}
\widetilde{\mathcal{G}}^{s}_M(D):=\big\{f:D\to\mathbb{C}:\,\|f\|_{\widetilde{\mathcal{G}}^{s}_M(D)}:=\sup_{x\in D,\,m\geq 0,\,|\alpha|\leq m}|D^\alpha f(x)|M^{-m}(m+1)^{-m/s}<\infty\big\}.
\end{equation}
We start with a lemma connecting the spaces $\mathcal{G}^{\mu,s}$ and $\widetilde{\mathcal{G}}^s_M$. 

\begin{lemma}\label{lm:Gevrey}
(i) Suppose that $s\in(0,1)$, $K>1$, and $f\in C^{\infty}(\mathbb{T}\times \mathbb{R})$ with ${\rm supp}\,f\subseteq \mathbb{T}\times[-L,L]$ satisfies the bounds $\|f\|_{\widetilde{\mathcal{G}}^{s}_K(\mathbb{T}\times \mathbb{R})}\leq 1$.
Then there is $\mu=\mu(K,s)>0$ such that 
\begin{equation}\label{gevreyP}
\big|\widetilde{f}(k,\xi)\big|\lesssim_{K,s} Le^{-\mu|k,\xi|^s}\qquad \text{ for all } k\in\mathbb{Z},\,\xi\in \R.
\end{equation}

(ii) Conversely, if $\mu>0$ and $s\in(0,1)$, then there is $K=K(s,\mu)>1$ such that
\begin{equation}\label{eq:four}
\big\|f\big\|_{\widetilde{\mathcal{G}}^{s}_K(\mathbb{T}\times \mathbb{R})}\lesssim_{\mu,s} \big\|f\big\|_{\mathcal{G}^{\mu,s}(\mathbb{T}\times \mathbb{R})}.
\end{equation}
\end{lemma}

Using this lemma one can construct cutoff functions in Gevrey spaces: for any points $a'<a\leq b<b'\in\mathbb{R}$ and any $s\in (0,1)$ there are functions $\Psi$ supported in $[a',b']$, equal to $1$ in $[a,b]$, and satisfying $\big|\widetilde{\Psi}(\xi)\big|\lesssim e^{-\langle\xi\rangle^s}$ for any $\xi\in\mathbb{R}$. See \cite[Subsection A.1]{IOJI} for an explicit construction of such functions, as well as an elementary proof of Lemma \ref{lm:Gevrey}. We use several functions of this type in the proof of our main theorem.

The physical space characterization of Gevrey functions is useful when studying compositions and algebraic operations of functions.

\begin{lemma}\label{GPF} (i) Assume  $s\in(0,1)$, $M\geq 1$, and $f_1,f_2\in \widetilde{\mathcal{G}}^{s}_M(D)$. Then $f_1f_2\in\widetilde{\mathcal{G}}^{s}_{M'}(D)$ and
\begin{equation*}
\|f_1f_2\|_{\widetilde{\mathcal{G}}^{s}_{M'}(D)}\lesssim \|f_1\|_{\widetilde{\mathcal{G}}^{s}_{M}(D)}\|f_2\|_{\widetilde{\mathcal{G}}^{s}_{M}(D)}
\end{equation*}
for some $M'=M'(s,M)\geq M$. Similarly, if $f_1\geq 1$ in $D$ then $\|(1/f_1)\|_{\widetilde{\mathcal{G}}^{s}_{M'}(D)}\lesssim 1$.

(ii) Suppose $s\in(0,1)$, $M\geq 1$, $I_1\subseteq \R$ is an interval, and $g:\mathbb{T}\times I_1\to \mathbb{T}\times I_2$ satisfies
\begin{equation}\label{gbo1}
|D^\alpha g(x)|\leq M^m(m+1)^{m/s}\qquad \text{ for any }x\in\T\times I_1,\,m\geq 1,\text{ and }|\alpha|\in [1,m].
\end{equation}
If $K\geq 1$ and $f\in \widetilde{G}^s_K(\T\times I_2)$ then $f\circ g\in \widetilde{G}^s_L(\T\times I_1)$ for some $L=L(s,K,M)\geq 1$ and
\begin{equation}\label{Ffgc}
\left\|f\circ g\right\|_{\widetilde{G}^s_L(\T\times I_1)}\lesssim_{s,K,M} \left\|f\right\|_{\widetilde{G}^s_K(\T\times I_2)}.
\end{equation}

(iii) Assume $s\in(0,1)$, $L\in[1,\infty)$, $I,J\subseteq\mathbb{R}$ are open intervals, and $g:I\to J$ is a smooth bijective map satisfying, for any $m\geq 1$,
\begin{equation}\label{gbo2}
|D^\alpha g(x)|\leq L^m(m+1)^{m/s}\qquad \text{ for any }x\in I\text{ and }|\alpha|\in [1,m].
\end{equation}
If $|g'(x)|\geq \rho>0$ for any $x\in I$ then the inverse function $g^{-1}:J\to I$ satisfies the bounds
\begin{equation}\label{gbo2.1}
|D^\alpha (g^{-1})(x)|\leq M^m(m+1)^{m/s}\qquad \text{ for any }x\in J\text{ and }|\alpha|\in [1,m],
\end{equation}
for some constant $M=M(s,L,\rho)\geq L$.
\end{lemma}

Lemma \ref{GPF} can be proved by elementary means using just the definition \eqref{Gevr2}. See also \cite[Theorems 6.1 and 3.2]{Yamanaka} for more general estimates on functions in Gevrey spaces.

\subsection{The weights $A_{NR},\,A_R,\,A_k$}\label{weights} We summarize here the construction of our main imbalanced weights $A_R,\,A_{NR},\, A_k$ in \cite{IOJI}. We start by defining the functions $w_{NR},w_R:[0,\infty)\times\mathbb{R}\to [0,1]$, which model the non-resonant and resonant growth. Assume that $\delta>0$ is small, $\delta\ll \delta_0$. For $|\eta|\leq\delta^{-10}$ we define simply
\begin{equation}\label{reb1}
w_{NR}(t,\eta):=1,\qquad w_R(t,\eta):=1.
\end{equation}
For $\eta>\delta^{-10}$ we define $k_0(\eta):=\lfloor\sqrt{\delta^3\eta}\rfloor$. For $l\in\{1,\ldots,k_0(\eta)\}$ we define
\begin{equation}\label{reb2}
t_{l,\eta}:=\frac{1}{2}\big(\frac{\eta}{l+1}+\frac{\eta}{l}\big),\qquad t_{0,\eta}:=2\eta,\qquad I_{l,\eta}:=[t_{l,\eta},\,t_{l-1,\eta}].
\end{equation}
Notice that $|I_{l,\eta}|\approx \eta/l^2$ and
\begin{equation*}
\delta^{-3/2}\sqrt{\eta}/2\leq t_{k_0(\eta),\eta}\leq\ldots\leq t_{l,\eta}\leq\eta/l\leq t_{l-1,\eta}\leq\ldots\leq t_{0,\eta}=2\eta.
\end{equation*}

We define
\begin{equation}\label{reb3}
w_{NR}(t,\eta):=1,\,w_{R}(t,\eta):=1\qquad\text{ if }\,\,t\geq t_{0,\eta}=2\eta.
\end{equation}
Then we define, for $k\in\{1,\ldots,k_0(\eta)\}$,
\begin{equation}\label{reb5}
\begin{split}
w_{NR}(t,\eta)&:=\Big(\frac{1+\delta^2|t-\eta/k|}{1+\delta^2|t_{k-1,\eta}-\eta/k|}\Big)^{\delta_0}w_{NR}(t_{k-1,\eta},\eta)\qquad\text{ if }t\in[\eta/k,t_{k-1,\eta}],\\
w_{NR}(t,\eta)&:=\Big(\frac{1}{1+\delta^2|t-\eta/k|}\Big)^{1+\delta_0}w_{NR}(\eta/k,\eta)\qquad\text{ if }t\in[t_{k,\eta},\eta/k].
\end{split}
\end{equation}
We define also the weight $w_R$ by the formula
\begin{equation}\label{reb5.5}
w_R(t,\eta):=
\begin{cases}
w_{NR}(t,\eta)\frac{1+\delta^2|t-\eta/k|}{1+\delta^2\eta/(8k^2)}\qquad&\text{ if }|t-\eta/k|\leq\eta/(8k^2)\\
w_{NR}(t,\eta)\qquad&\text{ if }t\in I_{k,\eta},\,|t-\eta/k|\geq\eta/(8k^2),
\end{cases}
\end{equation}
for any $k\in\{1,\ldots,k_0(\eta)\}$ and notice that for $t\in I_{k,\eta}$,
\begin{equation}\label{reb8}
\frac{\partial_tw_{NR}(t,\eta)}{w_{NR}(t,\eta)}\approx\frac{\partial_tw_R(t,\eta)}{w_R(t,\eta)}\approx \frac{\delta^2}{1+\delta^2\left|t-\eta/k\right|}.
\end{equation}

It is easy to see that, for $\eta>\delta^{-10}$,
\begin{equation}\label{reb8.6}
\begin{split}
&w_{NR}(t_{k_0(\eta),\eta},\eta)=w_R(t_{k_0(\eta),\eta},\eta)\in[X_\delta(\eta)^4,X_\delta(\eta)^{1/4}],\qquad X_\delta(\eta):=e^{-\delta^{3/2}\ln(\delta^{-1})\sqrt\eta}.
\end{split}
\end{equation}
For small values of $t=(1-\beta)t_{k_0(\eta),\eta}$, $\beta\in[0,1]$, we define $w_{NR}$ and $w_R$ by the formulas 
\begin{equation}\label{reb9}
w_{NR}(t,\eta)=w_R(t,\eta):=(e^{-\delta\sqrt\eta})^\beta w_{NR}(t_{k_0(\eta),\eta},\eta)^{1-\beta}.
\end{equation}

If $\eta<-\delta^{-10}$, then we define $w_R(t,\eta):=w_R(t,|\eta|)$, $w_{NR}(t,\eta):=w_{NR}(t,|\eta|)$ and $I_{k,\eta}:=I_{-k,-\eta}$. To summarize, the resonant intervals $I_{k,\eta}$ are defined for $(k,\eta)\in\mathbb{Z}\times\mathbb{R}$ satisfying $|\eta|>\delta^{-10}$, $1\leq |k|\leq  \sqrt{\delta^3|\eta|}$, and $\eta/k>0$. 

Finally, we define the weights $w_k(t,\eta)$ by the formula
\begin{equation}\label{eq:resonantweight}
w_k(t,\eta):=\left\{\begin{array}{lll}
w_{NR}(t,\eta)&{\rm \,if\,}&t\not\in I_{k,\eta},\\
w_R(t,\eta)&{\rm \,if\,}&t\in I_{k,\eta}.
\end{array}\right.
\end{equation}
If particular $w_k(t,\eta)=w_{NR}(t,\eta)$ unless $|\eta|>\delta^{-10}$, $1\leq |k|\leq  \sqrt{\delta^3|\eta|}$, $\eta/k>0$, and $t\in I_{k,\eta}$. 

The functions $w_{NR}$, $w_{R}$ and $w_k$ have the right size but lack optimal smoothness in the frequency parameter $\eta$, mainly due to the jump discontinuities of the function $k_0(\eta)$. This smoothness is important in symmetrization arguments (energy control of the transport terms) and in commutator arguments. To correct this problem we fix $\varphi: \R \to [0,1]$ an even smooth function supported in $[-8/5,8/5]$ and equal to $1$ in $[-5/4,5/4]$ and let $d_0:=\int_\mathbb{R}\varphi(x)\,dx$. For $k\in\mathbb{Z}$ and $Y\in \{NR,R,k\}$ let
\begin{equation}\label{dor1}
\begin{split}
b_Y(t,\xi)&:=\int_\R w_Y(t,\rho)\varphi\Big(\frac{\xi-\rho}{L_{\delta'}(t,\xi)}\Big)\frac{1}{d_0L_{\delta'}(t,\xi)}\,d\rho,\\
L_{\delta'}(t,\xi)&:=1+\frac{\delta'\langle\xi\rangle}{\langle\xi\rangle^{1/2}+\delta' t},\qquad\delta'\in[0,1].
\end{split}
\end{equation}
The length $L_{\delta'}(t,\xi)$ in \eqref{dor1} is chosen to optimize the smoothness in $\xi$ of the functions $b_Y(t,.)$, while not changing significantly the size of the weights. The parameter $\delta'$ is fixed sufficiently small, depending only on $\delta$. 

We can now finally define our main weights $A_{NR}$, $A_R$, and $A_k$. We define first the decreasing function $\lambda:[0,\infty)\to[\delta_0,3\delta_0/2]$ by
\begin{equation}\label{dor2}
\lambda(0)=\frac{3}{2}\delta_0,\,\,\,\,\lambda'(t)=-\frac{\delta_0\sigma_0^2}{\langle t\rangle^{1+\sigma_0}},
\end{equation}
for small positive constant $\sigma_0$ (say $\sigma_0=0.01$). Then we define
\begin{equation}\label{dor3}
A_R(t,\xi):=\frac{e^{\lambda(t)\langle\xi\rangle^{1/2}}}{b_R(t,\xi)}e^{\sqrt{\delta}\langle\xi\rangle^{1/2}},\qquad A_{NR}(t,\xi):=\frac{e^{\lambda(t)\langle\xi\rangle^{1/2}}}{b_{NR}(t,\xi)}e^{\sqrt{\delta}\langle\xi\rangle^{1/2}},
\end{equation}
and, for any $k\in\mathbb{Z}$,
\begin{equation}\label{dor4}
A_k(t,\xi):=e^{\lambda(t)\langle k,\xi\rangle^{1/2}}\Big(\frac{e^{\sqrt{\delta}\langle\xi\rangle^{1/2}}}{b_k(t,\xi)}+e^{\sqrt{\delta}|k|^{1/2}}\Big).
\end{equation}
We record the simple inequalities
\begin{equation}\label{dor4.1}
\begin{split}
&e^{\lambda(t)\langle\xi\rangle^{1/2}}\leq A_{NR}(t,\xi)\leq A_R(t,\xi)\leq e^{\lambda(t)\langle\xi\rangle^{1/2}}e^{2\sqrt\delta\langle\xi\rangle^{1/2}},\\
&e^{\lambda(t)\langle k,\xi\rangle^{1/2}}\leq A_k(t,\xi)\leq 2e^{\lambda(t)\langle k,\xi\rangle^{1/2}}e^{2\sqrt\delta\langle k,\xi\rangle^{1/2}},
\end{split}
\end{equation}
for any $k\in\mathbb{Z}$, $t\geq 0$, and $\xi\in\mathbb{R}$.

\subsubsection{Properties of the weights} We collect now several bounds on these weights, which are proved either in \cite{IOJI} or in \cite{IOJI2}. In these papers we prove many more properties of the weights, but we summarize here only the ones that we need explicitly in this paper.

We start with some bounds on $w_Y$ and $b_Y$, see \cite[Lemmas 7.1 and 7.2]{IOJI} for the proof.

\begin{lemma}\label{comparisonweights}
For $t\geq 0$, $\xi,\eta\in\R$, $k\in\Z$, and $Y\in\{NR,R,k\}$ we have
\begin{equation}\label{eq:comparisonweights1}
\frac{w_{Y}(t,\xi)}{w_{Y}(t,\eta)}\lesssim_\delta e^{\sqrt{\delta} |\eta-\xi|^{1/2}},
\end{equation}
\begin{equation}\label{dor20}
b_Y(t,\xi)\approx_\delta w_Y(t,\xi),\qquad |\partial_\xi b_Y(t,\xi)|\lesssim_\delta b_Y(t,\xi)\frac{1}{L_{\delta'}(t,\xi)}.
\end{equation}
\end{lemma}

We recall now several bounds on the main weights $A_{NR}, A_R, A_k$, see \cite[Lemma 7.3]{IOJI}.

\begin{lemma}\label{A_kA_ell}
(i) Assume $t\in[0,\infty)$, $k\in\mathbb{Z}$, and $Y\in\{NR,R,k\}$. If $\xi,\eta\in\mathbb{R}$ satisfy $|\eta|\geq |\xi|/4$ (or $|(k,\eta)|\geq|(k,\xi)|/4$ if $Y=k$) then
\begin{equation}\label{vfc25}
\frac{A_Y(t,\xi)}{A_Y(t,\eta)}\lesssim_\delta e^{0.9\lambda(t)|\xi-\eta|^{1/2}}.
\end{equation}

(ii) Assume $t\in[0,\infty)$, $k,\ell\in\mathbb{Z}$ and $\xi,\eta\in\mathbb{R}$ satisfy $|(\ell,\eta)|\geq |(k,\xi)|/4$. If $t\not\in I_{k,\xi}$ or if $t\in I_{k,\xi}\cap I_{\ell,\eta}$, then
\begin{equation}\label{vfc26}
\frac{A_k(t,\xi)}{A_\ell(t,\eta)}\lesssim_\delta e^{0.9\lambda(t)|(k-\ell,\xi-\eta)|^{1/2}}.
\end{equation}
If $t\in I_{k,\xi}$ and $t\not\in I_{\ell,\eta}$, then
\begin{equation}\label{vfc27}
\frac{A_k(t,\xi)}{A_\ell(t,\eta)}\lesssim_\delta \frac{|\xi|}{k^2}\frac{1}{1+\big|t-\xi/k\big|} e^{0.9\lambda(t)|(k-\ell,\xi-\eta)|^{1/2}}.
\end{equation}
\end{lemma}

In some commutator estimates we need an additional property of the weights $A_k$, which is proved in \cite[Lemma 7.5]{IOJI2}.

\begin{lemma}\label{A-A}
There is a constant constant $C_0(\delta)\gg 1$ such that if $\xi,\eta\in\mathbb{R}$, $t\geq 0$, $k\in\mathbb{Z}$, and $\langle\xi-\eta\rangle\leq (\langle k,\xi\rangle+\langle k,\eta\rangle)/8$ then
 \begin{equation}\label{A-A2}
\big|A_k(t,\xi)-A_k(t,\eta)\big|\lesssim  A_R(t,\xi-\eta)A_k(t,\eta) e^{-(\lambda(t)/40)\langle \xi-\eta\rangle^{1/2}}\Big[\frac{C_0(\delta)}{\langle k,\xi\rangle^{1/8}}+\sqrt{\delta}\Big].
\end{equation}

\end{lemma}

To control the space-time integrals defined in \eqref{rec1}--\eqref{rec3} we also need estimates on the time derivatives of the weights $A_Y$. 

\begin{lemma}\label{lm:CDW}
(i) For all $t\ge 0,$ $\rho\in\mathbb{R}$, and $Y\in\{NR,R\}$ we have
\begin{equation}\label{TLX3.5}
\frac{-\dot{A}_Y(t,\rho)}{A_Y(t,\rho)}\approx_\delta\left[\frac{\langle\rho\rangle^{1/2}}{\langle t\rangle^{1+\sigma_0}}+\frac{\partial_tw_Y(t,\rho)}{w_Y(t,\rho)}\right].
\end{equation}
and, for any $k\in\Z$,
\begin{equation}\label{eq:A_kxi}
\frac{-\dot{A}_k(t,\rho)}{A_k(t,\rho)}\approx_\delta\left[\frac{\langle k,\rho\rangle^{1/2}}{\langle t\rangle^{1+\sigma_0}}+\frac{\partial_tw_k(t,\rho)}{w_k(t,\rho)}\frac{1}{1+e^{\sqrt\delta(|k|^{1/2}-\langle\rho\rangle^{1/2})}w_k(t,\rho)}\right].
\end{equation}
In particular, if $k\in\mathbb{Z}^\ast$, $t\geq 0$ and $\rho\in\mathbb{R}$ then
\begin{equation}\label{newAdot}
|(\dot{A}_k/A_k)(t,\rho)|\gtrsim_\delta \langle t-\rho/k\rangle^{-1-\sigma_0}.
\end{equation}

(ii) For all $t\ge 0,$ $\xi,\,\eta\in\mathbb{R}$, and $Y\in\{NR,R\}$ we have
\begin{equation}\label{vfc30}
\big|(\dot{A}_Y/A_Y)(t,\xi)\big|\lesssim_\delta \big|(\dot{A}_Y/A_Y)(t,\eta)\big|e^{4\sqrt{\delta}|\xi-\eta|^{1/2}}.
\end{equation}
Moreover, if $k,\ell\in\mathbb{Z}$ then
\begin{equation}\label{eq:CDW}
\big|(\dot{A}_k/A_k)(t,\xi)\big|\lesssim_\delta \big|(\dot{A}_\ell/A_{\ell})(t,\eta)\big|e^{4\sqrt{\delta}|k-\ell,\xi-\eta|^{1/2}}.
\end{equation}
Finally, if $\rho\in\mathbb{R}$ and $k\in\mathbb{Z}$ satisfy $|k|\leq\langle\rho\rangle+10$ then
\begin{equation}\label{vfc30.5}
\big|(\dot{A}_k/A_k)(t,\rho)\big|\approx_\delta\big|(\dot{A}_{NR}/A_{NR})(t,\rho)\big|\approx_\delta\big|(\dot{A}_R/A_R)(t,\rho)\big|.
\end{equation}
\end{lemma}

\begin{proof}All the estimates except for \eqref{newAdot} are proved in Lemma 7.4 in \cite{IOJI}. To prove \eqref{newAdot}, we use first \eqref{eq:A_kxi}, thus $\big|(\dot{A}_k/A_k)(t,\rho)\big|\gtrsim_\delta\langle k,\rho\rangle^{1/2}\langle t\rangle^{-1-\sigma_0}$, and the bounds \eqref{newAdot} follow unless $|\rho/k|\geq \delta^{-12}$ and $|t-\rho/k|\leq |\rho|/(10|k|)$. In this case we use the second term in \eqref{eq:A_kxi}, thus
\begin{equation*}
|(\dot{A}_k/A_k)(t,\rho)|\gtrsim_\delta\frac{\partial_tw_k(t,\rho)}{w_k(t,\rho)}\gtrsim_\delta\frac{\partial_tw_{NR}(t,\rho)}{w_{NR}(t,\rho)}.
\end{equation*}
In view of \eqref{reb8} and \eqref{reb9}, this suffices to prove \eqref{newAdot} in the remaining range.
\end{proof}

To control commutators in the space-time integrals we need to also regularize the weights $|\dot{A}_Y/A_Y|$. We start by defining
\begin{equation}\label{muD}
\mu^{\#}(t,\xi):=
\begin{cases}
                            0&{\rm if}\,\,|\xi|\leq \delta^{-10}\text{ or if }\,\,|\xi|>\delta^{-10}\,{\rm and }\,\,t>2|\xi|,\\
                            \delta^2& {\rm if}\,\,|\xi|>\delta^{-10} \,\,{\rm and}\,\,t<t_{k_0(\xi),\xi},\\
                            \frac{\delta^2}{1+\delta^2|t-\xi/k|}&{\rm if}\,\,|\xi|>\delta^{-10}\,\,{\rm and}\,\,t\in I_{k,\xi}, \,\,k\in\{1,2,\dots,k_0(\xi)\},
\end{cases}
\end{equation}
for $t\ge 0$ and $\xi\geq 0$. Compare with the formulas \eqref{reb8}. Then we define $\mu^\#(t,\xi):=\mu^\#(t,|\xi|)$ if $\xi\leq 0$ and regularize the weight, as in \eqref{dor1}, 
 \begin{equation}\label{muaD}
 \mu^\ast(t,\xi):=\int_{\mathbb{R}}\mu^{\#}(t,\rho)\frac{1}{d_0L_{\delta'}(t,\xi)}\varphi\bigg(\frac{\xi-\rho}{L_{\delta'}(t,\xi)}\bigg)\,d\rho,\qquad L_{\delta'}(t,\xi):=1+\frac{\delta'\langle\xi\rangle}{\langle\xi\rangle^{1/2}+\delta' t}.
 \end{equation}
Finally, we define, motivated by the formulas \eqref{TLX3.5}--\eqref{eq:A_kxi},
 \begin{equation}\label{mu1}
 \begin{split}
 \mu_k(t,\xi)&:=\frac{\langle k,\xi\rangle^{1/2}}{\langle t\rangle^{1+\sigma_0}}+\frac{\mu^{\ast}(t,\xi)}{1+e^{\sqrt{\delta}(|k|^{1/2}-\langle\xi\rangle^{1/2})}b_k(t,\xi)},\\
 \mu_R(t,\xi)&:=\frac{\langle\xi\rangle^{1/2}}{\langle t\rangle^{1+\sigma_0}}+\mu^{\ast}(t,\xi).
 \end{split}
 \end{equation}
We record below the main properties of the weights $\mu_R$ and $\mu_k$. 

\begin{lemma}\label{Lmu1}
(i) For $t\ge0, \xi\in\R, k\in\mathbb{Z}$, we have
\begin{equation}\label{mudA1}
\mu_k(t,\xi)\approx_{\delta}\big|(\dot{A}_k/A_k)(t,\xi)\big|\qquad \mathrm{ and }\qquad\mu_R(t,\xi)\approx_{\delta}\big|(\dot{A}_R/A_R)(t,\xi)\big|.
\end{equation}

(ii) Assume that $\xi,\eta\in\mathbb{R}$, $k\in\mathbb{Z}$, and $t\ge 0$. Then
\begin{equation}\label{Genmu}
\mu_k(t,\xi)\lesssim_\delta\mu_k(t,\eta)e^{6\sqrt\delta|\xi-\eta|^{1/2}}.
\end{equation}
Moreover, if $\langle\xi-\eta\rangle\leq (\langle k,\xi\rangle+\langle k,\eta\rangle)/8$, then there is $C_1(\delta)\gg 1$ such that
\begin{equation}\label{Lmu3.2}
\big|\mu_k(t,\xi)-\mu_k(t,\eta)\big|\lesssim \langle\xi-\eta\rangle\mu_k(t,\eta)\,e^{4\sqrt{\delta}|\xi-\eta|^{1/2}}\Big[\frac{C_1(\delta)}{\langle k,\xi\rangle^{1/8}}+\sqrt{\delta}\Big].
\end{equation}
\end{lemma}

In other words, the weights $\mu_Y$ are proportional to the weights $|\dot{A}_Y/A_Y|$, but have better smoothness properties. See \cite[Lemmas 7.6 and 7.7]{IOJI2} for the proofs. 

\subsection{Bilinear estimates} To bound nonlinear terms we need bilinear estimates involving the weights. Many such estimates are proved in \cite[Section 8]{IOJI}. We use all of these bilinear estimates in this paper as well, since our proof here contains all the difficulties of the proof for the Couette flow treated in \cite{IOJI}. In addition, we need four more bilinear estimates to deal with the new terms in the equation \eqref{rea23.1} for $F$, which we prove in this section. 

We start with a lemma that is used many times in this paper. See \cite[Lemmas 8.2 and 8.3]{IOJI} for the proofs.

\begin{lemma}\label{OldBilin}
(i) For any $t\in[0,\infty)$, $\alpha\in[0,4]$, $\xi,\eta\in\R$, and $Y\in\{NR,R\}$ we have
\begin{equation}\label{TLX4}
\langle\xi\rangle^{-\alpha}A_Y(t,\xi)\lesssim_\delta \langle\xi-\eta\rangle^{-\alpha}A_Y(t,\xi-\eta)\langle\eta\rangle^{-\alpha}A_Y(t,\eta)e^{-(\delta_0/20)\min(\langle\xi-\eta\rangle,\langle \eta\rangle)^{1/2}}
\end{equation}
and
\begin{equation}\label{DtVMulti}
\big|(\dot{A}_Y/A_Y)(t,\xi)\big|\lesssim_\delta \big\{\big|(\dot{A}_Y/A_Y)(t,\xi-\eta)\big|+\big|(\dot{A}_Y/A_Y)(t,\eta)\big|\big\}e^{4\sqrt\delta\min(\langle\xi-\eta\rangle,\langle \eta\rangle)^{1/2}}.
\end{equation}

(ii) For any $t\in[0,\infty)$, $\xi,\eta\in\R$, and $k\in\mathbb{Z}$ we have
\begin{equation}\label{TLX7}
A_k(t,\xi)\lesssim_\delta A_R(t,\xi-\eta)A_k(t,\eta)e^{-(\delta_0/20)\min(\langle\xi-\eta\rangle,\langle k,\eta\rangle)^{1/2}}
\end{equation}
and
\begin{equation}\label{vfc30.7}
\big|(\dot{A}_k/A_k)(t,\xi)\big|\lesssim_\delta \big\{\big|(\dot{A}_R/A_R)(t,\xi-\eta)\big|+\big|(\dot{A}_k/A_k)(t,\eta)\big|\big\}e^{12\sqrt\delta\min(\langle\xi-\eta\rangle,\langle k,\eta\rangle)^{1/2}}.
\end{equation}
\end{lemma}

To state our new estimates we let $\delta'_0:=\delta_0/200$ and define the sets
\begin{equation}\label{nar18.1}
\begin{split}
R_0:=\Big\{&((k,\xi),(\ell,\eta))\in (\Z\times \R)^2:\\
&\min(\langle k,\xi\rangle,\,\langle\ell,\eta\rangle,\,\langle k-\ell,\xi-\eta\rangle)\geq \frac{\langle k,\xi\rangle+\langle\ell,\eta\rangle+\langle k-\ell,\xi-\eta\rangle}{20}\Big\},\\
\end{split}
\end{equation}
\begin{equation}\label{nar18.2}
R_1:=\Big\{((k,\xi),(\ell,\eta))\in (\Z\times \R)^2:\,\langle k-\ell,\xi-\eta\rangle\leq \frac{\langle k,\xi\rangle+\langle\ell,\eta\rangle+\langle k-\ell,\xi-\eta\rangle}{10}\Big\},
\end{equation}
\begin{equation}\label{nar18.3}
R_2:=\Big\{((k,\xi),(\ell,\eta))\in (\Z\times \R)^2:\,\langle\ell,\eta\rangle\leq \frac{\langle k,\xi\rangle+\langle\ell,\eta\rangle+\langle k-\ell,\xi-\eta\rangle}{10}\Big\},
\end{equation}
\begin{equation}\label{nar18.4}
R_3:=\Big\{((k,\xi),(\ell,\eta))\in (\Z\times \R)^2:\,\langle k,\xi\rangle\leq \frac{\langle k,\xi\rangle+\langle\ell,\eta\rangle+\langle k-\ell,\xi-\eta\rangle}{10}\Big\}.
\end{equation}

\begin{lemma}\label{TLXH1}
Assume that $t\ge1$, $k,\ell\in\mathbb{Z}$, $\xi,\eta\in\mathbb{R}$, let $(m,\rho):=(k-\ell,\xi-\eta)$, and assume that $m\neq 0$.

(i) If $((k,\xi),(\ell,\eta))\in R_0\cup R_1$, then
\begin{equation}\label{TLXH1.1}
\begin{split}
\frac{(|\rho/m|+\langle t \rangle)\langle \rho \rangle}{\langle t\rangle m^2\langle t-\rho/m \rangle^2}&\big|\ell A_k^2(t,\xi)-kA_{\ell}^2(t,\eta)\big|\\
&\lesssim_{\delta}\sqrt{|(A_k\dot{A}_k)(t,\xi)|}\,\sqrt{|(A_{\ell}\dot{A}_{\ell})(t,\eta)|}\,A_{m}(t,\rho) \,e^{-\delta'_0\langle m,\rho \rangle^{1/2}}
\end{split}
\end{equation}
and
\begin{equation}\label{TLD1}
\begin{split}
\frac{(|\rho/m|+\langle t \rangle)\langle \rho \rangle}{\langle t\rangle m^2\langle t-\rho/m \rangle^2}&\frac{\big|\ell A_k^2(t,\xi)\big|+\big|kA_{\ell}^2(t,\eta)\big|}{(1+\langle k,\xi\rangle/\langle t\rangle)^{1/2}}\\
&\lesssim_{\delta}\sqrt{|(A_k\dot{A}_k)(t,\xi)|}\,\sqrt{|(A_{\ell}\dot{A}_{\ell})(t,\eta)|}\,A_{m}(t,\rho) \,e^{-\delta'_0\langle m,\rho \rangle^{1/2}}.
\end{split}
\end{equation}

(ii) If $((k,\xi),(\ell,\eta))\in R_2$, then
\begin{equation}\label{TLXH1.2}
\begin{split}
\frac{(|\rho/m|+\langle t \rangle)\langle \rho \rangle}{\langle t\rangle m^2\langle t-\rho/m \rangle^2}&\big\{\big|\ell A_k^2(t,\xi)\big|+\big|kA_{\ell}^2(t,\eta)\big|\big\}\\
&\lesssim_{\delta}\sqrt{|(A_k\dot{A}_k)(t,\xi)|}\,\sqrt{|(A_{m}\dot{A}_{m})(t,\rho)|}\,A_{\ell}(t,\eta) \,e^{-\delta'_0\langle \ell,\eta \rangle^{1/2}}.
\end{split}
\end{equation}
\end{lemma}

\begin{proof} The bounds \eqref{TLXH1.1} and \eqref{TLXH1.2} are proved in \cite[Lemma 8.4]{IOJI}. The statement of \eqref{TLXH1.2} is slightly weaker in \cite[Lemma 8.4]{IOJI}, in the sense that the quantity $\big|\ell A_k^2(t,\xi)\big|+\big|kA_{\ell}^2(t,\eta)\big|$ in the left-hand side is replaced by the smaller quantity $\big|\ell A_k^2(t,\xi)-kA_{\ell}^2(t,\eta)\big|$, but the proof itself does not use the symmetrization and applies to the larger quantity as well.

We prove now the new bounds \eqref{TLD1}. Notice that
\begin{equation*}
\langle t\rangle^2\frac{(|\rho/m|+\langle t \rangle)\langle \rho \rangle}{\langle t\rangle m^2\langle t-\rho/m \rangle^2}+\frac{1+\langle \ell,\eta\rangle/\langle t\rangle}{1+\langle k,\xi\rangle/\langle t\rangle}+\frac{1+\langle k,\xi\rangle/\langle t\rangle}{1+\langle \ell,\eta\rangle/\langle t\rangle}\lesssim_\delta e^{\delta\langle m,\rho \rangle^{1/2}}.
\end{equation*}
By symmetry, for \eqref{TLD1} it suffices to prove that
\begin{equation*}
\begin{split}
\big|\ell A_k^2(t,\xi)\big|\lesssim_{\delta}\big(1+\langle k,\xi\rangle/\langle t\rangle\big)^{1/2}\langle t\rangle^2\sqrt{|(A_k\dot{A}_k)(t,\xi)|}\,\sqrt{|(A_{\ell}\dot{A}_{\ell})(t,\eta)|}\,A_{m}(t,\rho) \,e^{-(3\delta'_0/2)\langle m,\rho \rangle^{1/2}}.
\end{split}
\end{equation*}
This is equivalent to proving that 
\begin{equation}\label{TLD1.5}
\begin{split}
\big|\ell A_k(t,\xi)\big|\lesssim_{\delta}\big(1+\langle k,\xi\rangle/\langle t\rangle\big)^{1/2}\langle t\rangle^2\sqrt{|(\dot{A}_k/A_k)(t,\xi)|}\,\sqrt{|(\dot{A}_{\ell}/A_\ell)(t,\eta)|}\\
\times A_{\ell}(t,\eta)\,A_{m}(t,\rho) \,e^{-(3\delta'_0/2)\langle m,\rho \rangle^{1/2}}.
\end{split}
\end{equation}
In view of \eqref{eq:CDW} we may replace $|(\dot{A}_{\ell}/A_\ell)(t,\eta)|$ with $|(\dot{A}_k/A_k)(t,\xi)|$ at the expense of an acceptable factor. For later use (in the proof of \eqref{TLD10} below), we prove the stronger bounds
\begin{equation}\label{TLD2}
\begin{split}
\langle k,\xi\rangle \frac{A_k(t,\xi)}{A_\ell(t,\eta)}\lesssim_{\delta}(1+\langle k,\xi\rangle/\langle t\rangle)^{1/2}\langle t\rangle^2|(\dot{A}_k/A_k)(t,\xi)| A_{m}(t,\rho) \,e^{-2\delta'_0\langle m,\rho \rangle^{1/2}},
\end{split}
\end{equation}
provided that $((k,\xi),(\ell,\eta))\in R_0\cup R_1$ and $t\geq 1$.

For this we use first the following elementary observation: if $a,b\in\mathbb{R}^d$ and $\beta\in[0,1]$ then
\begin{equation}\label{TLD3}
\text{ if }\quad\langle b\rangle\geq\beta \langle a\rangle\quad\text{ then }\quad \langle a+b\rangle^{1/2}\leq \langle b\rangle^{1/2}+(1-\sqrt\beta/2)\langle a\rangle^{1/2}.
\end{equation}
Notice that $20\langle l,\eta\rangle\geq\langle m,\rho\rangle$ (since $((k,\xi),(\ell,\eta))\in R_0\cup R_1$), and $|(\dot{A}_k/A_k)(t,\xi)|\gtrsim\langle t\rangle^{-1-\sigma_0}$ (see \eqref{eq:A_kxi}). Using also \eqref{dor4.1} and \eqref{TLD3}, the bounds \eqref{TLD2} follow if $\langle k,\xi\rangle \leq 100\langle m,\rho\rangle$. 

On the other hand, if $\langle k,\xi\rangle \geq 100\langle m,\rho\rangle$ then we consider two cases. If $t\notin I_{k,\xi}$ then we simply use \eqref{eq:A_kxi} to bound $|(\dot{A}_k/A_k)(t,\xi)|\gtrsim_\delta \langle k,\xi\rangle^{1/2}\langle t\rangle^{-1-\sigma_0}$. The desired bounds \eqref{TLD2} follow using also \eqref{vfc26}. If $t\in I_{k,\xi}$ (in particular $1\leq |k|\leq \delta|\xi|$ and $t\approx \xi/k$), then 
\begin{equation*}
\begin{split}
\langle k,\xi\rangle \frac{A_k(t,\xi)}{A_\ell(t,\eta)}\lesssim_{\delta}\frac{|\xi|^2/k^2}{\langle t-\xi/k\rangle} e^{0.9\lambda(t)\langle m,\rho\rangle^{1/2}}\lesssim_{\delta}\frac{\langle t\rangle^2}{\langle t-\xi/k\rangle} A_{m}(t,\rho)e^{-(\delta_0/20)\langle m,\rho\rangle^{1/2}}
\end{split}
\end{equation*}
using \eqref{vfc27}. The bounds \eqref{TLD2} follow since
\begin{equation*}
\begin{split}
\frac{1}{\langle t-\xi/k\rangle}\lesssim_\delta\frac{\partial_tw_k(t,\xi)}{w_k(t,\xi)}\lesssim_\delta\frac{|\dot{A}_k(t,\xi)|}{A_k(t,\xi)},
\end{split}
\end{equation*}
as a consequence of \eqref{reb8} and \eqref{eq:A_kxi}.
\end{proof}

\begin{lemma}\label{TLXH3}
Assume that $t\ge1$, $k,\ell\in\mathbb{Z}$, $\xi,\eta\in\mathbb{R}$, let $(m,\rho):=(k-\ell,\xi-\eta)$, and assume that $m\neq 0$.

(i) If $((k,\xi),(\ell,\eta))\in R_0\cup R_1$, then
\begin{equation}\label{TLXH3.1}
\begin{split}
\frac{|\rho/m|^2+\langle t \rangle^2}{|m|\langle t\rangle^2\langle t-\rho/m \rangle^2}&\big|\eta A_k^2(t,\xi)-\xi A_{\ell}^2(t,\eta)\big|\\
&\lesssim_{\delta}\sqrt{|(A_k\dot{A}_k)(t,\xi)|}\,\sqrt{|(A_{\ell}\dot{A}_{\ell})(t,\eta)|}\,A_{m}(t,\rho) \,e^{-\delta'_0\langle m,\rho \rangle^{1/2}}
\end{split}
\end{equation}
and
\begin{equation}\label{TLD10}
\begin{split}
\frac{|\rho/m|^2+\langle t \rangle^2}{|m|\langle t\rangle^2\langle t-\rho/m \rangle^2}&\frac{\big|\eta A_k^2(t,\xi)\big|+\big|\xi A_{\ell}^2(t,\eta)\big|}{\big(1+\langle k,\xi\rangle/\langle t\rangle\big)^{1/2}}\\
&\lesssim_{\delta}\sqrt{|(A_k\dot{A}_k)(t,\xi)|}\,\sqrt{|(A_{\ell}\dot{A}_{\ell})(t,\eta)|}\,A_{m}(t,\rho) \,e^{-\delta'_0\langle m,\rho \rangle^{1/2}}
\end{split}
\end{equation}

(ii) If $((k,\xi),(\ell,\eta))\in R_2$, then
\begin{equation}\label{TLXH3.2}
\begin{split}
\frac{|\rho/m|^2+\langle t \rangle^2}{|m|\langle t\rangle^2\langle t-\rho/m \rangle^2}&\big\{\big|\eta A_k^2(t,\xi)\big|+\big|\xi A_{\ell}^2(t,\eta)\big|\big\}\\
&\lesssim_{\delta}\sqrt{|(A_k\dot{A}_k)(t,\xi)|}\,\sqrt{|(A_{m}\dot{A}_{m})(t,\rho)|}\,A_{\ell}(t,\eta) \,e^{-\delta'_0\langle \ell,\eta \rangle^{1/2}}.
\end{split}
\end{equation}
\end{lemma}

\begin{proof} The bounds \eqref{TLXH3.1} and \eqref{TLXH3.2} are proved in \cite[Lemma 8.5]{IOJI}, with the same remark as before that the inequality \eqref{TLXH3.2} is slightly weaker in \cite[Lemma 8.5]{IOJI}, but its proof does not use the symmetrization. To prove \eqref{TLD10} we notice again that
\begin{equation*}
\langle t\rangle^2\frac{|\rho/m|^2+\langle t \rangle^2}{|m|\langle t\rangle^2\langle t-\rho/m \rangle^2}+\frac{1+\langle \ell,\eta\rangle/\langle t\rangle}{1+\langle k,\xi\rangle/\langle t\rangle}+\frac{1+\langle k,\xi\rangle/\langle t\rangle}{1+\langle \ell,\eta\rangle/\langle t\rangle}\lesssim_\delta e^{\delta\langle m,\rho \rangle^{1/2}}.
\end{equation*}
Therefore for \eqref{TLD10} it suffices to prove that if $((k,\xi),(\ell,\eta))\in R_0\cup R_1$ then
\begin{equation*}
\begin{split}
|\eta| A_k^2(t,\xi)\lesssim_\delta \big(1+\langle k,\xi\rangle/\langle t\rangle\big)^{1/2}\langle t\rangle^2\sqrt{|(A_k\dot{A}_k)(t,\xi)|}\,\sqrt{|(A_{\ell}\dot{A}_{\ell})(t,\eta)|}\,A_{m}(t,\rho) \,e^{-(3\delta'_0/2)\langle m,\rho \rangle^{1/2}}.
\end{split}
\end{equation*}
As in the proof of Lemma \ref{TLXH1}, this follows from \eqref{TLD2}.
\end{proof}

\begin{lemma}\label{TLXH2}
Define the sets $R_n^\ast:=\{((k,\xi),(\ell,\eta))\in R_n:\,k=l\}$. Assume that $t\ge1$, $k\in\mathbb{Z}$, $\xi,\eta\in\mathbb{R}$, and let $\rho:=\xi-\eta$.

(i) If $((k,\xi),(k,\eta))\in R_0^\ast\cup R_1^\ast$, then
\begin{equation}\label{TLXH2.1}
\begin{split}
\frac{\big|\eta A_k^2(t,\xi)-\xi A_{k}^2(t,\eta)\big|}{\langle\rho\rangle\langle t\rangle+\langle \rho\rangle^{1/4}\langle t\rangle^{7/4}}\lesssim_{\delta}\sqrt{|(A_k\dot{A}_k)(t,\xi)|}\,\sqrt{|(A_{k}\dot{A}_{k})(t,\eta)|}\,A_{NR}(t,\rho) \,e^{-\delta'_0\langle \rho \rangle^{1/2}}.
\end{split}
\end{equation}
and
\begin{equation}\label{TLD20}
\begin{split}
\frac{\big|\eta A_k^2(t,\xi)\big|+\big|\xi A_{k}^2(t,\eta)\big|}{\big(\langle\rho\rangle\langle t\rangle+\langle \rho\rangle^{1/4}\langle t\rangle^{7/4}\big)\big(1+\langle k,\xi\rangle/\langle t\rangle\big)^{1/2}}&\\
\lesssim_{\delta}\sqrt{|(A_k\dot{A}_k)(t,\xi)|}&\,\sqrt{|(A_{k}\dot{A}_{k})(t,\eta)|}\,A_{NR}(t,\rho) \,e^{-\delta'_0\langle \rho \rangle^{1/2}}.
\end{split}
\end{equation}

(ii) If $((k,\xi),(k,\eta))\in R_2^\ast$, then
\begin{equation}\label{TLXH2.2}
\begin{split}
\frac{\big|\eta A_k^2(t,\xi)\big|+\big|\xi A_{k}^2(t,\eta)\big|}{\langle\rho\rangle\langle t\rangle+\langle \rho\rangle^{1/4}\langle t\rangle^{7/4}}\lesssim_{\delta}\sqrt{|(A_k\dot{A}_k)(t,\xi)|}\,\sqrt{|(A_{NR}\dot{A}_{NR})(t,\rho)|}\,A_{k}(t,\eta) \,e^{-\delta'_0\langle k,\eta \rangle^{1/2}}.
\end{split}
\end{equation}
\end{lemma}

\begin{proof} The bounds \eqref{TLXH2.1} and \eqref{TLXH2.2} are proved in \cite[Lemma 8.6]{IOJI}, with the same remark as before that the inequality \eqref{TLXH2.2} is slightly weaker in \cite[Lemma 8.6]{IOJI}, but its proof does not use the symmetrization. For \eqref{TLD20} it suffices to prove that if $((k,\xi),(k,\eta))\in R_0^\ast\cup R_1^\ast$ then
\begin{equation*}
\begin{split}
|\eta| A_k^2(t,\xi)\lesssim_\delta &\big(1+\langle k,\xi\rangle/\langle t\rangle\big)^{1/2}\langle t\rangle^{7/4}\\
&\times\sqrt{|(A_k\dot{A}_k)(t,\xi)|}\,\sqrt{|(A_k\dot{A}_k)(t,\eta)|}\,A_{NR}(t,\rho) \,e^{-(3\delta'_0/2)\langle \rho \rangle^{1/2}}.
\end{split}
\end{equation*}
As in the proof of Lemma \ref{TLXH1}, using \eqref{eq:CDW} it suffices to prove that
\begin{equation}\label{TLD22}
\begin{split}
\langle k,\xi\rangle \frac{A_k(t,\xi)}{A_k(t,\eta)}\lesssim_{\delta}(1+\langle k,\xi\rangle/\langle t\rangle)^{1/2}\langle t\rangle^{7/4}|(\dot{A}_k/A_k)(t,\xi)| A_{NR}(t,\rho) \,e^{-2\delta'_0\langle\rho \rangle^{1/2}}.
\end{split}
\end{equation} 
Since $|(\dot{A}_k/A_k)(t,\xi)|\gtrsim_\delta \langle k,\xi\rangle^{1/2}\langle t\rangle^{-1-\sigma_0}$ (see \eqref{eq:A_kxi}), for \eqref{TLD22} it suffices to prove that
\begin{equation}\label{TLD23}
A_k(t,\xi)\lesssim_{\delta}A_k(t,\eta)A_{NR}(t,\rho) \,e^{-2\delta'_0\langle\rho \rangle^{1/2}}.
\end{equation} 
This follows from \eqref{vfc25} if $((k,\xi),(k,\eta))\in R_1^\ast$, or from the bounds \eqref{TLD3},  $A_k(t,\eta)\geq e^{\lambda(t)\langle k,\eta\rangle^{1/2}}$, $A_{NR}(t,\rho)\geq e^{\lambda(t)\langle\rho\rangle^{1/2}}$, $A_k(t,\xi)\leq 2e^{\lambda(t)\langle k,\xi\rangle^{1/2}}e^{2\sqrt\delta\langle k,\xi\rangle^{1/2}}$ (see \eqref{dor4.1}) if $((k,\xi),(k,\eta))\in R_0^\ast$.
\end{proof}

\begin{lemma}\label{TLD30}
Assume that $t\ge1$, $k\in\mathbb{Z}$, $\xi,\eta\in\mathbb{R}$, and let $\rho:=\xi-\eta$.

(i) If $((k,\xi),(k,\eta))\in R_0^\ast\cup R_1^\ast$ and $k\neq 0$ then
\begin{equation}\label{TLD31}
\begin{split}
A_k^2(t,\xi)\lesssim_{\delta}\sqrt{|(A_k\dot{A}_k)(t,\xi)|}\,\sqrt{|(A_{k}\dot{A}_{k})(t,\eta)|}\frac{|k|\langle t\rangle\langle t-\eta/k\rangle^2}{\langle t\rangle +|\eta/k|}\,A_{R}(t,\rho) \,e^{-\delta'_0\langle \rho \rangle^{1/2}}.
\end{split}
\end{equation}

(ii) If $((k,\xi),(k,\eta))\in R_2^\ast$ and $k\neq 0$ then
\begin{equation}\label{TLD32}
\begin{split}
A_k^2(t,\xi)\lesssim_{\delta}\sqrt{|(A_k\dot{A}_k)(t,\xi)|}\,\sqrt{|(A_{R}\dot{A}_{R})(t,\rho)|}\frac{|k|\langle t\rangle\langle t-\eta/k\rangle^2}{\langle t\rangle +|\eta/k|}A_{k}(t,\eta) \,e^{-\delta'_0\langle k,\eta \rangle^{1/2}}.
\end{split}
\end{equation}

(iii) If $((k,\xi),(k,\eta))\in R_3^\ast$ and $k\neq 0$ then
\begin{equation}\label{TLD34}
\begin{split}
A_k^2(t,\xi)\lesssim_{\delta}\sqrt{|(A_k\dot{A}_k)(t,\eta)|}\,\sqrt{|(A_{R}\dot{A}_{R})(t,\rho)|}\frac{|k|\langle t\rangle\langle t-\eta/k\rangle^2}{\langle t\rangle +|\eta/k|}A_{k}(t,\xi) \,e^{-\delta'_0\langle k,\xi \rangle^{1/2}}.
\end{split}
\end{equation}
\end{lemma}

\begin{proof} (i) Using \eqref{eq:CDW}, it suffices to prove that
\begin{equation*}
\begin{split}
A_k(t,\xi)\lesssim_{\delta}A_k(t,\eta)|(\dot{A}_{k}/A_{k})(t,\eta)|\frac{|k|\langle t\rangle\langle t-\eta/k\rangle^2}{\langle t\rangle +|\eta/k|}\,A_{R}(t,\rho) \,e^{-(3\delta'_0/2)\langle \rho \rangle^{1/2}}.
\end{split}
\end{equation*}
This follows from \eqref{TLX7} and \eqref{eq:A_kxi}--\eqref{newAdot}.

(ii) Since $4\langle k,\eta\rangle\leq\min(\langle k,\xi\rangle,\langle\rho\rangle)$, we can apply \eqref{eq:CDW}--\eqref{vfc30.5}. Notice also that
\begin{equation*}
\frac{|k|\langle t\rangle\langle t-\eta/k\rangle^2}{\langle t\rangle +|\eta/k|}\gtrsim_\delta\langle t\rangle^2e^{-\delta\langle k,\eta\rangle^{1/2}}.
\end{equation*}
For \eqref{TLD32} it suffices to prove that 
\begin{equation*}\label{TLD33}
\begin{split}
A_k(t,\xi)\lesssim_{\delta}|(\dot{A}_k/A_k)(t,\xi)|\langle t\rangle^2A_{R}(t,\rho)A_{k}(t,\eta) \,e^{-(3\delta'_0/2)\langle k,\eta \rangle^{1/2}}.
\end{split}
\end{equation*}
Since $|(\dot{A}_k/A_k)(t,\xi)|\gtrsim_\delta \langle t\rangle^{-1-\sigma_0}$, this follows from  \eqref{TLX7}.

(iii) Since $4\langle k,\xi\rangle\leq\min(\langle k,\eta\rangle,\langle\rho\rangle)$, the desired bounds \eqref{TLD34} follow easily from \eqref{dor4.1} and \eqref{TLX3.5}--\eqref{eq:A_kxi}.
\end{proof}

\section{Nonlinear bounds and the main elliptic estimate} We prove now estimates on some of the functions defined in Proposition \ref{ChangedEquations}. In most cases we apply the definitions, the bootstrap assumptions \eqref{boot2}, and the following general  lemma (see \cite[Lemma 8.1]{IOJI} for the proof).

\begin{lemma}\label{Multi0}
(i) Assume that $m,m_1,m_2:\R\to\C$ are symbols satisfying
\begin{equation}\label{TLX5}
|m(\xi)|\leq |m_1(\xi-\eta)|\,|m_2(\eta)|\{\langle\xi-\eta\rangle^{-2}+\langle\eta\rangle^{-2}\}
\end{equation}
for any $\xi,\eta\in\R$. If $M, M_1, M_2$ are the operators defined by these symbols then
\begin{equation}\label{TLX6}
\|M(gh)\|_{L^2(\R)}\lesssim \|M_1g\|_{L^2(\R)}\|M_2h\|_{L^2(\R)}.
\end{equation}

(ii) Similarly, if $m,m_2:\Z\times\R\to\C$ and $m_1:\R\to\C$ are symbols satisfying
\begin{equation}\label{TLX5.1}
|m(k,\xi)|\leq |m_1(\xi-\eta)|\,|m_2(k,\eta)|\{\langle\xi-\eta\rangle^{-2}+\langle k,\eta\rangle^{-2}\}
\end{equation}
for any $\xi,\eta\in\R$, $k\in\Z$, and $M, M_1, M_2$ are the operators defined by these symbols, then
\begin{equation}\label{TLX6.1}
\|M(gh)\|_{L^2(\mathbb{T}\times\R)}\lesssim \|M_1g\|_{L^2(\R)}\|M_2h\|_{L^2(\mathbb{T}\times\R)}.
\end{equation}

(iii) Finally, assume that $m,m_1,m_2:\Z\times\R\to\C$ are symbols satisfying
\begin{equation}\label{TLX5.2}
|m(k,\xi)|\leq |m_1(k-\ell,\xi-\eta)|\,|m_2(\ell,\eta)|\{\langle k-\ell,\xi-\eta\rangle^{-2}+\langle \ell,\eta\rangle^{-2}\}
\end{equation}
for any $\xi,\eta\in\R$, $k,\ell\in\Z$. If $M, M_1, M_2$ are the operators defined by these symbols, then
\begin{equation}\label{TLX6.2}
\|M(gh)\|_{L^2(\mathbb{T}\times\R)}\lesssim \|M_1g\|_{L^2(\mathbb{T}\times\R)}\|M_2h\|_{L^2(\mathbb{T}\times\R)}.
\end{equation}
\end{lemma}

For $f\in C([0,T]:H^{4}(\mathbb{R}))$, $g\in C([0,T]:H^4(\mathbb{T}\times\mathbb{R}))$, $Y\in\{R,NR\}$, and $t_1,t_2\in[0,T]$ we define
\begin{equation}\label{rew0}
\|f\|^2_{Y[t_1,t_2]}:=\sup_{t\in [t_1,t_2]}\int_{\R}A_Y^2(t,\xi)\big|\widetilde{f}(t,\xi)\big|^2\,d\xi+\int_{t_1}^{t_2}\int_{\R}|\dot{A}_Y(s,\xi)|A_Y(s,\xi)\big|\widetilde{f}(s,\xi)\big|^2\,d\xi ds,
\end{equation}
\begin{equation}\label{rew0.2}
\begin{split}
\|g\|^2_{W[t_1,t_2]}:=\sup_{t\in [t_1,t_2]}\Big\{&\sum_{k\in\mathbb{Z}}\int_{\R}A_k^2(t,\xi)\big|\widetilde{g}(t,k,\xi)\big|^2\,d\xi\Big\}\\
&+ \int_{t_1}^{t_2}\sum_{k\in\mathbb{Z}}\int_{\R}\big|(A_k\dot{A}_k)(s,\xi)\big|\widetilde{g}(s,k,\xi)\big|^2\,d\xi ds,
\end{split}
\end{equation}
\begin{equation}\label{rew0.4}
\begin{split}
\|g\|^2_{\widetilde{W}[t_1,t_2]}:=\sup_{t\in [t_1,t_2]}\Big\{&\sum_{k\in\mathbb{Z}^\ast}\int_{\R}A_k^2(t,\xi)\frac{k^2\langle t\rangle^2}{|\xi|^2+k^2\langle t\rangle^2}\big|\widetilde{g}(t,k,\xi)\big|^2\,d\xi\Big\}\\
&+ \int_{t_1}^{t_2}\sum_{k\in\mathbb{Z}^\ast}\int_{\R}\big|(A_k\dot{A}_k)(s,\xi)\big|\frac{k^2\langle s\rangle^2}{|\xi|^2+k^2\langle s\rangle^2}\big|\widetilde{g}(s,k,\xi)\big|^2\,d\xi ds.
\end{split}
\end{equation}
For simplicity of notation, let $Y:=Y[1,T]$, $W:=W[1,T]$, $\widetilde{W}:=\widetilde{W}[1,T]$.

The resonant weights $A_R$ are our strongest weights. We show first that $R$ is an algebra, and, in fact, multiplication by functions in $R$ preserves the norms $W$ and $\widetilde{W}$. More precisely:

\begin{lemma}\label{nar8}
(i) If $f,g\in C([1,T]:H^{4}(\mathbb{R}))$ and $H\in C([1,T]:H^{4}(\mathbb{T}\times\mathbb{R}))$, then
\begin{equation}\label{rew1}
\|fg\|_R\lesssim_\delta \|f\|_R\|g\|_R.
\end{equation}
and
\begin{equation}\label{rew1.2}
\|fH\|_W\lesssim_\delta \|f\|_R\|H\|_W,\qquad \|fH\|_{\widetilde{W}}\lesssim_\delta \|f\|_R\|H\|_{\widetilde{W}}.
\end{equation}

(ii) As a consequence, if $\Psi_1$ is a Gevrey cutoff function supported in $\big[b(\vartheta_0/20), b(1-\vartheta_0/20)\big]$ and satisfying $\big\|e^{\langle\xi\rangle^{3/4}}\widetilde{\Psi_1}(\xi)\big\|_{L^\infty}\lesssim 1$,  and 
\begin{equation}\label{rew3}
h\in\big\{\Psi_1(V')^a,\,\Psi_1(B')^a,\,\langle\partial_v\rangle^{-1}V'',\,B'':\,a\in[-2,2]\cap\mathbb{Z}\big\},
\end{equation}
then
\begin{equation}\label{nar4}
\|h\|_R\lesssim_\delta 1.
\end{equation} 
Moreover, the functions $B'_0$ and $B''_0$ do not depend on $t$ and satisfy the stronger bounds
\begin{equation}\label{rew6.44}
\|\Psi_1B'_0\|_{\mathcal{G}^{4\delta_0,1/2}}+\|\Psi_1(1/B'_0)\|_{\mathcal{G}^{4\delta_0,1/2}}+\|B''_0\|_{\mathcal{G}^{4\delta_0,1/2}}\lesssim 1.
\end{equation}

(iii) With $\mathcal{K}$ as in \eqref{rec3} and $h$ satisfying $\|h\|_R+\|\partial_vh\|_R\leq 1$, for any $t\in[1,T]$ we have
\begin{equation}\label{nar7}
\begin{split}
&\int_{\R}A_{NR}^2(t,\xi)\big(\langle\xi\rangle^2\langle t\rangle^2+\mathcal{K}^2\langle\xi\rangle^{1/2}\langle t\rangle^{7/2}\big)\big|\widetilde{(h\dot{V})}(t,\xi)\big|^2\,d\xi\lesssim_\delta\eps_1^2,\\
&\int_1^t\int_{\R}|\dot{A}_{NR}(s,\xi)|A_{NR}(s,\xi)\big(\langle\xi\rangle^2\langle s\rangle^2+\mathcal{K}^2\langle\xi\rangle^{1/2}\langle s\rangle^{7/2}\big)\big|\widetilde{(h\dot{V})}(s,\xi)\big|^2\,d\xi ds\lesssim_\delta\eps_1^2.
\end{split}
\end{equation}
The implicit constants in \eqref{nar7} may depend on $\delta$, and $\mathcal{K}$ is assumed large enough compared to these constants.
\end{lemma}

\begin{proof} (i) 
The bounds \eqref{rew1} follow using Lemma \ref{Multi0} (i) and the bilinear estimates  \eqref{TLX4}--\eqref{DtVMulti} with $Y=R$ and $\alpha=0$. (see \cite[Lemma 4.2]{IOJI2} for complete details). To prove the bounds \eqref{rew1.2} we use the bilinear estimates \eqref{TLX7}--\eqref{vfc30.7}. Moreover, if $k\neq 0$ it is easy to see that
\begin{equation}\label{rew4.2}
\frac{|k|\langle t\rangle}{|\xi|+|k|\langle t\rangle}\lesssim_\delta \frac{|k|\langle t\rangle}{|\eta|+|k|\langle t\rangle}e^{\delta\min(\langle\xi-\eta\rangle,\langle k,\eta\rangle)^{1/2}},
\end{equation}
and the desired bounds \eqref{rew1.2} follow using also Lemma \ref{Multi0} (ii).

(ii) To prove the bounds \eqref{nar4} we write
\begin{equation}\label{rew6}
\begin{split}
&B'(t,v)=B'_{\ast}(t,v)+B'_0(v),\qquad V'(t,v)=V'_{\ast}(t,v)+B'_0(v),\\
&B''(t,v)=B''_{\ast}(t,v)+B''_0(v),\\
\end{split}
\end{equation}
and recall also that $V''=\partial_v(V')^2/2$, see \eqref{rea27}. The functions $B'_0$, $1/B'_0$, and $B''_0$ do not depend on $t$ and satisfy the bounds \eqref{rew6.44}, as a consequence of Lemmas \ref{lm:Gevrey} and \ref{GPF} and the assumptions \eqref{IntB1}--\eqref{IntB2}. The desired bounds \eqref{nar4} follow using the algebra property \eqref{rew1}, the bootstrap assumptions \eqref{boot2} on $V'_{\ast}, B'_{\ast},B''_{\ast}$, and the identities \eqref{rew6}, as long as $\eps_1$ is sufficiently small depending on $\delta$ (see \cite[Lemma 4.2]{IOJI2} for complete details).

(iii) To prove \eqref{nar7} we use the formula $\partial_v\dot{V}=\mathcal{H}/(tV')$ (see \eqref{rea27}) and the bootstrap assumptions \eqref{boot2}. Since $\dot{V}$ and $\mathcal{H}$  are supported in $[b(\vartheta_0),b(1-\vartheta_0)]$ we have
\begin{equation*}
tV'\partial_v\dot{V}=\Psi \mathcal{H}=\Psi(B'_\ast-V'_\ast-\langle F\rangle),
\end{equation*}
see \eqref{rea23}, where $\Psi$ is as in \eqref{rec0}. The bootstrap assumptions \eqref{boot2} show that
\begin{equation}\label{rew6.45}
\|V'_\ast\|_R+\|B'_\ast\|_R+\|\langle F\rangle\|_{NR}\lesssim \eps_1.
\end{equation}
The desired bounds \eqref{nar7} follow using \eqref{boot2}, the bilinear estimates \eqref{TLX4}--\eqref{DtVMulti} with $Y=NR$, and the bounds $\|\Psi (V')^{-1}\|_R\lesssim 1$ (see also \cite[Lemma 4.5]{IOJI2}).
\end{proof}

We record now bounds on some of the functions that apear in the right-hand sides of the equations \eqref{rea23.1} and \eqref{rea25}.

\begin{lemma}\label{nar13}
(i) For any $t\in[1,T]$ and $h_1\in\{(V')^a\partial_z(\Psi\phi):\,a\in[-2,2]\}$ we have
\begin{equation}\label{nar34}
\begin{split}
&\sum_{k\in \mathbb{Z}\setminus\{0\}}\int_{\R}A_k^2(t,\xi)\frac{k^2\langle t\rangle^4\langle t-\xi/k\rangle^4}{(|\xi/k|^2+\langle t\rangle^2)^2}\big|\widetilde{h_1}(t,k,\xi)\big|^2\,d\xi\lesssim_\delta\eps_1^2\\
&\int_1^t\sum_{k\in \mathbb{Z}\setminus\{0\}}\int_{\R}|\dot{A}_k(s,\xi)|A_k(s,\xi)\frac{k^2\langle s\rangle^4\langle s-\xi/k\rangle^4}{(|\xi/k|^2+\langle s\rangle^2)^2}\big|\widetilde{h_1}(s,k,\xi)\big|^2\,d\xi ds\lesssim_\delta\eps_1^2.
\end{split}
\end{equation}

(ii) For any $t\in[1,T]$ and $h_2\in\{(V')^a\partial_v\mathbb{P}_{\neq 0}(\Psi\phi):\,a\in[-2,2]\}$ we have
\begin{equation}\label{nar14} 
\begin{split}
&\sum_{k\in \mathbb{Z}\setminus\{0\}}\int_{\R}A_k^2(t,\xi)\frac{k^4\langle t\rangle^2\langle t-\xi/k\rangle^4}{(|\xi/k|^2+\langle t\rangle^2)\langle\xi\rangle^2}\big|\widetilde{h_2}(t,k,\xi)\big|^2\,d\xi\lesssim_\delta\eps_1^2\\
&\int_1^t\sum_{k\in \mathbb{Z}\setminus\{0\}}\int_{\R}|\dot{A}_k(s,\xi)|A_k(s,\xi)\frac{k^4\langle s\rangle^2\langle s-\xi/k\rangle^4}{(|\xi/k|^2+\langle s\rangle^2)\langle\xi\rangle^2}\big|\widetilde{h_2}(s,k,\xi)\big|^2\,d\xi ds\lesssim_\delta\eps_1^2.
\end{split}
\end{equation}

(iii) If $g\in \{(V')^a\langle \partial_z\phi\partial_vF\rangle,\,(V')^a\langle\partial_v\mathbb{P}_{\neq0}\phi\partial_zF\rangle,\,a\in[-2,2]\cap\mathbb{Z}\}$ then, for any $t\in[1,T]$,
\begin{equation}\label{yar24}
\begin{split}
&\int_{\R}|\dot{A}_{NR}(t,\xi)|^{-2}A^4_{NR}(t,\xi)\big(\langle t\rangle^{3/2}\langle\xi\rangle^{-3/2}\big)|\widetilde{g}(t,\xi)|^2\,d\xi\lesssim_\delta \eps_1^4,\\
&\int_1^t\int_{\R}|\dot{A}_{NR}(s,\xi)|^{-1}A^3_{NR}(s,\xi)\big(\langle s\rangle^{3/2}\langle\xi\rangle^{-3/2}\big)|\widetilde{g}(s,\xi)|^2\,d\xi ds\lesssim_\delta \eps_1^4.
\end{split}
\end{equation}
\end{lemma}

\begin{proof} See \cite[Lemma 4.3]{IOJI2} for (i) and (ii), and \cite[Lemma 4.6]{IOJI2} for the proof of (iii).
\end{proof}

\subsection{Green's functions and elliptic estimates} Assume that $\varphi',f':\mathbb{T}\times[0,1]\to\mathbb{C}$ are $C^2$ functions satisfying
\begin{equation}\label{elli9}
(\partial_x^2+\partial_y^2)\varphi'=f',\qquad \varphi'(x,0)=\varphi'(x,1)=0.
\end{equation}
Then $\varphi'$ can be determined explicitly through an integral operator. Indeed, we can write
\begin{equation}\label{elli10}
\varphi'_k(y)=-\int_0^1f'_k(y')G_k(y,y')\,dy',
\end{equation}
where
\begin{equation}\label{elli11}
G_k(y,z):=\frac{1}{k\sinh k}
\begin{cases}
\sinh(k(1-z))\sinh(ky)&\qquad\text{ if }y\leq z,\\
\sinh(kz)\sinh(k(1-y))&\qquad\text{ if }y\geq z,
\end{cases}
\qquad k\in\mathbb{Z}\setminus\{0\},
\end{equation}
\begin{equation}\label{elli11.2}
G_0(y,z):=
\begin{cases}
(1-z)y&\qquad\text{ if }y\leq z,\\
z(1-y)&\qquad\text{ if }y\geq z,
\end{cases}
\end{equation}
is the Green's function associated to the equation \eqref{elli9}, and
\begin{equation}\label{elli11.5}
\varphi'_k(y):=\frac{1}{2\pi}\int_{\mathbb{T}}\varphi'(x,y)e^{-ikx}\,dx,\qquad f'_k(y):=\frac{1}{2\pi}\int_{\mathbb{T}}f'(x,y)e^{-ikx}\,dx,
\end{equation}
denote the $k$-th Fourier coefficient of the functions $\varphi'$ and $f'$.

We prove now an important lemma concerning elliptic estimates adapted to our situation. 

\begin{lemma}\label{lm:elli3}
Assume that $f\in C([0,T]:H^4(\mathbb{T}\times[b(0),b(1)]))$ is supported in $\T\times[b(\vartheta_0),b(1-\vartheta_0)]$. Then there is a unique solution $\varphi\in C([0,T]:H^4(\mathbb{T}\times[b(0),b(1)]))$ of the problem
\begin{equation}\label{elli4}
\partial_z^2\varphi+(B'_0)^2(\partial_v-t\partial_z)^2\varphi+B''_0(\partial_v-t\partial_z)\varphi=f(t,z,v),
\end{equation}
with Dirichlet boundary conditions $\varphi(t,z,b(0))=\varphi(t,z,b(1))=0$. Moreover, if $t_1,t_2\in[0,T]$ then, recalling the definitions \eqref{rew0}--\eqref{rew0.4},
\begin{equation}\label{elli6}
\begin{split}
\big\|P_{\neq 0}[\partial_z^2+(\partial_v-t\partial_z)^2](\Psi\varphi)\big\|_{W[t_1,t_2]}\lesssim_\delta \|f\|_{W[t_1,t_2]},\\
\big\|P_{\neq 0}[\partial_z^2+(\partial_v-t\partial_z)^2](\Psi\varphi)\big\|_{\widetilde{W}[t_1,t_2]}\lesssim_\delta\|f\|_{\widetilde{W}[t_1,t_2]}.
\end{split}
\end{equation}
\end{lemma}

\begin{proof} We reverse the change of variables \eqref{changeofvariables1}, so we define
\begin{equation}\label{elli8}
\varphi'(t,x,y):=\varphi(t,x-tb(y),b(y)),\qquad f'(t,x,y):=f(t,x-tb(y),b(y)).
\end{equation}
The functions $\varphi',f'$ satisfy the equation \eqref{elli9} for any $t\in[0,T]$, therefore, using \eqref{elli8},
\begin{equation*}
\varphi_k(t,b(y))=-\int_0^1f_k(t,b(y'))G_k(y,y')e^{ikt((b(y)-b(y'))}\,dy'.
\end{equation*}
Thus, letting $\mathcal{G}_k(b(y),b(y')):=G_k(y,y')$, we have
\begin{equation}\label{elli12}
\varphi_k(t,v)=-\int_{b(0)}^{b(1)}f_k(t,w)\mathcal{G}_k(v,w)e^{ikt(v-w)}(1/B'_0(w))\,dw,
\end{equation}
Recall that $f(t)$ is supported in $\T\times[b(\vartheta_0),b(1-\vartheta_0)]$. We  multiply \eqref{elli12} by $\Psi(v)\Psi(w)$, and take the Fourier transform in $v$ and $w$. Thus
\begin{equation}\label{elli13}
\begin{split}
\widetilde{\Psi\varphi}(t,k,\xi)&=C\int_{\mathbb{R}}\widetilde{f}(t,k,\eta)K(\xi-kt,kt-\eta)\,d\eta,\\
K(\mu,\nu):&=\int_{\mathbb{R}^2}\Psi(v)\Psi(w)\mathcal{G}_k(v,w)(1/B'_0(w))e^{-iv\mu}e^{-iw\nu}\,dvdw.
\end{split}
\end{equation}

The kernel $K$ satisfies the bounds, for $k\neq 0$,
\begin{equation}\label{elli15}
|K(\mu,\nu)|\lesssim \frac{e^{-4\delta_0\langle\mu+\nu\rangle^{1/2}}}{k^2+|\mu|^2}.
\end{equation}
This is proved in \cite[Lemma A3]{JiaG}, using the explicit formula \eqref{elli11}. Therefore, using \eqref{elli13}, 
\begin{equation}\label{elli16}
(k^2+|\xi-kt|^2)|\widetilde{\Psi\varphi}(t,k,\xi)|\lesssim \int_{\mathbb{R}}|\widetilde{f}(t,k,\eta)|e^{-4\delta_0\langle\xi-\eta\rangle^{1/2}}\,d\eta,
\end{equation}
for $k\neq 0$. It follows from \eqref{TLX7} that, for any $\xi,\eta\in\mathbb{R}$, $k\in\mathbb{Z}\setminus \{0\}$, and $t\geq 0$,
\begin{equation}\label{elli18}
A_k(t,\xi)\lesssim_{\delta}A_k(t,\eta)e^{2\delta_0\langle \xi-\eta\rangle^{1/2}}.
\end{equation}
The inequalities in \eqref{elli6} follow from \eqref{elli16}--\eqref{elli18} and the definitions \eqref{rew0.2}--\eqref{rew0.4}, using also \eqref{eq:CDW} (for the space-time bound).
\end{proof}

\section{Improved control on the coordinate functions $V'_{\ast}, B'_{\ast}, B''_{\ast}, \mathcal{H}$}\label{SecCoor}

In this subsection we prove the following bounds:

\begin{proposition}\label{Coor0.1}
With the definitions and assumptions in Proposition \ref{MainBootstrap}, we have
\begin{equation}\label{Coor0.2}
\sum_{g\in\{V'_{\ast},B'_{\ast},B''_{\ast},\mathcal{H}\}}\left[\E_{g}(t)+\mathcal{B}_g(t)\right]\leq \epsilon_1^2/2\qquad {\rm for\,\,any\,\,}t\in[1,T].
\end{equation}
\end{proposition}

The rest of the subsection is concerned with the proof of this proposition. The arguments are similar to the arguments in \cite[Section 6]{IOJI}, and we will be somewhat brief.

Using the definitions \eqref{rec2}--\eqref{rec3} we calculate
\begin{equation*}
\begin{split}
\frac{d}{dt}\sum_{g\in\{\mathcal{H},V'_{\ast},B'_{\ast},B''_{\ast}\}}\mathcal{E}_{g}(t)&=2\mathcal{K}^{2}\int_{\R}\dot{A}_{NR}(t,\xi)A_{NR}(t,\xi)\big(\langle t\rangle^{3/2}\langle\xi\rangle^{-3/2}\big)\big|\widetilde{\mathcal{H}}(t,\xi)\big|^2\,d\xi\\
&+\mathcal{K}^{2}2\Re\int_{\R}A^2_{NR}(t,\xi)\big(\langle t\rangle^{3/2}\langle\xi\rangle^{-3/2}\big)\partial_t\widetilde{\mathcal{H}}(t,\xi)\overline{\widetilde{\mathcal{H}}(t,\xi)}\,d\xi\\
&+\mathcal{K}^{2}\int_{\R}A^2_{NR}(t,\xi)\frac{3}{2}\big(t\langle t\rangle^{-1/2}\langle\xi\rangle^{-3/2}\big)\big|\widetilde{\mathcal{H}}(t,\xi)\big|^2\,d\xi\\
&+2\sum_{U\in\{V'_\ast,B'_{\ast},B''_{\ast}\}}\int_\R \dot{A}_R(t,\xi)A_R(t,\xi)\big|\widetilde{U}(t,\xi)\big|^2\,d\xi\\
&+2\sum_{U\in\{V'_\ast,B'_{\ast},B''_{\ast}\}}\Re\int_\R A^2_R(t,\xi)\partial_t\widetilde{U}(t,\xi)\overline{\widetilde{U}(t,\xi)}\,d\xi.
\end{split}
\end{equation*}
Therefore, since $\partial_tA_R\leq 0$ and $\partial_tA_{NR}\leq 0$, for any $t\in[1,T]$ we have
\begin{equation}\label{yar2}
\begin{split}
&\sum_{g\in\{\mathcal{H},V'_{\ast},B'_{\ast},B''_{\ast}\}}\left[\mathcal{E}_{g}(t) +\mathcal{B}_{g}(t)\right]\\
&=\sum_{g\in\{\mathcal{H},V'_{\ast},B'_{\ast},B''_{\ast}\}}\mathcal{E}_{g}(1)-\bigg[\sum_{g\in\{\mathcal{H},V'_{\ast},B'_{\ast},B''_{\ast}\}}\mathcal{B}_g(t)\bigg]+\mathcal{L}_1(t)+\mathcal{L}_2(t),
\end{split}
\end{equation}
where
\begin{equation}\label{yar3}
\begin{split}
\mathcal{L}_1(t):=2\Re\sum_{U\in\{V'_\ast,B'_{\ast},B''_{\ast}\}}\int_1^t\int_\R A^2_R(s,\xi)\partial_s\widetilde{U}(s,\xi)\overline{\widetilde{U}(s,\xi)}\,d\xi ds,
\end{split}
\end{equation}
\begin{equation}\label{yar4}
\begin{split}
\mathcal{L}_2(t):&=\mathcal{K}^{2}2\Re\int_1^t\int_{\R}A^2_{NR}(s,\xi)\big(\langle s\rangle^{3/2}\langle\xi\rangle^{-3/2}\big)\partial_s\widetilde{\mathcal{H}}(s,\xi)\overline{\widetilde{\mathcal{H}}(s,\xi)}\,d\xi ds\\
&+\mathcal{K}^{2}\int_1^t\int_{\R}A^2_{NR}(s,\xi)\frac{3}{2}\big(s\langle s\rangle^{-1/2}\langle\xi\rangle^{-3/2}\big)\big|\widetilde{\mathcal{H}}(s,\xi)\big|^2\,d\xi ds.
\end{split}
\end{equation}
Since $\sum_{g\in\{\mathcal{H},V'_{\ast},B'_{\ast},B''_{\ast}\}}\mathcal{E}_{g}(1)\lesssim\eps_1^3$, for \eqref{Coor0.2} it suffices to prove that, for any $t\in[1,T]$,
\begin{equation}\label{yar6}
-\bigg[\sum_{g\in\{\mathcal{H},V'_{\ast},B'_{\ast},B''_{\ast}\}}\mathcal{B}_g(t)\bigg]+\mathcal{L}_1(t)+\mathcal{L}_2(t)\leq \eps_1^2/4.
\end{equation}

To prove \eqref{yar6} we rewrite the equations \eqref{rea23.8}--\eqref{rea24} in the form
\begin{equation}\label{yar6.50}
\begin{split}
\partial_tB'_\ast&=-\dot{V}\partial_vB'_\ast-\dot{V}\partial_vB'_0,\\
\partial_tB''_\ast&=-\dot{V}\partial_vB''_\ast-\dot{V}\partial_vB''_0,\\
\partial_tV'_\ast&=-\dot{V}\partial_vV'_\ast-\dot{V}\partial_vB'_0+\mathcal{H}/t,\\
\end{split}
\end{equation}

We extract the quadratic components of $\mathcal{L}_1$ and $\mathcal{L}_2$ (corresponding to the linear terms in the right-hand sides of \eqref{yar6.50} and \eqref{rea25}, so we define 
\begin{equation}\label{yar7}
\mathcal{L}_{1,2}(t):=2\Re\int_1^t\int_\R A^2_R(s,\xi)\Big\{\Big[\frac{\widetilde{\mathcal{H}}(s,\xi)}{s}-\widetilde{\dot{V}'_1}(s,\xi)\Big]\overline{\widetilde{V'_{\ast}}(s,\xi)}-\sum_{a\in\{1,2\}}\widetilde{\dot{V}'_a}(s,\xi)\overline{\widetilde{U_a}(s,\xi)}\Big\}\,d\xi ds,
\end{equation}
where $\dot{V}'_1:=\dot{V}\partial_vB'_0, \dot{V}'_2:=\dot{V}\partial_vB''_0$, $U_1:=B'_{\ast}, U_2:=B''_{\ast}$, and
\begin{equation}\label{yar8}
\begin{split}
\mathcal{L}_{2,2}(t)&:=\mathcal{K}^{2}\int_1^t\int_{\R}A^2_{NR}(s,\xi)\Big\{-\frac{2\langle s\rangle^{3/2}}{s\langle\xi\rangle^{3/2}}|\widetilde{\mathcal{H}}(s,\xi)|^2+\frac{3s/2}{\langle s\rangle^{1/2}\langle\xi\rangle^{3/2}}\big|\widetilde{\mathcal{H}}(s,\xi)\big|^2\Big\}\,d\xi ds\\
&=-\mathcal{K}^{2}\int_1^t\int_{\R}A^2_{NR}(s,\xi)\frac{2+s^2/2}{s\langle\xi\rangle^{3/2}\langle s\rangle^{1/2}}|\widetilde{\mathcal{H}}(s,\xi)|^2\,d\xi ds.
\end{split}
\end{equation}
We examine the identities \eqref{yar6.50} and \eqref{rea25} and let
\begin{equation}\label{yar18.5}
\begin{split}
& f_1:=-\dot{V}\partial_vB'_{\ast},\qquad f_2:=-\dot{V}\partial_vB''_{\ast},\qquad f_3:=-\dot{V}\partial_vV'_{\ast},\\
&  g_1:=-\dot{V}\partial_v\mathcal{H},\qquad g_2:=V'[\langle\partial_z\phi\,\partial_vF\rangle-\langle\partial_vP_{\neq0}\phi\,\partial_zF\rangle].
\end{split}
\end{equation}
Notice that
\begin{equation}\label{yar19}
\begin{split}
&\mathcal{L}_1(t)=\mathcal{L}_{1,2}(t)+2\Re\int_1^t\int_\R A^2_R(s,\xi)\bigg\{\sum_{a\in\{1,2\}}\widetilde{f_a}(s,\xi)\overline{\widetilde{U_a}(s,\xi)}+\widetilde{f_3}(s,\xi)\overline{\widetilde{V'_{\ast}}(s,\xi)}\bigg\}\,d\xi ds,\\
&\mathcal{L}_2(t)=\mathcal{L}_{2,2}(t)+\sum_{a\in\{1,2\}}\mathcal{K}^{2}2\Re\int_1^t\int_{\R}A^2_{NR}(s,\xi)\big(\langle s\rangle^{3/2}\langle\xi\rangle^{-3/2}\big)\widetilde{g_a}(s,\xi)\overline{\widetilde{\mathcal{H}}(s,\xi)}\,d\xi ds.
\end{split}
\end{equation}
The desired bounds \eqref{yar6} follow from Lemmas \ref{yar10} and \ref{yar20} below.

\begin{lemma}\label{yar10}
For any $t\in [1,T]$ we have
\begin{equation}\label{yar11}
-\bigg[\sum_{g\in\{\mathcal{H},V'_{\ast},B'_{\ast},B''_{\ast}\}}\mathcal{B}_g(t)\bigg]+\mathcal{L}_{1,2}(t)+\mathcal{L}_{2,2}(t)\leq \eps_1^2/8.
\end{equation}
\end{lemma}

\begin{proof} Since $\mathcal{L}_{2,2}(t)\leq 0$, it suffices to prove that, for any $t\in[1,T]$,
\begin{equation}\label{yar11.4}
\mathcal{L}_{1,2}(t)\leq \bigg[\sum_{g\in\{\mathcal{H},V'_{\ast},B'_{\ast},B''_{\ast}\}}\mathcal{B}_g(t)\bigg]+\eps_1^2/8.
\end{equation}
Using Cauchy-Schwarz and the definitions, we have
\begin{equation*}
\begin{split}
\mathcal{L}_{1,2}(t)&\leq \frac{1}{2}\mathcal{B}_{V'_{\ast}}(t)+32\int_1^t\int_{\R}\frac{A^3_R(s,\xi)}{|\dot{A}_R(s,\xi)|}\frac{|\widetilde{\mathcal{H}}(s,\xi)|^2}{s^2}\,d\xi ds+32\int_1^t\int_{\R}\frac{A^3_R(s,\xi)}{|\dot{A}_R(s,\xi)|}|\dot{V}'_1(s,\xi)|^2\,d\xi ds\\
&+\frac{1}{2}\mathcal{B}_{B'_{\ast}}(t)+8\int_1^t\int_{\R}\frac{A^3_R(s,\xi)}{|\dot{A}_R(s,\xi)|}|\dot{V}'_1(s,\xi)|^2\,d\xi ds\\
&+\frac{1}{2}\mathcal{B}_{B''_{\ast}}(t)+8\int_1^t\int_{\R}\frac{A^3_R(s,\xi)}{|\dot{A}_R(s,\xi)|}|\dot{V}'_2(s,\xi)|^2\,d\xi ds
\end{split}
\end{equation*}

The functions $\dot{V}'_a, a\in\{1,2\}$ satisfy the bounds \eqref{nar7}. Notice also that for any $C_\delta\geq 1$ there is $\mathcal{K}(\delta)$ large enough such that
\begin{equation*}
\frac{A^3_R(s,\xi)}{s^2|\dot{A}_R(s,\xi)|}\leq A_{NR}(s,\xi)|\dot{A}_{NR}(s,\xi)|(C_\delta^{-1}+\mathcal{K}(\delta)^{2}\langle s\rangle^{3/2}\langle \xi\rangle^{-3/2}).
\end{equation*}
This inequality is proved in \cite[Lemma 6.2]{IOJI}. The desired bounds \eqref{yar11.4} follow by letting $\mathcal{K}$ large enough, using also the estimates \eqref{nar7}.
\end{proof}

We prove now estimates on the cubic terms. 

\begin{lemma}\label{yar20} 
For any $t\in[1,T]$ and $a\in\{1,2\}$ we have
\begin{equation}\label{yar21.7}
\Big|2\Re\int_1^t\int_\R A^2_R(s,\xi)\widetilde{f_a}(s,\xi)\overline{\widetilde{U_a}(s,\xi)}\,d\xi ds\Big|\lesssim_\delta\eps_1^3,
\end{equation}
\begin{equation}\label{yar21}
\Big|2\Re\int_1^t\int_\R A^2_R(s,\xi)\widetilde{f_3}(s,\xi)\overline{\widetilde{V'_{\ast}}(s,\xi)}\,d\xi ds\Big|\lesssim_\delta\eps_1^3,
\end{equation}
and
\begin{equation}\label{yar22}
\Big|2\Re\int_1^t\int_{\R}A^2_{NR}(s,\xi)\big(\langle s\rangle^{3/2}\langle\xi\rangle^{-3/2}\big)\widetilde{g_a}(s,\xi)\overline{\widetilde{\mathcal{H}}(s,\xi)}\,d\xi ds\Big|\lesssim_\delta\eps_1^3.
\end{equation}
\end{lemma} 

\begin{proof} {\bf{Step 1.}} We start with \eqref{yar21.7}-\eqref{yar21}. The two bounds are similar, so we only provide all the details for the estimate \eqref{yar21}. See also \cite[Lemma 6.5]{IOJI} for a similar argument.

We write the left-hand side of \eqref{yar21} in the form
\begin{equation*}
\begin{split}
&C\Big|2\Re\int_1^t\int_\R\int_\R A^2_R(s,\xi)\widetilde{\dot{V}}(s,\xi-\eta)(i\eta)\widetilde{V'_{\ast}}(s,\eta)\overline{\widetilde{V'_{\ast}}(s,\xi)}\,d\xi d\eta ds\Big|\\
&=C\Big|\int_1^t\int_\R\int_\R [\eta A^2_R(s,\xi)-\xi A_R^2(s,\eta)]\widetilde{\dot{V}}(s,\xi-\eta)\widetilde{V'_{\ast}}(s,\eta)\overline{\widetilde{V'_{\ast}}(s,\xi)}\,d\xi d\eta ds\Big|,
\end{split}
\end{equation*}
using symmetrization and the fact that $\dot{V}$ is real-valued. We define the sets
\begin{equation}\label{tol4}
\begin{split}
&S_0:=\Big\{(\xi,\eta)\in\R^2:\,\min(\langle\xi\rangle,\,\langle\eta\rangle,\,\langle\xi-\eta\rangle)\geq \frac{\langle\xi\rangle+\langle\eta\rangle+\langle\xi-\eta\rangle}{20}\Big\},\\
&S_1:=\Big\{(\xi,\eta)\in\R^2:\,\langle\xi-\eta\rangle\leq \frac{\langle\xi\rangle+\langle\eta\rangle+\langle\xi-\eta\rangle}{10}\Big\},\\
&S_2:=\Big\{(\xi,\eta)\in\R^2:\,\langle\eta\rangle\leq \frac{\langle\xi\rangle+\langle\eta\rangle+\langle\xi-\eta\rangle}{10}\Big\},\\
&S_3:=\Big\{(\xi,\eta)\in\R^2:\,\langle\xi\rangle\leq \frac{\langle\xi\rangle+\langle\eta\rangle+\langle\xi-\eta\rangle}{10}\Big\}.
\end{split}
\end{equation}
and the corresponding integrals
\begin{equation}\label{tol5}
\begin{split}
\mathcal{I}_n:=\int_1^t\int_\R\int_\R \mathbf{1}_{S_n}(\xi,\eta)&|\eta A^2_R(s,\xi)-\xi A_R^2(s,\eta)|\,|\widetilde{\dot{V}}(s,\xi-\eta)||\widetilde{V'_{\ast}}(s,\eta)|\,|\widetilde{V'_{\ast}}(s,\xi)|\,d\xi d\eta ds.
\end{split}
\end{equation}
For \eqref{yar21} it suffices to prove that
\begin{equation}\label{tol6}
\mathcal{I}_n\lesssim_\delta \eps_1^3\qquad\text{ for }n\in\{0,1,2,3\}.
\end{equation}

We use the following bilinear estimates for the weights, proved in \cite[Lemma 8.9]{IOJI}. Letting $\delta'_0=\delta_0/200$, we have:

$\bullet\,\,$ If $(\xi,\eta)\in S_0\cup S_1$, $\rho=\xi-\eta$, $s\geq 1$, $\alpha\in[0,4]$, and $Y\in\{NR,R\}$ then
\begin{equation}\label{tol7}
\begin{split}
|\eta A^2_Y(s,\xi)&\langle\xi\rangle^{-\alpha}-\xi A_Y^2(s,\eta)\langle\eta\rangle^{-\alpha}|\\
&\lesssim_\delta s^{1.6}\frac{\sqrt{|(A_Y\dot{A}_Y)(s,\xi)|}}{\langle\xi\rangle^{\alpha/2}}\frac{\sqrt{|(A_Y\dot{A}_Y)(s,\eta)|}}{\langle\eta\rangle^{\alpha/2}}\cdot A_{NR}(s,\rho)e^{-\delta'_0\langle\rho\rangle^{1/2}}.
\end{split}
\end{equation}

$\bullet\,\,$ If $(\xi,\eta)\in S_2$, $\rho=\xi-\eta$, and $s\geq 1$ then
\begin{equation}\label{TLY2.2}
\langle\eta\rangle A^2_{R}(s,\xi)\lesssim_\delta s^{1.1}\langle\xi\rangle^{0.6}\sqrt{|(A_R\dot{A}_R)(s,\xi)|}\sqrt{|(A_{NR}\dot{A}_{NR})(s,\rho)|}\cdot A_{R}(s,\eta)e^{-\delta'_0\langle\eta\rangle^{1/2}}
\end{equation}
and
\begin{equation}\label{TLY2.3}
\langle\eta\rangle A^2_{NR}(s,\xi)\lesssim_\delta s^{1.1}\langle\xi\rangle^{-0.4}\sqrt{|(A_{NR}\dot{A}_{NR})(s,\xi)|}\sqrt{|(A_{NR}\dot{A}_{NR})(s,\rho)|}\cdot A_{NR}(s,\eta)e^{-\delta'_0\langle\eta\rangle^{1/2}}.
\end{equation}

For $n\in\{0,1\}$ we can now estimate, using \eqref{tol7},
\begin{equation*}
\begin{split}
\mathcal{I}_n\lesssim_\delta \Big\|\sqrt{|(A_R\dot{A}_R)(s,\xi)|}&\widetilde{V'_{\ast}}(s,\xi)\Big\|_{L^2_sL^2_\xi}\Big\|\sqrt{|(A_R\dot{A}_R)(s,\eta)|}\widetilde{V'_{\ast}}(s,\eta)\Big\|_{L^2_sL^2_\eta}\\
&\times\Big\|s^{1.6}A_{NR}(s,\rho)\langle\rho\rangle^2 e^{-\delta'_0\langle\rho\rangle^{1/2}}\widetilde{\dot{V}}(s,\rho)\Big\|_{L^\infty_sL^2_\rho},
\end{split}
\end{equation*}
and the bounds \eqref{tol6} follow for $n\in\{0,1\}$ from \eqref{boot2} and \eqref{nar7}. Similarly, for $n=2$ we use \eqref{TLY2.2} and \eqref{eq:comparisonweights1} to estimate
\begin{equation*}
\begin{split}
\mathcal{I}_2\lesssim_\delta \Big\|\sqrt{|(A_R\dot{A}_R)(s,\xi)|}&\widetilde{V'_{\ast}}(s,\xi)\Big\|_{L^2_sL^2_\xi}\Big\|s^{1.1}\langle\rho\rangle^{0.6}\sqrt{|(A_{NR}\dot{A}_{NR})(s,\rho)|} \widetilde{\dot{V}}(s,\rho)\Big\|_{L^2_sL^2_\rho}\\
&\times\Big\|A_R(s,\eta)\langle\eta\rangle e^{-\delta'_0\langle\eta\rangle^{1/2}}\widetilde{V'_{\ast}}(s,\eta)\Big\|_{L^\infty_sL^2_\eta},
\end{split}
\end{equation*}
and the desired bounds follow from \eqref{boot2} and \eqref{nar7}. The case $n=3$ is similar, by changes of variables, which completes the proof of \eqref{yar21}.

{\bf{Step 2.}} The bounds \eqref{yar22} for $a=1$ are similar, using symmetrization, the bounds \eqref{tol7} with $Y=NR$, and the bounds \eqref{TLY2.3}. See also \cite[Lemma 6.6]{IOJI} for a similar argument. Finally, the bounds \eqref{yar22} for $a=2$ follow from \eqref{yar24}, \eqref{boot2}, and the Cauchy inequality (see also \cite[Lemma 6.4]{IOJI} for a similar proof).
\end{proof}

\section{Improved control on the auxiliary variables $\Theta^\ast$ and $F^{\ast}$}

In this section we prove the main bootstrap bounds \eqref{boot3} for the functions $\Theta^{\ast}$ and $F^\ast$.

\begin{proposition}\label{ImprovBoots1}
With the definitions and assumptions in Proposition \ref{MainBootstrap}, we have
\begin{equation}\label{elli1}
\E_{\Theta^{\ast}}(t)+\mathcal{B}_{\Theta^{\ast}}(t)\lesssim_{\delta}\epsilon_1^4\qquad\text{ for any }\,\,t\in[1,T].
\end{equation}
\end{proposition}

\begin{proof} We use the equations \eqref{rea26} and \eqref{Ph1}, thus 
\begin{equation}\label{elli2}
\begin{split}
&\partial_z^2(\phi-\phi')+(B'_0)^2(\partial_v-t\partial_z)^2(\phi-\phi')+B''_0(\partial_v-t\partial_z)(\phi-\phi')=\mathcal{G}_1+\mathcal{G}_2,\\
&\mathcal{G}_1:=\left[(B'_0)^2-(V')^2\right](\partial_v-t\partial_z)^2\phi,\\
&\mathcal{G}_2:=(B''_0-V'')(\partial_v-t\partial_z)\phi.
\end{split}
\end{equation}
In view of Lemma \ref{lm:elli3}, it suffices to prove that
\begin{equation}\label{elli2.5}
\|\mathcal{G}_1\|_{\widetilde{W}}+\|\mathcal{G}_2\|_{\widetilde{W}}\lesssim_\delta \eps_1^2.
\end{equation}

Since $V'_\ast$ is supported in $[b(\va_0),b(1-\va_0)]$, we can write
\begin{equation*}
\mathcal{G}_1=-V'_\ast\cdot \Psi(B'_0+V')\cdot (\partial_v-t\partial_z)^2(\Psi\phi),
\end{equation*}
where $\Psi$ is the Gevrey cut-off function in \eqref{rec0}. Using Lemma \ref{nar8} (i), (ii), and the bootstrap assumptions \eqref{boot2} for $V'_\ast$ and $\Theta$, we can estimate
\begin{equation*}
\|\mathcal{G}_1\|_{\widetilde{W}}\lesssim_\delta \|V'_\ast\|_R\|\Psi(B'_0+V')\|_R\|(\partial_v-t\partial_z)^2(\Psi\phi)\|_{\widetilde{W}}\lesssim_\delta \eps_1^2,
\end{equation*}
as claimed in \eqref{elli2.5}. 

Similarly, since $V''=V'\partial_vV'$ and $B''_0=B'_0\partial_vB'_0$, we can write
\begin{equation}\label{elli2.4}
\mathcal{G}_2=-(1/2)\partial_v[V'_\ast\cdot \Psi(B'_0+V')]\cdot (\partial_v-t\partial_z)(\Psi\phi).
\end{equation}
Moreover
\begin{equation}\label{elli2.6}
\frac{|k|\langle t\rangle}{|\xi|+|k|\langle t\rangle}\lesssim_\delta \frac{\langle\eta-tk\rangle}{\langle \xi-\eta\rangle}\frac{|k|\langle t\rangle}{|\eta|+|k|\langle t\rangle}e^{\delta\min(\langle\xi-\eta\rangle,\langle k,\eta\rangle)^{1/2}},
\end{equation}
if $k\in\mathbb{Z}^\ast$, $t\geq 1$, and $\xi,\eta\in\mathbb{R}$, as one can check easily by considering the cases $|\xi-\eta|\leq 10 |k,\eta|$ and $|\xi-\eta|\geq 10 |k,\eta|$. Therefore, using also \eqref{TLX7}--\eqref{vfc30.7},
\begin{equation*}
\frac{A_k(t,\xi)|k|\langle t\rangle}{|\xi|+|k|\langle t\rangle}\lesssim_\delta \frac{A_R(t,\xi-\eta)}{\langle\xi-\eta\rangle}\frac{A_k(t,\eta)|k|\langle t\rangle}{|\eta|+|k|\langle t\rangle}\langle\eta-tk\rangle e^{-(\delta_0/30)\min(\langle\xi-\eta\rangle,\langle k,\eta\rangle)^{1/2}}
\end{equation*}
and
\begin{equation*}
\begin{split}
\frac{|(A_k\dot{A}_k)(t,\xi)|^{1/2}|k|\langle t\rangle}{|\xi|+|k|\langle t\rangle}&\lesssim_\delta e^{-(\delta_0/30)\min(\langle\xi-\eta\rangle,\langle k,\eta\rangle)^{1/2}}\Big\{\frac{|(A_R\dot{A}_R)(t,\xi-\eta)|^{1/2}}{\langle\xi-\eta\rangle}\\
\times\frac{A_k(t,\eta)|k|\langle t\rangle}{|\eta|+|k|\langle t\rangle}&\langle\eta-tk\rangle+\frac{A_R(t,\xi-\eta)}{\langle\xi-\eta\rangle}\frac{|(A_k\dot{A}_k)(t,\eta)|^{1/2}|k|\langle t\rangle}{|\eta|+|k|\langle t\rangle}\langle\eta-tk\rangle\Big\}.
\end{split}
\end{equation*}

We examine the formula \eqref{elli2.4} and notice that $\|V'_\ast\cdot \Psi(B'_0+V')\|_R\lesssim_\delta\eps_1$ (due to Lemma \ref{nar8} (i), (ii)) and $\|\langle\partial_v-t\partial_z\rangle(\partial_v-t\partial_z)(\Psi\phi)\|_{\widetilde{W}}\lesssim_\delta\eps_1$ (due to the bootstrap assumption \eqref{boot2}). The desired conclusion $\|\mathcal{G}_2\|_{\widetilde{W}}\lesssim_\delta \eps_1^2$ in \eqref{elli2.5} follows using Lemma \ref{Multi0} (ii) and the two weighted estimates above.
\end{proof}

We prove now bootstrap bounds on the function $F^\ast$. 

\begin{proposition}\label{ImprovBoots2}
With the definitions and assumptions in Proposition \ref{MainBootstrap}, we have
\begin{equation}\label{lkj1}
\E_{F^{\ast}}(t)+\mathcal{B}_{F^{\ast}}(t)\lesssim_{\delta}\epsilon_1^3\qquad\text{ for any }\,\,t\in[1,T].
\end{equation}
\end{proposition}

\begin{proof} The function $F^\ast$ satisfies the evolution equation
\begin{equation}\label{lkj2}
\partial_tF^\ast=V'\partial_vP_{\neq 0}\phi\,\partial_zF-(\dot{V}+V'\partial_z\phi)\,\partial_vF+(B''\partial_z\phi-B''_0\partial_z\phi'),
\end{equation}
which follows from \eqref{rea23.1} and \eqref{reb14}. Recalling the definition \eqref{rec1}, we calculate
\begin{equation}\label{lkj3}
\begin{split}
\frac{d}{dt}\E_{F^\ast}(t)=&\sum_{k\in \mathbb{Z}}\int_\R 2\dot{A}_k(t,\xi)A_k(t,\xi)\big|\widetilde{F^\ast}(t,k,\xi)\big|^2\,d\xi\\
&+2\Re\sum_{k\in \mathbb{Z}}\int_{\R}A_k^2(t,\xi)\partial_t\widetilde{F^\ast}(t,k,\xi)\overline{\widetilde{F^\ast}(t,k,\xi)}\,d\xi.
\end{split}
\end{equation}
Therefore, since $\partial_tA_k\leq 0$, for any $t\in[1,T]$ we have
\begin{equation*}
\begin{split}
&\E_f(t)+\int_1^t\sum_{k\in \mathbb{Z}}\int_\R 2|\dot{A}_k(s,\xi)|A_k(s,\xi)\big|\widetilde{F^\ast}(s,k,\xi)\big|^2\,d\xi ds\\
&=\E_f(1)+\int_1^t\Big\{2\Re\sum_{k\in \mathbb{Z}}\int_{\R}A_k^2(s,\xi)\partial_s\widetilde{F^\ast}(s,k,\xi)\overline{\widetilde{F^\ast}(s,k,\xi)}\,d\xi\Big\}ds.
\end{split}
\end{equation*}

We examine the equation \eqref{lkj2} and decompose the nonlinearity in the right-hand side. Let
\begin{equation}\label{lkj5}
\begin{split}
&\mathcal{N}_1:=V'\partial_vP_{\neq 0}\phi\,\partial_zF^\ast,\qquad \mathcal{N}_2:=-V'\partial_z\phi\partial_vF^\ast,\qquad\mathcal{N}_3:=-\dot{V}\partial_vF^\ast,\\
&\mathcal{N}_4:=V'\partial_vP_{\neq 0}\phi\,\partial_z(F-F^\ast),\qquad \mathcal{N}_5:=-V'\partial_z\phi\partial_v(F-F^\ast),\\
&\mathcal{N}_6:=-\dot{V}\partial_v(F-F^\ast),\qquad \mathcal{N}_7:=B''\partial_z\phi-B''_0\partial_z\phi'.
\end{split}
\end{equation}

Since $\E_f(1)\lesssim\eps_1^3$ (see \eqref{boot1}), for \eqref{lkj1} it suffices to prove that, for any $t\in[1,T]$,
\begin{equation}\label{lkj6}
\Big|2\Re\int_1^t\sum_{k\in \mathbb{Z}}\int_{\R}A_k^2(s,\xi)\widetilde{\mathcal{N}_a}(s,k,\xi)\overline{\widetilde{F^\ast}(s,k,\xi)}\,d\xi ds\Big|\lesssim_\delta \eps_1^3,
\end{equation}
for $a\in\{1,\ldots,7\}$. We prove these bounds in Lemmas \ref{LemN1}--\ref{LemN3} below. 
\end{proof}

\begin{lemma}\label{LemN1}
The bounds \eqref{lkj6} hold for $a\in\{1,2,3\}$.
\end{lemma}

\begin{proof} This is similar to the proofs of \cite[Lemmas 4.4, 4.6, and 4.8]{IOJI}, and we will be somewhat brief. The common point is that one can symmetrize the integrals to avoid loss of derivatives.

{\bf{Step 1.}} We consider first the nonlinearity $\mathcal{N}_1$. Letting $H_1:=V'\partial_vP_{\neq 0}(\Psi\phi)$, we write
\begin{equation*}
\begin{split}
&\Big|2\Re\int_1^t\sum_{k\in \mathbb{Z}}\int_{\R}A_k^2(s,\xi)\widetilde{\mathcal{N}_1}(s,k,\xi)\overline{\widetilde{F^\ast}(s,k,\xi)}\,d\xi ds\Big|\\
&=C\Big|2\Re\Big\{\sum_{k,\ell\in \mathbb{Z}}\int_1^t\int_{\R^2}A_k^2(s,\xi)\widetilde{H_1}(s,k-\ell,\xi-\eta)i\ell\widetilde{F^\ast}(s,\ell,\eta)\overline{\widetilde{F^\ast}(s,k,\xi)}\,d\xi d\eta ds\Big\}\Big|\\
&=C\Big|\int_1^t\sum_{k,\ell\in \mathbb{Z}}\int_{\R^2}\big[\ell A_k^2(s,\xi)-k A_\ell^2(s,\eta)\big]\widetilde{H_1}(s,k-\ell,\xi-\eta)\widetilde{F^\ast}(s,\ell,\eta)\overline{\widetilde{F^\ast}(s,k,\xi)}\,d\xi d\eta ds\Big|,
\end{split}
\end{equation*}
where the second identity uses symmetrization based on the fact that $H_1$ is real-valued. The cutoff function $\Psi$ can be inserted in the definition of $H_1$ because $F$ and $F^\ast$ are supported in $[0,T]\times\mathbb{T}\times[b(\va_0),b(1-\va_0)]$. With $R_n$ as in \eqref{nar18.1}--\eqref{nar18.4}, we define the integrals
\begin{equation}\label{nar19}
\begin{split}
\mathcal{U}^n_1:=\int_1^t\sum_{k,\ell\in \mathbb{Z}}\int_{\R^2}&\mathbf{1}_{R_n}((k,\xi),(\ell,\eta))\big|\ell A_k^2(s,\xi)-k A_\ell^2(s,\eta)\big|\,|\widetilde{H_1}(s,k-\ell,\xi-\eta)|\\
&\times|\widetilde{F^\ast}(s,\ell,\eta)|\,|\widetilde{F^\ast}(s,k,\xi)|\,d\xi d\eta ds.
\end{split}
\end{equation}
We use Lemma \ref{TLXH1} and remark that $\widetilde{H_1}(t,0,\rho)=0$ for $\rho\in\mathbb{R}$. Denote $(m,\rho)=(k-\ell,\xi-\eta)$. Using \eqref{TLXH1.1}, \eqref{nar14}, and \eqref{boot2}, for $n\in\{0,1\}$ we can bound
\begin{equation*}
\begin{split}
\mathcal{U}_1^n&\lesssim_{\delta}\int_1^t\sum_{k,\ell\in \mathbb{Z}}\int_{\R^2}\sqrt{|(A_k\dot{A}_k)(s,\xi)|}\,\big|\widetilde{F^\ast}(s,k,\xi)\big|\sqrt{|(A_{\ell}\dot{A}_{\ell})(s,\eta)|}\,\big|\widetilde{F^\ast}(s,\ell,\eta)\big|\\
&\qquad\times\, \mathbf{1}_{\mathbb{Z}^\ast}(m)\frac{\langle s\rangle\langle s-\rho/m\rangle^2m^2}{(|\rho/m|+\langle s\rangle)\langle \rho\rangle}A_{m}(s,\rho)\big|\widetilde{H_1}(s,m,\rho)\big|e^{-\delta'_0\langle m,\rho\rangle^{1/2}}\,d\xi d\eta ds\\
&\lesssim_{\delta} \Big\|\sqrt{|(A_k\dot{A}_k)(s,\xi)|}\,\widetilde{F^\ast}(s,k,\xi)\Big\|_{L^2_{s}L^2_{k,\xi}}\Big\|\sqrt{|(A_{\ell}\dot{A}_{\ell})(s,\eta)|}\,\widetilde{F^\ast}(s,\ell,\eta)\Big\|_{L^2_sL^2_{\ell,\eta}}\\
&\qquad\times \Big\|\mathbf{1}_{\mathbb{Z}^\ast}(m)A_{m}(s,\rho)\frac{\langle s\rangle\langle s-\rho/m\rangle^2m^2}{(|\rho/m|+\langle s\rangle)\langle \rho\rangle}e^{-(\delta'_0/2)\langle m,\rho\rangle^{1/2}}\widetilde{H_1}(s,m,\rho)\Big\|_{L^{\infty}_sL^2_{m,\rho}}\\
&\lesssim_{\delta}\epsilon_1^3.
\end{split}
\end{equation*}
Similarly, for $n=2$ we use \eqref{TLXH1.2}, \eqref{nar14}, and \eqref{boot2} to bound
\begin{equation}\label{lkj7.1}
\begin{split}
\mathcal{U}_1^2&\lesssim_{\delta}\int_1^t\sum_{k,\ell\in \mathbb{Z}}\int_{\R^2}\mathbf{1}_{\mathbb{Z}^\ast}(m)\sqrt{|(A_{m}\dot{A}_{m})(s,\rho)|}\,\frac{\langle s\rangle\langle s-\rho/m\rangle^2m^2}{(|\rho/m|+\langle s\rangle)\langle \rho\rangle}\big|\widetilde{H_1}(s,m,\rho)\big|\\
&\qquad\times\sqrt{|(A_k\dot{A}_k)(s,\xi)|}\,\big|\widetilde{F^\ast}(s,k,\xi)\big|A_{\ell}(s,\eta) e^{-\delta'_0\langle \ell,\eta\rangle^{1/2}}|\widetilde{F^\ast}(s,\ell,\eta)|\,d\xi d\eta ds\\
&\lesssim_{\delta} \Big\|\sqrt{|(A_k\dot{A}_k)(s,\xi)|}\,\widetilde{F^\ast}(s,k,\xi)\Big\|_{L^2_{s}L^2_{k,\xi}}\Big\|A_{\ell}(s,\eta)\,e^{-(\delta'_0/2)\langle \ell,\eta\rangle^{1/2}}\widetilde{F^\ast}(s,\ell,\eta)\Big\|_{L^{\infty}_sL^2_{\ell,\eta}}\\
&\qquad\times \Big\|\mathbf{1}_{\mathbb{Z}^\ast}(m)\sqrt{|(A_{m}\dot{A}_{m})(s,\rho)|}\,\frac{\langle s\rangle\langle s-\rho/m\rangle^2m^2}{(|\rho/m|+\langle s\rangle)\langle \rho\rangle}\big|\widetilde{H_1}(s,m,\rho)\Big\|_{L^{2}_sL^2_{m,\rho}}\\
&\lesssim_{\delta}\epsilon_1^3.
\end{split}
\end{equation}
The case $n=3$ is identical to the case $n=2$, by symmetry, so $\mathcal{U}_1^n\lesssim_\delta\epsilon_1^3$ for all $n\in\{0,1,2,3\}$. The desired bounds \eqref{lkj6} follow for $a=1$.

{\bf{Step 2.}} The bounds for the nonlinearity $\mathcal{N}_2$ follow in the same way, using the estimates \eqref{TLXH3.1}--\eqref{TLXH3.2} and \eqref{nar34} (see \cite[Lemma 4.6]{IOJI} for complete details). The bounds for the nonlinearity $\mathcal{N}_3$ also follow in the same way, using the estimates \eqref{TLXH2.1}--\eqref{TLXH2.2} and \eqref{nar7} (see \cite[Lemma 4.8]{IOJI} for complete details). 
\end{proof}

\begin{lemma}\label{LemN2}
The bounds \eqref{lkj6} hold for $a\in\{4,5,6\}$.
\end{lemma}

\begin{proof} {\bf{Step 1.}} We consider first the nonlinearity $\mathcal{N}_4$, and estimate
\begin{equation*}
\begin{split}
&\Big|2\Re\int_1^t\sum_{k\in \mathbb{Z}}\int_{\R}A_k^2(s,\xi)\widetilde{\mathcal{N}_4}(s,k,\xi)\overline{\widetilde{F^\ast}(s,k,\xi)}\,d\xi ds\Big|\\
&\lesssim\Big|\sum_{k,\ell\in \mathbb{Z}}\int_1^t\int_{\R^2}A_k^2(s,\xi)\widetilde{H_1}(s,k-\ell,\xi-\eta)\ell\widetilde{(F-F^\ast)}(s,\ell,\eta)\overline{\widetilde{F^\ast}(s,k,\xi)}\,d\xi d\eta ds\Big|\lesssim \sum_{n\in\{0,1,2,3\}}\mathcal{U}_4^n,
\end{split}
\end{equation*}
where $H_1=V'\partial_vP_{\neq 0}(\Psi\phi)$ as in the proof of Lemma \ref{LemN1}, and
\begin{equation*}
\begin{split}
\mathcal{U}^n_4:=\int_1^t\sum_{k,\ell\in \mathbb{Z}}\int_{\R^2}&\mathbf{1}_{R_n}((k,\xi),(\ell,\eta))\big|\ell A_k^2(s,\xi)\big|\,|\widetilde{H_1}(s,k-\ell,\xi-\eta)|\\
&\times|\widetilde{(F-F^\ast)}(s,\ell,\eta)|\,|\widetilde{F^\ast}(s,k,\xi)|\,d\xi d\eta ds.
\end{split}
\end{equation*}
Recall that $\widetilde{H_1}(t,0,\rho)=0$ for $\rho\in\mathbb{R}$. Letting $(m,\rho)=(k-\ell,\xi-\eta)$ and using \eqref{TLD1}, \eqref{nar14}, and \eqref{boot2}, for $n\in\{0,1\}$ we can bound
\begin{equation*}
\begin{split}
&\mathcal{U}_4^n\lesssim_{\delta}\int_1^t\sum_{k,\ell\in \mathbb{Z}}\int_{\R^2}\sqrt{|(A_{\ell}\dot{A}_{\ell})(s,\eta)|}(1+\langle\ell,\eta\rangle/\langle s\rangle)^{1/2}\big|\widetilde{(F-F^\ast)}(s,\ell,\eta)\big|\cdot \sqrt{|(A_k\dot{A}_k)(s,\xi)|}\\
&\qquad\big|\widetilde{F^\ast}(s,k,\xi)\big|\cdot\mathbf{1}_{\mathbb{Z}^\ast}(m)\frac{\langle s\rangle\langle s-\rho/m\rangle^2m^2}{(|\rho/m|+\langle s\rangle)\langle \rho\rangle}A_{m}(s,\rho)\big|\widetilde{H_1}(s,m,\rho)\big|e^{-\delta'_0\langle m,\rho\rangle^{1/2}}\,d\xi d\eta ds\\
&\lesssim_{\delta} \Big\|\sqrt{|(A_k\dot{A}_k)(s,\xi)|}\,\widetilde{F^\ast}(s,k,\xi)\Big\|_{L^2_{s}L^2_{k,\xi}}\Big\|\sqrt{|(A_{\ell}\dot{A}_{\ell})(s,\eta)|(1+\langle\ell,\eta\rangle/\langle s\rangle)}\widetilde{(F-F^\ast)}(s,\ell,\eta)\Big\|_{L^2_sL^2_{\ell,\eta}}\\
&\qquad\times \Big\|\mathbf{1}_{\mathbb{Z}^\ast}(m)A_{m}(s,\rho)\frac{\langle s\rangle\langle s-\rho/m\rangle^2m^2}{(|\rho/m|+\langle s\rangle)\langle \rho\rangle}e^{-(\delta'_0/2)\langle m,\rho\rangle^{1/2}}\widetilde{H_1}(s,m,\rho)\Big\|_{L^{\infty}_sL^2_{m,\rho}}\\
&\lesssim_{\delta}\epsilon_1^3.
\end{split}
\end{equation*}
Moreover, we can also estimate $\mathcal{U}_4^2\lesssim_{\delta}\epsilon_1^3$, using \eqref{TLXH1.2}, \eqref{nar14}, and \eqref{boot2} as in \eqref{lkj7.1}. Then we can estimate $\mathcal{U}_4^3\lesssim_{\delta}\epsilon_1^3$ by symmetry. The desired bounds \eqref{lkj6} follow for $a=4$.

{\bf{Step 2.}} We consider now the nonlinearity $\mathcal{N}_5$, and estimate
\begin{equation*}
\begin{split}
&\Big|2\Re\int_1^t\sum_{k\in \mathbb{Z}}\int_{\R}A_k^2(s,\xi)\widetilde{\mathcal{N}_5}(s,k,\xi)\overline{\widetilde{F^\ast}(s,k,\xi)}\,d\xi ds\Big|\lesssim \sum_{n\in\{0,1,2,3\}}\mathcal{U}_5^n,
\end{split}
\end{equation*}
where $H_2:=V'\partial_z(\Psi\phi)$, and
\begin{equation*}
\begin{split}
\mathcal{U}^n_5:=\int_1^t\sum_{k,\ell\in \mathbb{Z}}\int_{\R^2}&\mathbf{1}_{R_n}((k,\xi),(\ell,\eta))\big|\eta A_k^2(s,\xi)\big|\,|\widetilde{H_2}(s,k-\ell,\xi-\eta)|\\
&\times|\widetilde{(F-F^\ast)}(s,\ell,\eta)|\,|\widetilde{F^\ast}(s,k,\xi)|\,d\xi d\eta ds.
\end{split}
\end{equation*}
Notice that $\widetilde{H_2}(t,0,\rho)=0$ for $\rho\in\mathbb{R}$. Letting $(m,\rho)=(k-\ell,\xi-\eta)$ and using \eqref{TLD10}, \eqref{nar34}, and \eqref{boot2}, for $n\in\{0,1\}$ we can bound, as before,
\begin{equation*}
\begin{split}
&\mathcal{U}_5^n\lesssim_{\delta}\int_1^t\sum_{k,\ell\in \mathbb{Z}}\int_{\R^2}\sqrt{|(A_{\ell}\dot{A}_{\ell})(s,\eta)|}(1+\langle\ell,\eta\rangle/\langle s\rangle)^{1/2}\big|\widetilde{(F-F^\ast)}(s,\ell,\eta)\big|\cdot \sqrt{|(A_k\dot{A}_k)(s,\xi)|}\\
&\qquad\big|\widetilde{F^\ast}(s,k,\xi)\big|\cdot\mathbf{1}_{\mathbb{Z}^\ast}(m)\frac{|m|\langle s\rangle^2\langle s-\rho/m\rangle^2}{|\rho/m|^2+\langle s\rangle^2}A_{m}(s,\rho)\big|\widetilde{H_2}(s,m,\rho)\big|e^{-\delta'_0\langle m,\rho\rangle^{1/2}}\,d\xi d\eta ds\\
&\lesssim_{\delta} \Big\|\sqrt{|(A_k\dot{A}_k)(s,\xi)|}\,\widetilde{F^\ast}(s,k,\xi)\Big\|_{L^2_{s}L^2_{k,\xi}}\Big\|\sqrt{|(A_{\ell}\dot{A}_{\ell})(s,\eta)|(1+\langle\ell,\eta\rangle/\langle s\rangle)}\widetilde{(F-F^\ast)}(s,\ell,\eta)\Big\|_{L^2_sL^2_{\ell,\eta}}\\
&\qquad\times \Big\|\mathbf{1}_{\mathbb{Z}^\ast}(m)\frac{|m|\langle s\rangle^2\langle s-\rho/m\rangle^2}{|\rho/m|^2+\langle s\rangle^2}A_{m}(s,\rho)e^{-(\delta'_0/2)\langle m,\rho\rangle^{1/2}}\widetilde{H_2}(s,m,\rho)\Big\|_{L^{\infty}_sL^2_{m,\rho}}\\
&\lesssim_{\delta}\epsilon_1^3.
\end{split}
\end{equation*}
The term $\mathcal{U}_5^2$ can be bounded in the same way, using \eqref{TLXH3.2}, \eqref{nar34}, and \eqref{boot2}, while the term $\mathcal{U}_5^3$ can be bounded by symmetry. The desired bounds \eqref{lkj6} follow for $a=5$. 

{\bf{Step 3.}} Similarly, for $a=6$ we estimate
\begin{equation*}
\begin{split}
&\Big|2\Re\int_1^t\sum_{k\in \mathbb{Z}}\int_{\R}A_k^2(s,\xi)\widetilde{\mathcal{N}_6}(s,k,\xi)\overline{\widetilde{F^\ast}(s,k,\xi)}\,d\xi ds\Big|\lesssim \sum_{n\in\{0,1,2,3\}}\mathcal{U}_6^n,
\end{split}
\end{equation*}
where $R_n^\ast:=\{((k,\xi),(\ell,\eta))\in R_n:\,k=l\}$ and
\begin{equation*}
\begin{split}
\mathcal{U}^n_6:=\int_1^t\sum_{k\in \mathbb{Z}}\int_{\R^2}&\mathbf{1}_{R^\ast_n}((k,\xi),(k,\eta))\big|\eta A_k^2(s,\xi)\big|\,|\widetilde{\dot{V}}(s,\xi-\eta)|\\
&\times|\widetilde{(F-F^\ast)}(s,k,\eta)|\,|\widetilde{F^\ast}(s,k,\xi)|\,d\xi d\eta ds.
\end{split}
\end{equation*}
Letting $\rho=\xi-\eta$ and using \eqref{TLD20}, \eqref{nar7}, and \eqref{boot2}, for $n\in\{0,1\}$ we can bound
\begin{equation*}
\begin{split}
&\mathcal{U}_6^n\lesssim_{\delta}\int_1^t\sum_{k\in \mathbb{Z}}\int_{\R^2}\sqrt{|(A_{k}\dot{A}_{k})(s,\eta)|}(1+\langle k,\eta\rangle/\langle s\rangle)^{1/2}\big|\widetilde{(F-F^\ast)}(s,\ell,\eta)\big|\cdot \sqrt{|(A_k\dot{A}_k)(s,\xi)|}\\
&\qquad\big|\widetilde{F^\ast}(s,k,\xi)\big|\cdot \big(\langle\rho\rangle\langle s\rangle+\langle \rho\rangle^{1/4}\langle s\rangle^{7/4}\big)A_{NR}(s,\rho)\big|\widetilde{\dot{V}}(s,\rho)\big|e^{-\delta'_0\langle\rho\rangle^{1/2}}\,d\xi d\eta ds\\
&\lesssim_{\delta} \Big\|\sqrt{|(A_k\dot{A}_k)(s,\xi)|}\,\widetilde{F^\ast}(s,k,\xi)\Big\|_{L^2_{s}L^2_{k,\xi}}\Big\|\sqrt{|(A_k\dot{A}_{k})(s,\eta)|(1+\langle k,\eta\rangle/\langle s\rangle)}\widetilde{(F-F^\ast)}(s,k,\eta)\Big\|_{L^2_sL^2_{k,\eta}}\\
&\qquad\times \Big\|\big(\langle\rho\rangle\langle s\rangle+\langle \rho\rangle^{1/4}\langle s\rangle^{7/4}\big)A_{NR}(s,\rho)e^{-(\delta'_0/2)\langle \rho\rangle^{1/2}}\widetilde{\dot{V}}(s,\rho)\Big\|_{L^{\infty}_sL^2_{\rho}}\\
&\lesssim_{\delta}\epsilon_1^3.
\end{split}
\end{equation*}
The term $\mathcal{U}_6^2$ can be bounded in the same way, using \eqref{TLXH2.2}, \eqref{nar7}, and \eqref{boot2}, while the term $\mathcal{U}_6^3$ can be bounded by symmetry. The desired bounds \eqref{lkj6} follow for $a=6$. 
\end{proof}

\begin{lemma}\label{LemN3}
The bounds \eqref{lkj6} hold for $a=7$.
\end{lemma}

\begin{proof} Since $B''$ and $B''_0$ are supported in $[0,T]\times\mathbb{T}\times[b(\va_0),b(1-\va_0)]$ we can write
\begin{equation*}
\mathcal{N}_7=B''_\ast\partial_z(\Psi\phi)+B''_0\partial_z(\Psi(\phi-\phi')).
\end{equation*}
In view of \eqref{nar4}, \eqref{boot2}, and \eqref{elli1}, and recalling the definition \eqref{rew0}--\eqref{rew0.4} we have
\begin{equation}\label{lkj20}
\begin{split}
&\|B''_\ast\|_R\lesssim_\delta \eps_1,\qquad \|(\partial_z^2+(\partial_v-t\partial_z)^2)(\Psi\phi)\|_{\widetilde{W}}\lesssim_\delta\eps_1,\\
&\|B''_0\|_{R}\lesssim_\delta 1,\qquad\,\,\big\|(\partial_z^2+(\partial_v-t\partial_z)^2)(\Psi(\phi-\phi'))\big\|_{\widetilde{W}}\lesssim_\delta\eps_1^2.
\end{split}
\end{equation}

Therefore, to prove \eqref{lkj6} for $a=7$ it suffices to show that
\begin{equation}\label{lkj21}
\Big|\sum_{k\in \mathbb{Z}}\int_1^t\int_{\R^2}kA_k^2(s,\xi)\widetilde{h}(s,\xi-\eta)\widetilde{\varphi}(s,k,\eta)\overline{\widetilde{F^\ast}(s,k,\xi)}\,d\xi d\eta ds\Big|\lesssim_\delta\eps_1,
\end{equation}
for any functions $h$ and $\varphi$ satisfying $\|h\|_R\leq 1$ and $\|(\partial_z^2+(\partial_v-t\partial_z)^2)\varphi\|_{\widetilde{W}}\leq 1$. With $R_n^\ast$ defined as before, for $n\in\{0,1,2,3\}$ we let
\begin{equation*}
\begin{split}
\mathcal{U}^n_7:=\int_1^t\sum_{k\in \mathbb{Z}}\int_{\R^2}&\mathbf{1}_{R^\ast_n}((k,\xi),(k,\eta))|k| A_k^2(s,\xi)\,|\widetilde{h}(s,\xi-\eta)||\widetilde{\varphi}(s,k,\eta)||\widetilde{F^\ast}(s,k,\xi)|\,d\xi d\eta ds.
\end{split}
\end{equation*}

Letting $\rho=\xi-\eta$ and using \eqref{TLD31}, for $n\in\{0,1\}$ we can estimate
\begin{equation*}
\begin{split}
\mathcal{U}_7^n&\lesssim_{\delta}\int_1^t\sum_{k\in \mathbb{Z}^\ast}\int_{\R^2}\sqrt{|(A_{k}\dot{A}_{k})(s,\eta)|}\frac{|k|^2\langle s\rangle\langle s-\eta/k\rangle^2}{\langle s\rangle +|\eta/k|}\big|\widetilde{\varphi}(s,k,\eta)\big|\\
&\qquad\times \sqrt{|(A_k\dot{A}_k)(s,\xi)|}\big|\widetilde{F^\ast}(s,k,\xi)\big|\cdot A_R(s,\rho)\big|\widetilde{h}(s,\rho)\big|e^{-\delta'_0\langle \rho\rangle^{1/2}}\,d\xi d\eta ds\\
&\lesssim_{\delta} \Big\|\sqrt{|(A_k\dot{A}_k)(s,\xi)|}\,\widetilde{F^\ast}(s,k,\xi)\Big\|_{L^2_{s}L^2_{k,\xi}}\Big\|A_R(s,\rho)e^{-(\delta'_0/2)\langle \rho\rangle^{1/2}}\widetilde{h}(s,\rho)\Big\|_{L^{\infty}_sL^2_{\rho}}\\
&\qquad\times \Big\|\mathbf{1}_{\mathbb{Z}^\ast}(k)\sqrt{|(A_{k}\dot{A}_{k})(s,\eta)|}\frac{|k|^2\langle s\rangle\langle s-\eta/k\rangle^2}{\langle s\rangle +|\eta/k|}\widetilde{\varphi}(s,k,\eta)\Big\|_{L^2_sL^2_{k,\eta}}\\
&\lesssim_{\delta}\epsilon_1.
\end{split}
\end{equation*}
Similarly, we can use \eqref{TLD32} to estimate $\mathcal{U}_7^2\lesssim_{\delta}\epsilon_1$ and then use \eqref{TLD34} to estimate $\mathcal{U}_7^3\lesssim_{\delta}\epsilon_1$. This completes the proof of the lemma.
\end{proof}

\section{Improved control on $F-F^\ast$ and the main variables $F$ and $\Theta$}\label{SecDif}

In this section we improve the remaining bootstrap bounds:

\begin{proposition}\label{Pr1}
With the definitions and assumptions in Proposition \ref{MainBootstrap}, we have
\begin{equation}\label{hard1}
\sum_{g\in\{F-F^\ast,F,\Theta\}}\left[\E_{g}(t)+\mathcal{B}_g(t)\right]\lesssim_\delta \epsilon_1^3\qquad {\rm for\,\,any\,\,}t\in[1,T].
\end{equation}
\end{proposition}

The key issue is to prove the bounds \eqref{hard1} for the variable $F-F^\ast$, from which the other bounds follow easily. Our main tool is the following precise estimates on the linearized flow.

\begin{proposition}\label{B20}
Given $k\in\mathbb{Z}^\ast$, assume that $f_k$ is a smooth solution to the equation
 \begin{equation}\label{B21}
\partial_tf_k-ikB_0''\psi_k=X_k(t,v),
\end{equation}
\begin{equation}\label{B22}
(B'_0)^2(\partial_v-itk)^2\psi_k+B_0''(\partial_v-itk)\psi_k-k^2\psi_k=f_k, \qquad \psi_k(b(0))=\psi_k(b(1))=0,
\end{equation}
for $t\in[0,T],\,v\in[b(0),b(1)]$, with vanishing initial data $f_k(0,v)\equiv0$. Assume that $X_k$ is supported in $[0,T]\times[b(\va_0),b(1-\va_0)]$. Then 
\begin{equation}\label{B23}
\begin{split}
\widetilde{f_k}(t,\xi)=\int_0^t\widetilde{X_k}(s,\xi)\,ds+ik\int_0^t\int_s^{t}\int_{\R}\widetilde{B_0''}(\zeta)\,\widetilde{\,\Pi_k}(s,\xi-\zeta-k\tau,\xi-\zeta)\,d\zeta d\tau ds,
\end{split}
\end{equation}
for some functions $\Pi_k:[0,T]\times\mathbb{R}^2\to\mathbb{C}$. Moreover, the functions $\Pi_k$ satisfy the bounds 
\begin{equation}\label{B25}
\left\|(|k|+|\xi|)W_k(\eta)\widetilde{\,\Pi_k}(t,\xi,\eta)\right\|_{L^2_{\xi,\eta}}\lesssim_{\delta} \left\|W_k(\xi)\widetilde{X_k}(t,\xi)\right\|_{L^2_\xi},
\end{equation}
for any $t\in[0,T]$ and $\delta$ sufficiently small. Here $W_k\geq 1$ is a family of weights which depend on a small parameter $\delta\in(0,1]$, and satisfy, for any $k\in\mathbb{Z}^\ast$ and $\xi,\eta\in\R$,
\begin{equation}\label{B22.1}
\left|W_k(\xi)-W_k(\eta)\right|\lesssim e^{2\delta_0\langle\xi-\eta\rangle^{1/2}}W_k(\eta)\Big[\frac{C(\delta)}{\langle k,\eta\rangle^{1/8}}+\sqrt{\delta}\Big],
\end{equation}
where $C(\delta)\gg1$ is a large constant, and the implied constant in \eqref{B22.1} does not depend on $k, \delta$. 
\end{proposition}

The weights $W_k$ we use for our application are connected to the main weights $A_k$, see \eqref{B35.3} and \eqref{B35.7}. They are allowed to depend on $t$ as well, as long as the bounds \eqref{B22.1} hold uniformly. 

The proof of Proposition \ref{B20} is based on the ideas introduced in \cite{JiaG}. For our purposes here, we need to consider the linearized flow with an inhomogeneous term and to obtain more precise estimates. We provide the detailed proof of this proposition in the next section.

\subsection{Proof of Proposition \ref{Pr1}} In the rest of this section, we assume Proposition \ref{B20} and prove Proposition \ref{Pr1}. For $k\in\mathbb{Z}^\ast$ we define the function $\phi^{\ast}_k(t,v)$ as the solution to
\begin{equation}\label{B27}
(B'_0)^2(\partial_v-itk)^2\phi^{\ast}_k+B_0''(\partial_v-itk)\phi^{\ast}_k-k^2\phi^{\ast}_k=F^{\ast}_k,\qquad {\rm for}\,\,v\in(b(0),b(1)),
\end{equation}
with boundary value $\phi^{\ast}_k(t,b(0))=\phi^{\ast}_k(t,b(1))=0$. Notice that
\begin{equation}\label{B32}
\begin{split}
\sup_{t\in[0,T]}\Big\{\sum_{k\in\mathbb{Z}^\ast}\int_{\mathbb{\R}}A^2_k(t,\xi)(|k|^2+|\xi-kt|^2)^2\,\big|\widetilde{h_k}(t,\xi)\big|^2\,d\xi\Big\}\lesssim_{\delta} \epsilon_1^3,\\
\sum_{k\in\mathbb{Z}^\ast}\int_{0}^T\int_{\mathbb{\R}}\big|\dot{A}_kA_k(s,\xi)\big|(|k|^2+|\xi-ks|^2)^2\,\big|\widetilde{h_k}(s,\xi)\big|^2\,d\xi ds\lesssim_{\delta} \epsilon_1^3,
\end{split}
\end{equation}
where $h_k=\Psi\phi^{\ast}_k$ or $h_k=B''_0\phi^{\ast}_k$. Indeed, since $\|F^\ast\|_{W[0,T]}\lesssim_\delta\eps_1^{3/2}$ (see \eqref{lkj1}), the bounds \eqref{B32} follow from the elliptic bounds in Lemma \ref{lm:elli3} if $h_k=\Psi\phi_k^\ast$. The bounds for $B''_0\phi^\ast_k=B''_0\Psi\phi^{\ast}_k$ then follow using Lemma \ref{OldBilin} (ii) and the bounds \eqref{rew6.44} on the function $B''_0$.

We now write
\begin{equation}\label{B28}
F_k(t,v)-F^{\ast}_k(t,v)-ik\int_0^tB_0''(v)(\phi'_k-\phi^{\ast}_k)(\tau,v)\,d\tau=ik\int_0^tB_0''(v)\phi^{\ast}_k(\tau,v)\,d\tau.
\end{equation}
Setting 
\begin{equation}\label{B29}
g_k(t,v):=F_k(t,v)-F^{\ast}_k(t,v),\qquad \psi_k(t,v):=\phi_k'(t,v)-\phi^{\ast}_k(t,v),
\end{equation}
then $g_k$ satisfies the equation
\begin{equation}\label{B30}
\partial_tg_k-ikB_0''(v)\psi_k=ik(B_0''\phi^{\ast}_k)(t,v),
\end{equation}
with initial data $g_k(0,v)\equiv0$, while $\psi_k$ solves the elliptic equation
\begin{equation}\label{B31}
\begin{split}
(B'_0)^2(\partial_v-itk)^2\psi_k+&B_0''(\partial_v-itk)\psi_k-k^2\psi_k=g_k,\qquad \psi_k(t,b(0))=\psi_k(t,b(1))=0,
\end{split}
\end{equation}
in $[0,T]\times[b(0),b(1)]$. Using Proposition \ref{B20}, we obtain that
\begin{equation}\label{B33}
\begin{split}
\widetilde{g_k}(t,\xi)=ik\int_0^t\widetilde{(B''_0\phi^{\ast}_k)}(s,\xi)\,ds+ik\int_0^t\int_s^{t}\int_{\R}\widetilde{B_0''}(\zeta)\,\widetilde{\Pi_k}(s,\xi-\zeta-k\tau,\xi-\zeta)\, d\zeta\,d\tau\, ds.
\end{split}
\end{equation}
This is the main formula we need to estimate the functions $g_k=F_k-F^\ast_k$. To use it effectively, we need bounds on the functions $\Pi_k$, which we prove below:

\begin{lemma}\label{B33.3}
The functions $\widetilde{\,\Pi_k}$ satisfy the bounds
\begin{equation}\label{B35}
\begin{split}
\sup_{t\in[0,T]}\Big\{\sum_{k\in\mathbb{Z}^\ast}\int_{\mathbb{R}^2}(1+|\xi/k|)^2(|k|^2+|\eta-kt|^2)^2A^2_k(t,\eta)\big|\widetilde{\,\Pi_k}(t,\xi,\eta)\big|^2\,d\xi d\eta\Big\} \lesssim_{\delta} \epsilon_1^3,\\
\int_0^T\sum_{k\in\mathbb{Z}^\ast}\int_{\mathbb{R}^2}(1+|\xi/k|)^2(|k|^2+|\eta-ks|^2)^2\big|\dot{A}_kA_k(s,\eta)\big|\,\big|\widetilde{\,\Pi_k}(s,\xi,\eta)\big|^2\,d\xi d\eta ds
\lesssim_{\delta} \epsilon_1^{3}.
\end{split}
\end{equation}
\end{lemma}

\begin{proof} We would like to use the bounds \eqref{B25} and \eqref{B32}, but we need to be careful because our weights have to satisfy the condition \eqref{B22.1}. We use first the weights 
\begin{equation}\label{B35.3}
W_k(\eta):=A_k(t,\eta)(|k|^2+\delta^2|\eta-kt|^2).
\end{equation}
We verify now the estimates \eqref{B22.1}. If $\langle k,\xi\rangle+\langle k,\eta\rangle\leq 8\langle\xi-\eta\rangle$ then, in view of \eqref{dor4.1},
\begin{equation}\label{B35.6}
A_k(t,\xi)\leq 2A_k(t,\eta)e^{\lambda(t)\langle\xi-\eta\rangle^{1/2}}e^{2\sqrt{\delta}\langle k,\xi\rangle^{1/2}},
\end{equation}
which gives \eqref{B22.1} in the stronger form $W_k(\xi)+W_k(\eta)\lesssim_\delta e^{2\delta_0\langle\xi-\eta\rangle^{1/2}}W_k(\eta)\langle k,\eta\rangle^{-1/8}$. On the other hand, if $\langle\xi-\eta\rangle\leq (\langle k,\xi\rangle+\langle k,\eta\rangle)/8$ then we write $|W_k(\xi)-W_k(\eta)|\leq I+II$, where 
\begin{equation*}\label{B35.4}
\begin{split}
I&:=|A_k(t,\xi)-A_k(t,\eta)|(|k|^2+\delta^2|\eta-kt|^2),\\
II&:=\delta^2A_k(t,\xi)\big||\xi-kt|^2-|\eta-kt|^2\big|.
\end{split}
\end{equation*}
The desired estimates \eqref{B22.1} follow easily using \eqref{A-A2}.

We can therefore use \eqref{B25} to estimate, for any $t\in[0,T]$,
\begin{equation*}
\begin{split}
\int_{\mathbb{R}^2}(|k|+|\xi|)^2(|k|^2+&\delta^2|\eta-kt|^2)^2A^2_k(t,\eta)\big|\widetilde{\,\Pi_k}(t,\xi,\eta)\big|^2\,d\xi d\eta\\
&\lesssim_\delta \int_{\mathbb{\R}}A^2_k(t,\xi)(|k|^2+\delta^2|\xi-kt|^2)^2\,k^2\big|\widetilde{(B''_0\phi_k^\ast)}(t,\xi)\big|^2\,d\xi,
\end{split}
\end{equation*}
and the desired bounds in the first line of \eqref{B35} follow from \eqref{B32}, after dividing by $k^2$ and summing over $k\in\mathbb{Z}^\ast$.

The bounds in the second line of \eqref{B35} are similar, using \eqref{B25} with the different weights 
\begin{equation}\label{B35.7}
W_k(\eta):=\sqrt{\mu_k(t,\eta)}A_k(t,\eta)(|k|^2+\delta^2|\eta-kt|^2),
\end{equation}
where the functions $\mu_k$ are defined in \eqref{mu1}. These weights satisfy the bounds \eqref{B22.1} as well, using \eqref{B35.6} and \eqref{Genmu} if $\langle k,\xi\rangle+\langle k,\eta\rangle\leq 8\langle\xi-\eta\rangle$ or \eqref{A-A2} and \eqref{Lmu3.2} if $\langle\xi-\eta\rangle\leq (\langle k,\xi\rangle+\langle k,\eta\rangle)/8$. We can therefore use \eqref{B25} to estimate, for any $t\in[0,T]$,
\begin{equation*}
\begin{split}
\int_{\mathbb{R}^2}(|k|+|\xi|)^2(|k|^2+&\delta^2|\eta-kt|^2)^2\mu_k(t,\eta)A^2_k(t,\eta)\big|\widetilde{\,\Pi_k}(t,\xi,\eta)\big|^2\,d\xi d\eta\\
&\lesssim_\delta \int_{\mathbb{\R}}\mu_k(t,\xi)A^2_k(t,\xi)(|k|^2+\delta^2|\xi-kt|^2)^2\,k^2\big|\widetilde{(B''_0\phi_k^\ast)}(t,\xi)\big|^2\,d\xi,
\end{split}
\end{equation*}
and the desired bounds in the second line of \eqref{B35} follow from \eqref{B32} and \eqref{mudA1}, after dividing by $k^2$, summing over $k\in\mathbb{Z}^\ast$ and integrating in $t\in[0,T]$.
\end{proof}

We are now ready to bound the functions $g_k$. 

\begin{lemma}\label{Blemma}
For any $t\in[1,T]$ we have
\begin{equation}\label{B37}
\sum_{k\in \mathbb{Z}^\ast}\int_{\R} (1+\langle k,\xi\rangle/\langle t\rangle)A_k^2(t,\xi)\big|\widetilde{g_k}(t,\xi)\big|^2\,d\xi\lesssim_{\delta} \epsilon^3_1,
\end{equation}
and
\begin{equation}\label{B38}
\sum_{k\in \mathbb{Z}^\ast}\int_1^t\int_{\R}(1+\langle k,\xi\rangle/\langle s\rangle)\big|\dot{A}_kA_k(s,\xi)\big|\,\big|\widetilde{g_k}(s,\xi)\big|^2\,d\xi\,ds\lesssim_{\delta} \epsilon_1^3.
\end{equation}
\end{lemma}

\begin{proof} Using the identity \eqref{B33}, we have $|\widetilde{g_k}(t,\xi)|\leq |k|\gamma_{k,1}(t,\xi)+|k|\gamma_{k,2}(t,\xi)$, where
\begin{equation}\label{B36}
\begin{split}
\gamma_{k,1}(t,\xi)&:=\int_0^t|\widetilde{(B''_0\phi_k^{\ast})}(s,\xi)|\,ds,\\
\gamma_{k,2}(t,\xi)&:=\int_0^t\int_s^{t}\int_{\R}|\widetilde{B_0''}(\zeta)|\,|\widetilde{\,\Pi_k}(s,\xi-\zeta-k\tau,\xi-\zeta)|\, d\zeta\,d\tau\, ds.
\end{split}
\end{equation}
To simplify notations we define, for any $k\in\mathbb{Z}^\ast$, $t\geq 0$, and $\xi,\eta\in\mathbb{R}$,
\begin{equation}\label{B35''}
\begin{split}
\alpha_k(t,\xi)&:=(k^2+|\xi-kt|^2)|\widetilde{(B''_0\phi_k^{\ast})}(t,\xi)|,\\
\beta_k(t,\xi,\eta)&:=(1+|\xi/k|)(|k|^2+|\eta-kt|^2)\int_{\mathbb{R}}|\widetilde{B_0''}(\zeta)||\widetilde{\,\Pi_k}(t,\xi-\zeta,\eta-\zeta)\big|\,d\zeta.
\end{split}
\end{equation}
Using \eqref{B32} we have
\begin{equation}\label{ML3}
\sup_{t\in[0,T]}\Big\{\sum_{k\in\mathbb{Z}^\ast}\int_{\mathbb{\R}}A^2_k(t,\xi)\alpha_k^2(t,\xi)\,d\xi\Big\}+\sum_{k\in\mathbb{Z}^\ast}\int_{0}^T\int_{\mathbb{\R}}\big|\dot{A}_kA_k(s,\xi)\big|\alpha_k^2(s,\xi)\,d\xi ds\lesssim_{\delta} \epsilon_1^3.
\end{equation}
Also, using \eqref{B35}, the strong smoothness bounds \eqref{rew6.44} on $B''_0$, and the bilinear estimates \eqref{TLX7}--\eqref{vfc30.7}, we have
\begin{equation}\label{ML3.5}
\sup_{t\in[0,T]}\Big\{\sum_{k\in\mathbb{Z}^\ast}\int_{\mathbb{R}^2}A^2_k(t,\eta)\beta_k^2(t,\xi,\eta)\,d\xi d\eta\Big\}+\int_0^T\sum_{k\in\mathbb{Z}^\ast}\int_{\mathbb{R}^2}\big|\dot{A}_kA_k(s,\eta)\big|\beta_k^2(s,\xi,\eta)\,d\xi d\eta ds
\lesssim_{\delta} \epsilon_1^{3}.
\end{equation}

{\bf Step 1.} We first prove the bounds \eqref{B37}. Using the definitions \eqref{B36}--\eqref{B35''} we estimate for any $k\in\mathbb{Z}^\ast$ and $t\in[1,T]$, 
\begin{equation}\label{B38.3}
\begin{split}
&\left\|(1+\langle k,\xi\rangle/\langle t\rangle)^{1/2} A_k(t,\xi)|k|\gamma_{k,1}(t,\xi)\right\|_{L^2_\xi}\\
&=\sup_{\|P\|_{L^2_\xi}\leq 1}\int_\mathbb{R}\int_0^t |P(\xi)|(1+\langle k,\xi\rangle/\langle t\rangle)^{1/2} A_k(t,\xi)|k|\frac{\alpha_k(s,\xi)}{k^2+|\xi-ks|^2}\,ds d\xi\\
&\lesssim\frac{1}{|k|}\big\|\alpha_k(s,\xi)|(A_k\dot{A}_k)(s,\xi)|^{1/2}\big\|_{L^2_{s,\xi}}\sup_{\|P\|_{L^2_\xi}\leq 1}\Big\|\frac{(1+\langle k,\xi\rangle/\langle s\rangle)^{1/2}}{{1+|\xi/k-s|^2}}\frac{P(\xi)A_k(s,\xi)}{|(A_k\dot{A}_k)(s,\xi)|^{1/2}}\Big\|_{L^2_{s,\xi}},
\end{split}
\end{equation}
using also the fact that $A_k(t,\xi)\leq A_k(s,\xi)$ if $s\in[0,t]$. Using \eqref{newAdot} we have
\begin{equation}\label{B39}
\frac{1}{|k|^{1/2}}\Big\|P(\xi)\frac{(1+\langle k,\xi\rangle/\langle s\rangle)^{1/2}}{{1+|\xi/k-s|^2}}\frac{A_k(s,\xi)}{|(A_k\dot{A}_k)(s,\xi)|^{1/2}}\Big\|_{L^2_{s,\xi}}\lesssim_\delta\|P\|_{L^2}.
\end{equation}
Therefore
\begin{equation}\label{B40}
\left\|(1+\langle k,\xi\rangle/\langle t\rangle)^{1/2} A_k(t,\xi)|k|\gamma_{k,1}(t,\xi)\right\|_{L^2_\xi}^2\lesssim_\delta\big\|\alpha_k(s,\xi)|(A_k\dot{A}_k)(s,\xi)|^{1/2}\big\|^2_{L^2_{s,\xi}}.
\end{equation}

Similarly, to bound the contribution of $\gamma_{k,2}$ we write
\begin{equation}\label{B38.4}
\begin{split}
&\left\|(1+\langle k,\xi\rangle/\langle t\rangle)^{1/2} A_k(t,\xi)|k|\gamma_{k,2}(t,\xi)\right\|_{L^2_\xi}\\
&=\sup_{\|P\|_{L^2_\xi}\leq 1}\int_\mathbb{R}\int_0^t\int_s^t |P(\xi)|(1+\langle k,\xi\rangle/\langle t\rangle)^{1/2} \frac{A_k(t,\xi)}{|k|\langle\xi/k-s\rangle^2}\frac{\beta_k(s,\xi-k\tau,\xi)}{1+|\xi-k\tau|/|k|}\,d\tau ds d\xi\\
&\lesssim\sup_{\|P\|_{L^2_\xi}\leq 1}\frac{1}{k^2}\int_{\mathbb{R}^2}\int_0^t |P(\xi)|(1+\langle k,\xi\rangle/\langle s\rangle)^{1/2} \frac{A_k(s,\xi)}{\langle\xi/k-s\rangle^2}\frac{\beta_k(s,\eta,\xi)}{1+|\eta|/|k|}\,ds d\eta d\xi\\
&\lesssim\frac{1}{k^2}\big\|\beta_k(s,\eta,\xi)|(A_k\dot{A}_k)(s,\xi)|^{1/2}\big\|_{L^2_{s,\xi,\eta}}\sup_{\|P\|_{L^2_\xi}\leq 1}\Big\|\frac{(1+\langle k,\xi\rangle/\langle s\rangle)^{1/2}}{{\langle\xi/k-s\rangle^2}\langle\eta/k\rangle}\frac{P(\xi)A_k(s,\xi)}{|(A_k\dot{A}_k)(s,\xi)|^{1/2}}\Big\|_{L^2_{s,\xi,\eta}}.
\end{split}
\end{equation}
We can use again \eqref{B39}, and note that bounding the $L^2$ norm in $\eta$ requires an additional factor of $|k|^{1/2}$. It follows that
\begin{equation}\label{B41}
\left\|(1+\langle k,\xi\rangle/\langle t\rangle)^{1/2} A_k(t,\xi)|k|\gamma_{k,2}(t,\xi)\right\|_{L^2_\xi}^2\lesssim_\delta\big\|\beta_k(s,\eta,\xi)|(A_k\dot{A}_k)(s,\xi)|^{1/2}\big\|^2_{L^2_{s,\xi,\eta}}.
\end{equation}
The bounds \eqref{B37} follow from \eqref{B40}--\eqref{B41} and \eqref{ML3}--\eqref{ML3.5}, by summation over $k\in\mathbb{Z}^\ast$.

{\bf Step 2.} We now prove the bounds \eqref{B38}. Using the definitions \eqref{B36}--\eqref{B35''} we estimate for any $k\in\mathbb{Z}^\ast$ and $t\in[1,T]$
\begin{equation*}
\begin{split}
&\left\|(1+\langle k,\xi\rangle/\langle s\rangle)^{1/2} |(A_k\dot{A}_k)(s,\xi)|^{1/2}|k|\gamma_{k,1}(s,\xi)\right\|_{L^2_{s,\xi}}\\
&=\sup_{\|P\|_{L^2_{s,\xi}}\leq 1}\int_\mathbb{R}\int_1^t\int_0^s |P(s,\xi)|(1+\langle k,\xi\rangle/\langle s\rangle)^{1/2} |(A_k\dot{A}_k)(s,\xi)|^{1/2}\frac{|k|\alpha_k(\tau,\xi)}{k^2+|\xi-k\tau|^2}\, d\tau ds d\xi\\
&\leq\sup_{\|P\|_{L^2_{s,\xi}}\leq 1}\int_\mathbb{R}\int_0^t\int_\tau^t \Big\{|P(s,\xi)||(A_k\dot{A}_k)(s,\xi)|^{1/2}\Big\}(1+\langle k,\xi\rangle/\langle \tau\rangle)^{1/2} \frac{\alpha_k(\tau,\xi)}{|k|\langle\xi/k-\tau\rangle}\, ds d\tau d\xi.
\end{split}
\end{equation*}
The integral in $s\in[\tau,t]$ can be estimated using the Cauchy inequality and the observation
\begin{equation}\label{B43}
2\int_{\tau}^t|\dot{A}_kA_k(s,\xi)|\,ds=A_k^2(\tau,\xi)-A_k^2(t,\xi)\leq A_k^2(\tau,\xi),
\end{equation}
since the functions $A_k$ are decreasing in $s$. Therefore, the right hand side of the expression above is bounded by
\begin{equation*}
\begin{split}
&\sup_{\|P'\|_{L^2_{\xi}}\leq 1}\int_\mathbb{R}\int_0^t |P'(\xi)|A_k(\tau,\xi)(1+\langle k,\xi\rangle/\langle \tau\rangle)^{1/2} \frac{\alpha_k(\tau,\xi)}{|k|\langle\xi/k-\tau\rangle}\,d\tau d\xi.
\end{split}
\end{equation*}
This is similar to the expression in the second line of \eqref{B38.3}, so it can be estimated in the same way to give
\begin{equation}\label{B43.5}
\left\|(1+\langle k,\xi\rangle/\langle s\rangle)^{1/2} |(A_k\dot{A}_k)(s,\xi)|^{1/2}|k|\gamma_{k,1}(s,\xi)\right\|^2_{L^2_{s,\xi}}\lesssim_\delta\big\|\alpha_k(s,\xi)|(A_k\dot{A}_k)(s,\xi)|^{1/2}\big\|^2_{L^2_{s,\xi}}.
\end{equation}

Similarly, to bound the contribution of $\gamma_{k,2}$ we write
\begin{equation*}
\begin{split}
&\left\|(1+\langle k,\xi\rangle/\langle s\rangle)^{1/2} |(A_k\dot{A}_k)(s,\xi)|^{1/2}|k|\gamma_{k,2}(s,\xi)\right\|_{L^2_{s,\xi}}\\
&=\sup_{\|P\|_{L^2_{s,\xi}}\leq 1}\int_\mathbb{R}\int_1^t\int_0^s\int_u^s |P(s,\xi)|(1+\langle k,\xi\rangle/\langle s\rangle)^{1/2} \frac{|(A_k\dot{A}_k)(s,\xi)|^{1/2}}{|k|\langle\xi/k-u\rangle^2}\frac{\beta_k(u,\xi-k\tau,\xi)}{1+|\xi-k\tau|/|k|}\,d\tau du ds d\xi\\
&\leq\sup_{\|P\|_{L^2_{s,\xi}}\leq 1}\frac{1}{k^2}\int_{\mathbb{R}^2}\int_0^t\int_u^t |P(s,\xi)|(1+\langle k,\xi\rangle/\langle s\rangle)^{1/2} \frac{|(A_k\dot{A}_k)(s,\xi)|^{1/2}}{\langle\xi/k-u\rangle^2}\frac{\beta_k(u,\eta,\xi)}{1+|\eta|/|k|}\,ds du d\xi d\eta.
\end{split}
\end{equation*}
We use \eqref{B43} and the Cauchy inequality to estimate first the integral in $s\in[u,t]$, so the right-hand side of the expression above is bounded by 
\begin{equation*}
\sup_{\|P'\|_{L^2_{\xi}}\leq 1}\frac{1}{k^2}\int_{\mathbb{R}^2}\int_0^t P'(\xi)(1+\langle k,\xi\rangle/\langle u\rangle)^{1/2} \frac{A_k(u,\xi)}{\langle\xi/k-u\rangle^2}\frac{\beta_k(u,\eta,\xi)}{1+|\eta|/|k|}\,du d\xi d\eta.
\end{equation*}
This is similar to the expression in the third line of \eqref{B38.4}, therefore
\begin{equation}\label{B43.7}
\left\|(1+\langle k,\xi\rangle/\langle s\rangle)^{1/2} |(A_k\dot{A}_k)(s,\xi)|^{1/2}|k|\gamma_{k,2}(s,\xi)\right\|^2_{L^2_{s,\xi}}\lesssim_\delta\big\|\beta_k(s,\eta,\xi)|(A_k\dot{A}_k)(s,\xi)|^{1/2}\big\|^2_{L^2_{s,\xi,\eta}}.
\end{equation}
The bounds \eqref{B38} follow from \eqref{B43.5}--\eqref{B43.7} and \eqref{ML3}--\eqref{ML3.5}, by summation over $k\in\mathbb{Z}^\ast$.
\end{proof}

We can now complete the proof of Proposition \ref{Pr1}. The bounds for the function $F-F^\ast$ follow from Lemma \ref{Blemma}, once we recall that $g_k=F_k-F^\ast_k$ (compare with the definitions in \eqref{rec1''}). The bounds for the main variable $F$ then follow using also Proposition \ref{ImprovBoots2}.  Finally, to prove the bounds for $\Theta$ we start from the main elliptic equation \eqref{rea26}, and rewrite it in the form
\begin{equation*}
\partial_z^2\phi+(B'_0)^2(\partial_v-t\partial_z)^2\phi+B''_0(\partial_v-t\partial_z)\phi=F+\mathcal{G}_1+\mathcal{G}_2,
\end{equation*}
where $\mathcal{G}_1=[(B'_0)^2-(V')^2](\partial_v-t\partial_z)^2\phi$ and $\mathcal{G}_2=[B''_0-V''](\partial_v-t\partial_z)\phi$ are as in \eqref{elli2}. In view of \eqref{elli2.5} we have $\|\mathcal{G}_1\|_{\widetilde{W}[1,T]}+\|\mathcal{G}_2\|_{\widetilde{W}[1,T]}\lesssim_\delta\eps_1^2$, while the bounds $\mathcal{E}_F+\mathcal{B}_F\lesssim_\delta\eps_1^3$ we have just proved show that  $\|F\|_{W[1,T]}\lesssim_\delta\eps_1^{3/2}$. The desired bounds $\|\Theta\|_{\widetilde{W}[1,T]}\lesssim_\delta\eps_1^{3/2}$ follow from Lemma \ref{lm:elli3}. This completes the proof of Proposition \ref{Pr1}.

\section{Analysis of the linearized operator: proof of Proposition \ref{B20}}\label{PropD}

In this section we provide the proof of the key Proposition \ref{B20}, which is the only place where the spectral assumption on the linearized operator $L_k$ is used. As we have seen before, the linear estimates we prove here are essential to link the nonlinear profile $F^\ast$, which evolves perturbatively, with the full profile $F$. The proof of Proposition \ref{B20} relies on the following homogeneous bounds on the linearized flow:

\begin{lemma}\label{Le1}
Assume $k\in\mathbb{Z}^\ast$ and $a\in\R$, and consider the initial value problem 
\begin{equation}\label{B1}
\partial_tg_k+ikvg_k-ikB''_0\varphi_k=0,\qquad g_k(0,v)=X_k(v)e^{-ikav},
\end{equation}
for $(v,t)\in[\lv,\uv]\times[0,\infty)$, where $\varphi_k$ is determined through the elliptic equation
\begin{equation}\label{B2}
(B'_0)^2\partial_v^2\varphi_k+B_0''(v)\partial_v\varphi_k-k^2\varphi_k=g_k,\qquad \varphi_k(b(0))=\varphi_k(b(1))=0.
\end{equation}
Assume that $X_k\in L^2[b(0),b(1)]$ and ${\rm supp}\,X_k\subseteq[b(\vartheta_0),b(1-\vartheta_0)]$. Then there is a unique global solution $g_k\in C^1([0,\infty):L^2[b(0),b(1)])$ of the initial-value problem \eqref{B1} with ${\rm supp}\,g_k(t)\subseteq[b(\vartheta_0),b(1-\vartheta_0)]$ for any $t\geq 0$. Moreover there is a function $\Pi'_k=\Pi_k'(\xi,\eta,a)$ such that 
\begin{equation}\label{B7}
\widetilde{g_k}(t,\xi)=\widetilde{X_k}(\xi+kt+ka)+ik\int_0^t\int_{\R}\widetilde{B_0''}(\zeta)\widetilde{\,\,\Pi_k'}(\xi+kt-\zeta-k\tau,\xi+kt-\zeta,a)\,d\zeta\,d\tau.
\end{equation}
Finally, if the weights $W_k$ satisfy the bounds \eqref{B22.1} then
\begin{equation}\label{B5}
\left\|(|k|+|\xi|)W_k(\eta+ka)\widetilde{\Pi'_k}(\xi,\eta,a)\right\|_{L^2_{\xi,\eta}}\lesssim_{\delta} \big\|W_k(\eta)\widetilde{X_k}(\eta)\big\|_{L^2_\eta}.
\end{equation}
\end{lemma}

The equation \eqref{B1} is of the form
\begin{equation}\label{B5.5}
\partial_tG-iT_k(G)=0,\qquad T_k(G):=-kvG+kB''_0\Phi,\qquad G(0)=G_0,
\end{equation}
where $\Phi$ is the solution of the elliptic equation $(B'_0)^2\partial_v^2\Phi+B_0''(v)\partial_v\Phi-k^2\Phi=G$ with Dirichlet boundary conditions $\Phi(b(0))=\Phi(b(1))=0$. This elliptic equation can be solved explicitly using the change of variables $v=b(y)$ (see \eqref{LE7} below), thus $T_k$ is a bounded operator on $L^2[b(0),b(1)]$. Therefore, the equation \eqref{B5.5} can be solved explicitly
\begin{equation}\label{B5.6}
G(t)=S_k(G_0)=e^{itT_k}G_0=\sum_{n\geq 0}\frac{(itT_k)^nG_0}{n!},
\end{equation}
and the solution $G(t)$ is unique by energy estimates. The main point of the lemma is to derive the representation formula \eqref{B7} and the strong bounds \eqref{B5}. 

We first show that Lemma \ref{Le1} implies Proposition \ref{B20}.

\begin{proof}[Proof of Proposition \ref{B20}]
With $f_k$ and $\psi_k$ as in Proposition \ref{B20} let 
\begin{equation}\label{AB10}
g_k(t,v):=f_k(t,v)e^{-ikvt},\qquad \varphi_k(t,v):=\psi_k(t,v)e^{-ikvt}.
\end{equation}
The functions $g_k,\varphi_k$ satisfy for $(t,v)\in[0,T]\times[b(0),b(1)]$,
\begin{equation}\label{AB11}
\partial_tg_k+ikvg_k-ikB_0''(v)\varphi_k=X_k(t,v)e^{-iktv},
\end{equation}
\begin{equation}\label{AB12}
(B'_0)^2\partial_v^2\varphi_k+B_0''\partial_v\varphi_k-k^2\varphi_k=g_k,
\end{equation}
with initial data $g_k(0,v)=0$. By Duhamel's formula, we obtain the representation formula
\begin{equation}\label{AB13}
g_k(t,v):=\int_0^t\left\{S_k(t-a)\big[X_k(a,\cdot)e^{-ika\cdot}\big]\right\}(v)\,da,
\end{equation}
where $S_k$ is the evolution operator defined in \eqref{B5.6}. Thus, in view of the formula \eqref{B7}, we have
\begin{equation}\label{AB14}
\begin{split}
\widetilde{g_k}(t,\xi)&=\int_0^t\widetilde{X_k}(a,\xi+kt)\,da\\
                               &\,\,+ik\int_0^t\int_0^{t-a}\int_{\R}\widetilde{B_0''}(\zeta)\widetilde{\,\,\Pi_k'}(\xi+k(t-a)-\zeta-k\tau,\xi+k(t-a)-\zeta,a)\,d\zeta\,d\tau \,da,
\end{split}
\end{equation}
where the functions $\Pi_k'$ satisfy the bounds
\begin{equation}\label{AB16}
\left\|(|k|+|\xi|)W_k(\eta+ka)\widetilde{\,\Pi_k'}(\xi,\eta,a)\right\|_{L^2_{\xi,\eta}}\lesssim_{\delta} \left\|W_k(\xi)\widetilde{X_k}(a,\xi)\right\|_{L^2_\xi}.
\end{equation}
Define for $\xi,\eta\in\R, a\in[0,T]$,
\begin{equation}\label{AB17}
\widetilde{\,\Pi_k}(a,\xi,\eta):=\widetilde{\,\Pi_k'}(\xi,\eta-ka,a),
\end{equation}
The desired bounds \eqref{B25} follow from \eqref{AB16}. Using \eqref{AB10}, \eqref{AB14} and \eqref{AB17}, we also have
\begin{equation}\label{AB19}
\begin{split}
\widetilde{f_k}(t,\xi)&=
                               \int_0^t\widetilde{X_k}(a,\xi)\,da+ik\int_0^t\int_0^{t-a}\int_{\R}\widetilde{B_0''}(\zeta)\,\widetilde{\,\Pi_k}(a,\xi-ka-\zeta-k\tau,\xi-\zeta)\,d\zeta\,d\tau\,da \\
                               &=\int_0^t\widetilde{X_k}(a,\xi)\,da+ik\int_0^t\int_a^{t}\int_{\R}\widetilde{B_0''}(\zeta)\,\widetilde{\,\Pi_k}(a, \xi-\zeta-k\tau,\xi-\zeta)\,d\zeta\,d\tau \,da.
\end{split}
\end{equation}
The proposition is now proved.
\end{proof}

In the rest of this section we provide the proof of Lemma \ref{Le1}. The main idea is the same as in \cite{JiaG}. However we need to consider more general initial data with the additional modulation factor $e^{-ikav}$, in order to analyze the inhomogeneous linear evolution, and we need to prove stronger estimates. We divide the proof into several steps, organized in subsections.

\subsection{The representation formula and limiting absorption principle}\label{sec:spectral}
In this subsection we recall some important properties of the linear evolution operator from \cite{JiaG}. Throughout this section, we use the change of variables
\begin{equation}\label{F8}
v=b(y)  \qquad {\rm for}\,\,y\in[0,1].
\end{equation}
The change of variable \eqref{F8} is just the nonlinear change of variable \eqref{changeofvariables1} at $t=0$.
Define
\begin{equation}\label{LE5}
g^{\ast}_k(t,y):=g_k(t,v),\qquad \varphi^{\ast}_k(t,y):=\varphi_k(t,v), \qquad X^{\ast}_k(y):=X_k(v),
\end{equation}
where $v=b(y)$ for $y\in[0,1]$. Then $g^{\ast}_k, \varphi^{\ast}_k$ satisfy
\begin{equation}\label{LE6}
\partial_tg^{\ast}_k(t,y)+ikb(y)g^{\ast}_k(t,y)-ikb''(y)\varphi^{\ast}_k(t,y)=0,
\end{equation}
\begin{equation}\label{LE7}
-k^2\varphi^{\ast}_k(t,y)+\partial_y^2\varphi^{\ast}_k(t,y)=g^{\ast}_k(t,y),\qquad\varphi^{\ast}_k(t,0)=\varphi^\ast(t,1)=0,
\end{equation}
for $(y,t)\in[0,1]\times[0,\infty)$, with initial data $g^{\ast}_k(0,y)=X^{\ast}_k(y)e^{-ikab(y)}$.

 For each $k\in\mathbb{Z}\backslash\{0\}$, we set for any $f\in L^2[0,1]$,
 \begin{equation}\label{F3.1}
 L_kf(y):=b(y)f(y)+b''(y)\int_0^1G_k(y,z)f(z)dz,
 \end{equation}
 where $G_k$ is the Green's function for the operator $k^2-\partial_y^2$ on $[0,1]$ with zero Dirichlet boundary condition defined in \eqref{elli11}. Then the system \eqref{LE6}-\eqref{LE7} can be reformulated as
 \begin{equation}\label{F3.2}
 \partial_tg^{\ast}_k(t,y)+ikL_kg^{\ast}_k(t,y)=0.
 \end{equation}

We first record an important representation formula, see \cite[Proposition 2.1]{JiaG}.

\begin{proposition}\label{F3.3}
For $k\in\mathbb{Z}^\ast$ we have the following representation formula for $\varphi^{\ast}_k$
 \begin{equation}\label{F3.4}
 \varphi^{\ast}_k(t,y)=-\frac{1}{2\pi i}\lim_{\epsilon\to0+}\int_{0}^1e^{-ikb(y_0) t}b'(y_0)\left[\psi_{k,\epsilon}^{-}(y,y_0)-\psi_{k,\epsilon}^{+}(y,y_0)\right]dy_0,
 \end{equation}
 where $\psi_{k,\epsilon}^{\iota}:[0,1]^2\to\mathbb{C}$ are defined for $\iota\in\{+,-\}$ and $\epsilon\in[-1/4,1/4]\backslash\{0\}$ by
\begin{equation}\label{F3.6}
\psi_{k,\eps}^\pm(y,y_0):=\int_0^1G_k(y,z)\big[(-b(y_0)+L_k\pm i\eps)^{-1}(X_k(\cdot)e^{-ikab(\cdot)}\big](z)\,dz,
\end{equation}
where $G_k$ are the Green functions defined in \eqref{elli11}. Moreover, the generalized eigenfunctions $\psi_{k,\eps}^\pm$ are solutions of the equation
 \begin{equation}\label{F7}
 \begin{split}
  -k^2\psi_{k,\epsilon}^{\iota}(y,y_0)+\frac{d^2}{dy^2}\psi_{k,\epsilon}^{\iota}(y,y_0)-\frac{b''(y)}{b(y)-b(y_0)+i\iota\epsilon}\psi_{k,\epsilon}^{\iota}(y,y_0)=\frac{-X^{\ast}_k(y)e^{-ikab(y)}}{b(y)-b(y_0)+i\iota\epsilon}.
 \end{split}
 \end{equation}
 \end{proposition}

\begin{remark}\label{F6.1}
The existence of the functions $\psi^{\iota}_{k,\epsilon}$ for $\epsilon\in[-1/4,1/4]\setminus\{0\}$ follows from our spectral assumptions. These functions depend on the parameter $a$ as well, but we suppress this dependence for simplicity of notation.
\end{remark}

We transfer now the results of Lemma \ref{T5} to the new variables.

\begin{lemma}\label{F16.1}
For any $f\in H^1_k(\R)$ and $\epsilon\in[-1/4,1/4]\backslash\{0\}, k\in\mathbb{Z}^\ast, w\in[\lv,\uv]$, let 
\begin{equation}\label{F16.2}
S_{k,w,\epsilon}f(v):=\int_{\R}\Psi(v)\mathcal{G}_k(v,v')(\partial_{v'}B'_0)(v')\frac{f(v')}{v'-w+i\epsilon}dv',
\end{equation}
where $\mathcal{G}_k(v,v')=G_k(b^{-1}(v),b^{-1}(v'))$ are the renormalized Green functions defined in the proof of Lemma \ref{lm:elli3}. Then for all $k\in\mathbb{Z}^\ast,\,w\in[\lv,\uv]$, $f\in H^1_k(\mathbb{R})$, and sufficiently small $\epsilon\neq 0$, 
\begin{equation}\label{F16.3}
\|S_{k,w,\epsilon}f\|_{H^1_k(\R)} \lesssim|k|^{-1/3} \|f\|_{H^1_k(\R)}\qquad {\rm and}\qquad \|f\|_{H^1_k(\R)}\lesssim\|f+S_{k,w,\epsilon}f\|_{H^1_k(\R)}.
\end{equation}

Define also
\begin{equation}\label{F16.5}
S'_{k,w,\epsilon}f(v):=\int_{\R}\Psi(v+w)\mathcal{G}_k(v+w,v'+w)(\partial_{v'}B'_0)(v'+w)\frac{f(v')}{v'+i\epsilon}dv'.
\end{equation}
Then for all $k\in\mathbb{Z}^\ast,\,w\in[\lv,\uv]$, $f\in H^1_k(\mathbb{R})$, and sufficiently small $\epsilon\neq 0$, 
\begin{equation}\label{F16.6}
\|S'_{k,w,\epsilon}f\|_{H^1_k(\R)} \lesssim|k|^{-1/3} \|f\|_{H^1_k(\R)}\qquad {\rm and}\qquad\|f\|_{H^1_k(\R)}\lesssim\|f+S'_{k,w,\epsilon}f\|_{H^1_k(\R)}.
\end{equation}
\end{lemma}

\begin{proof}
The bounds \eqref{F16.3} follow from Lemma \ref{T5}, using the change of the variables formula \eqref{F8}.
The bounds \eqref{F16.6} follow by a shift of variables $v\to v-w$. 
\end{proof}

\subsection{Gevrey bounds for generalized eigenfunctions}\label{gev}
In this section we study the regularity of the generalized eigenfunctions $\psi^{\iota}_{k,\epsilon}(y,y_0), y,y_0\in[0,1], \iota\in\{\pm\}$. The starting point is the equation \eqref{F7}, which can be reformulated as
\begin{equation}\label{H0}
\psi^{\iota}_{k,\epsilon}(y,y_0)+\int_0^1G_k(y,z)\frac{b''(z)\psi^{\iota}_{k,\epsilon}(z,y_0)}{b(z)-b(y_0)+i\iota\epsilon}\,dz=\int_0^1G_k(y,z)\frac{X^{\ast}_k(z)e^{-ikab(z)}}{b(z)-b(y_0)+i\iota\epsilon}dz.
\end{equation}
Denote for $y,y_0\in[0,1]$,
\begin{equation}\label{LE9}
h(y,y_0):=\int_0^1G_k(y,z)\frac{X^{\ast}_k(z)e^{-ikab(z)}}{b(z)-b(y_0)+i\iota\epsilon}dz.
\end{equation}

We can now prove bounds on the low frequencies of the generalized eigenfunctions.

\begin{lemma}\label{H1}
(i) We have
\begin{equation}\label{LE10}
\left\|h(y,y_0)\right\|_{L^2_{y,y_0}}+|k|^{-1}\left\|\partial_yh(y,y_0)\right\|_{L^2_{y,y_0}}\lesssim |k|^{-1}\left\|X_k^\ast\right\|_{L^2_v}.
\end{equation}

(ii) For $\iota\in\{\pm\}$, $k\in\mathbb{Z}^\ast$, and $\epsilon\in[-1/4,1/4]\setminus\{0\}$ sufficiently small, we have
\begin{equation}\label{H0.1}
|k|\left\|\Psi(b(y))\psi^{\iota}_{k,\epsilon}(y,y_0)\right\|_{L^2_{y,y_0}}+\left\|\partial_y(\Psi(b(y))\psi^{\iota}_{k,\epsilon})(y,y_0)\right\|_{L^2_{y,y_0}}\lesssim \left\|X_k\right\|_{L^2}.
\end{equation}
Moreover
\begin{equation}\label{F10.4}
\lim_{\epsilon\to0+}\Big[\psi^{-}_{k,\epsilon}(y,y_0)-\psi^{+}_{k,\epsilon}(y,y_0)\Big]\equiv 0, \qquad{\rm for}\,\,y_0\in[0,\vartheta_0/2]\cup[1-\vartheta_0/2,1].
\end{equation}
\end{lemma}

\begin{proof} (i) Using $L^2$ boundedness of the Hilbert transform, we estimate 
\begin{equation*}
\begin{split}
\|h(y,y_0)\|_{L^2_{y,y_0}}&=\sup_{\|P\|_{L^2}\leq 1}\Big|\int_{[0,1]^3}P(y,y_0)G_k(y,z)\frac{X^{\ast}_k(z)e^{-ikab(z)}}{b(z)-b(y_0)+i\iota\epsilon}\,dzdydy_0\Big|\\
&\lesssim \sup_{\|P'\|_{L^2}\leq 1}\int_{[0,1]^2}\big|P'(y,z)G_k(y,z)X^{\ast}_k(z)\big|\,dzdy\\
&\lesssim |k|^{-3/2}\|X^\ast_k\|_{L^2},
\end{split}
\end{equation*}
where in the last inequality we used the bounds $\|G_k(y,z)\|_{L^2_y}\lesssim |k|^{-3/2}$ for any $z\in[0,1]$ (compare with \eqref{elli11}). The estimate on the second term in the left-hand side of \eqref{LE10} is similar, since $\|(\partial_yG_k)(y,z)\|_{L^2_y}\lesssim |k|^{-1/2}$.

(ii) The bounds \eqref{H0.1} follow from \eqref{T6} and \eqref{H0}--\eqref{LE9} (the functions $\Psi(b(.))\psi_{k,\eps}^\pm(.,y_0)$ are in $H^1_k(\mathbb{R})$ due to \eqref{F3.6}). The identities \eqref{F10.4} were proved in \cite[Lemma 4.1]{JiaG}.
\end{proof}

We now turn to the main case when $y_0\in[\vartheta_0/2,1-\vartheta_0/2]$.
Recall the change of variables \eqref{F8}, and set for $k\in\mathbb{Z}\backslash\{0\}, \iota\in\{\pm\}$ and sufficiently small $\epsilon\neq0$,
\begin{equation}\label{F9}
\phi^{\iota}_{k,\epsilon}(v,w):=\psi^{\iota}_{k,\epsilon}(y,y_0),\qquad v=b(y),\,w=b(y_0).
\end{equation}


The following lemma contains the main estimates for the generalized eigenfunctions.
\begin{lemma}\label{S1}
Define for $\iota\in\{\pm\}, k\in\mathbb{Z}\backslash\{0\}$ and sufficiently small $\epsilon>0$
\begin{equation}\label{F17.1}
\Pi^{\iota}_{k,\epsilon}(v,w):=\Psi(v+w)\phi^{\iota}_{k,\epsilon}(v+w,w)\Psi(w), \qquad {\rm for}\,\,v,w\in\R.
\end{equation}
If $W_k$ are weights satisfying \eqref{B22.1} then, for $\delta,\epsilon>0$ sufficiently small,
\begin{equation}\label{S2}
\left\|(|k|+|\xi|)W_k(\eta+ka)\,\widetilde{\,\Pi^{\iota}_{k,\epsilon}}(\xi,\eta)\right\|_{L^2_{\xi,\eta}}\lesssim_\delta \left\|W_k(\eta)\,\widetilde{X_k}(\eta)\right\|_{L^2_\eta}.
\end{equation}
\end{lemma}

\begin{proof}
Using \eqref{H0.1} and the definitions \eqref{F17.1}, we have the bounds
\begin{equation}\label{F17.3}
\left\|(|k|+|\partial_v|)\Pi^{\iota}_{k,\epsilon}\right\|_{L^2_{v,w}}\lesssim\left\|X_k\right\|_{L^2},
\end{equation}
which are useful to control the low frequency components of $\Pi^{\iota}_{k,\epsilon}$. We divide the rest of the proof into several steps.

{\bf Step 1.} We derive first the main equations for $\Pi^{\iota}_{k,\epsilon}(v,w)$. Using the definitions \eqref{F9}, we can reformulate equation \eqref{H0} as
\begin{equation}\label{F17}
\begin{split}
&\Psi(v)\phi^{\iota}_{k,\epsilon}(v,w)+\int_{\mathbb{R}}\Psi(v)\mathcal{G}_k(v,v')(\partial_{v'}B'_0)(v')\frac{\Psi(v')\phi^{\iota}_{k,\epsilon}(v',w)}{v'-w+i\iota\epsilon}\,dv'\\
&=\int_{\mathbb{R}}\Psi(v)\mathcal{G}_k(v,v')\frac{1}{B'_0(v')}\frac{X_k(v')e^{-ikav'}}{v'-w+i\iota\epsilon}\,dv',
\end{split}
\end{equation}
for $v\in\R$ and $w\in[\lv,\uv]$, since $\Psi\equiv 1$ on the support of $\partial_vB'_0$. Recall also that $\Psi\equiv 1$ on the support of $X_k$, and let $\mathcal{G}'_k(v,w):=\Psi(v)\mathcal{G}_k(v,w)\Psi(w)$. It follows that the function $\Pi^{\iota}_{k,\epsilon}(v,w)$ satisfies the more regular (in $w$) equation
\begin{equation}\label{F17.2}
\begin{split}
\Pi^{\iota}_{k,\epsilon}(v,w)&+\int_{\mathbb{R}}\mathcal{G}'_k(v+w,v'+w)(\partial_{v'}B'_0)(v'+w)\frac{\Pi^{\iota}_{k,\epsilon}(v',w)}{v'+i\iota\epsilon}\,dv'\\
&=\int_{\mathbb{R}}\mathcal{G}'_k(v+w,v'+w)\frac{\Psi(w)}{B'_0(v'+w)}\frac{X_k(v'+w)e^{-ika(v'+w)}}{v'+i\iota\epsilon}\,dv'.
\end{split}
\end{equation}

{\bf Step 2.} We now study the regularity of the functions $\Pi^{\iota}_{k,\epsilon}$ using equation \eqref{F17.2}. Define the operator $W_k^a$ by the Fourier multiplier
\begin{equation}\label{F24}
\widetilde{\,W^{a}_kh}(\eta):=W_k(\eta+ka)\widetilde{h}(\eta),\qquad{\rm for\,\,any\,\,}h\in L^2(\R).
\end{equation}
The basic idea is to use the limiting absorption principle in Lemma \ref{F16.1} to bound $\Pi^{\iota}_{k,\epsilon}$. We note that $\Pi^{\iota}_{k,\epsilon}$ is very smooth in $w$ but not so smooth in $v$, due to the presence of the singular factor $1/(v'+i\iota\epsilon)$. In order to prove Gevrey regularity of $\Pi^{\iota}_{k,\epsilon}$ in $w$, we apply the operator $W_k^a$, which acts on the variable $w$, to equation \eqref{F17.2} and obtain
\begin{equation}\label{F25}
\begin{split}
&W^a_k\,\Pi^{\iota}_{k,\epsilon}(v,w)+\int_{\mathbb{R}}\mathcal{G}'_k(v+w,v'+w)(\partial_{v'}B'_0)(v'+w)\frac{W^a_k\,\Pi^{\iota}_{k,\epsilon}(v',w)}{v'+i\iota\epsilon}\,dv'\\
&=W^a_k\Big[\int_{\mathbb{R}}\mathcal{G}'_k(v+\cdot,v'+\cdot)\frac{\Psi(\cdot)}{B'_0(v'+\cdot)}\frac{X_k(v'+\cdot)e^{-ika(v'+\cdot)}}{v'+i\iota\epsilon}\,dv'\Big](w)+\mathcal{C}^{\iota}_{k,\epsilon}(v,w)\\
&=:F^{\iota}_{k,\epsilon}(v,w)+\mathcal{C}^{\iota}_{k,\epsilon}(v,w),
\end{split}
\end{equation}
for $v,w\in\R$, where the commutator term $\mathcal{C}^{\iota}_{k,\epsilon}(v,w)$ is defined as
\begin{equation}\label{S26}
\begin{split}
\mathcal{C}^{\iota}_{k,\epsilon}(v,w):=&\int_{\mathbb{R}}\mathcal{G}'_k(v+w,v'+w)(\partial_{v'}B'_0)(v'+w)\frac{W^a_k\,\Pi^{\iota}_{k,\epsilon}(v',w)}{v'+i\iota\epsilon}\,dv'\\
&-W^a_k\Big[\int_{\mathbb{R}}\mathcal{G}'_k(v+\cdot,v'+\cdot)(\partial_{v'}B'_0)(v'+\cdot)\frac{\Pi^{\iota}_{k,\epsilon}(v',\cdot)}{v'+i\iota\epsilon}\,dv'\Big](w).
\end{split}
\end{equation}

We fix now a cutoff function $\Psi_0$ supported in $[b(\va_0/8),b(1-\va_0/8)]$, equal to $1$ in $[b(\va_0/4),b(1-\va_0/4)]$ and satisfying $\|e^{\langle \xi\rangle^{3/4}}\widetilde{\Psi_0}(\xi)\|_{L^\infty}\lesssim 1$. Applying \eqref{F16.6} for each $w$ and taking $L^2$ in $w$, we obtain from \eqref{F25} 
\begin{equation}\label{F26.1}
\left\|\Psi_0(w)(|k|+|\partial_v|)W^a_k\,\Pi^{\iota}_{k,\epsilon}\right\|_{L^2_{v,w}}\lesssim \left\|(|k|+|\partial_v|)F^{\iota}_{k,\epsilon}\right\|_{L^2_{v,w}}+\left\|(|k|+|\partial_v|)\mathcal{C}^{\iota}_{k,\epsilon}\right\|_{L^2_{v,w}}.
\end{equation}

{\bf Step 3.} We bound now the terms in the right-hand side of \eqref{F26.1}. We show first that
\begin{equation}\label{F28}
\left\|(|k|+|\partial_v|)F^{\iota}_{k,\epsilon}\right\|_{L^2_{v,w}}\lesssim_{\delta} \big\|W_k(\eta)\widetilde{X_k}(\eta)\big\|_{L^2_\eta}.
\end{equation}
Using \eqref{F25} and taking Fourier transform in $v,w$, we obtain that
\begin{equation}\label{F30}
\begin{split}
\widetilde{F^{\iota}_{k,\epsilon}}(\xi,\eta)&=CW_k(\eta+ka)\int_{\R^4}\widetilde{\mathcal{G}'_k}(\xi,\zeta)e^{-iw\eta+i\xi w+i\zeta(v'+w)}\Psi(w)\\
&\hspace{1in}\times\frac{X_k(v'+w)e^{-ika(v'+w)}}{B'_0(v'+w)} e^{i(v'+i\iota\epsilon)\gamma}\mathbf{1}_+(\iota\gamma)\,dv'dw d\zeta d\gamma\\
&=CW_k(\eta+ka)\int_{\R^2}\widetilde{\mathcal{G}'_k}(\xi,\zeta)\widetilde{h^a_k}(-\zeta-\gamma,\eta-\xi-\zeta)e^{-\epsilon\iota\gamma}\mathbf{1}_+(\iota\gamma)\,d\zeta d\gamma.
\end{split}
\end{equation}
where $\mathbf{1}_+$ denotes the characteristic function of the interval $[0,\infty)$ and
\begin{equation}\label{F31}
h^a_k(v,w):=\Psi(w)X_k(v+w)e^{-ika(v+w)}/B'_0(v+w),\qquad {\rm for}\,\,v,w\in\R.
\end{equation}
Since $\Psi\equiv 1$ on the support of $X_k$, we can write
\begin{equation}\label{F31.1}
h^a_k(v,w)=\Upsilon(v,w)X_k(v+w)e^{-ika(v+w)}, \qquad{\rm where}\,\,\Upsilon(v,w):=\Psi(w)\Psi(v+w)/B'_0(v+w).
\end{equation}
Using general properties of Gevrey spaces (Lemmas \ref{lm:Gevrey} and \ref{GPF}), and the regularity of $b$, see \eqref{IntB}--\eqref{IntB1}, we obtain that $\big|\widetilde{\,\Upsilon\,}(\xi,\eta)\big|\lesssim e^{-4\delta_0\langle\xi,\eta\rangle^{1/2}}$ for any $\xi,\eta\in\R$. Therefore
\begin{equation}\label{F31.3}
\big|\widetilde{\,h^a_k}(\xi,\eta)\big|\lesssim \int_{\R}e^{-4\delta_0\langle \xi-\alpha,\eta-\alpha\rangle^{1/2}}\big|\widetilde{X_k}(\alpha+ka)\big|\,d\alpha.
\end{equation}
As in \eqref{elli15}, in view of \cite[Lemma A3]{JiaG}, we have $\big|\widetilde{\mathcal{G}'_k}(\xi,\zeta)\big|\lesssim e^{-4\delta_0\langle\xi+\zeta\rangle^{1/2}}(k^2+|\xi|^2)^{-1}$. Using now \eqref{F30} it follows that
\begin{equation*}
\begin{split}
(k^2+\xi^2)\big|\widetilde{F^{\iota}_{k,\epsilon}}(\xi,\eta)\big|&\lesssim\int_{\R^3}W_k(\eta+ka)e^{-4\delta_0\langle\xi+\zeta\rangle^{1/2}}e^{-4\delta_0\langle -\zeta-\gamma-\alpha,\eta-\xi-\zeta-\alpha\rangle^{1/2}}\big|\widetilde{X_k}(\alpha+ka)\big|\,d\zeta d\gamma d\alpha\\
&\lesssim\int_{\R}W_k(\eta+ka)e^{-3\delta_0\langle\eta-\alpha\rangle^{1/2}}\big|\widetilde{X_k}(\alpha+ka)\big|\,d\alpha.
\end{split}
\end{equation*}
The desired bounds \eqref{F28} then follow since $W_k(\eta+ka)\lesssim_\delta W_k(\alpha+ka)e^{2\delta_0\langle\eta-\alpha\rangle^{1/2}}$.

{\bf{Step 4.}} We show now that the term $\mathcal{C}^{\iota}_{k,\epsilon}$ satisfies the bounds 
\begin{equation}\label{F29}
\|(|k|+|\partial_v|)\mathcal{C}^{\iota}_{k,\epsilon}\|_{L^2_{v,w}}\lesssim\delta^{1/2}\big\|(|k|+|\partial_v|)W^a_k\,\Pi^{\iota}_{k,\epsilon}\big\|_{L^2_{v,w}}+C_{\delta} \big\|(|k|+|\partial_v|)\Pi^{\iota}_{k,\epsilon}\big\|_{L^2_{v,w}}.
\end{equation}
Indeed, using the definition \eqref{S26} and expanding as in \eqref{F30}, we have 
\begin{equation*}
\begin{split}
\widetilde{\mathcal{C}^{\iota}_{k,\epsilon}}(\xi,\eta)&=C\int_{\R^3}\widetilde{\mathcal{G}'_k}(\xi,\zeta)\widetilde{\partial_{v}B'_0}(\alpha)\Big[W^a_k(\eta-\xi-\zeta-\alpha)-W^a_k(\eta)\Big]\\
&\qquad\qquad\times\widetilde{\,\Pi^{\iota}_{k,\epsilon}}(-\zeta-\gamma-\alpha,\eta-\xi-\zeta-\alpha)\mathbf{1}_+(\iota\gamma)e^{-\eps\iota\gamma}\,d\alpha d\zeta d\gamma.
\end{split}
\end{equation*}
Since $\big|\widetilde{\mathcal{G}'_k}(\xi,\zeta)\big|\lesssim e^{-4\delta_0\langle\xi+\zeta\rangle^{1/2}}(k^2+|\xi|^2)^{-1}$ and using also \eqref{B22.1}, we can estimate
\begin{equation*}
\begin{split}
&(k^2+\xi^2)\left|\widetilde{\mathcal{C}^{\iota}_{k,\epsilon}}(\xi,\eta)\right|\lesssim \int_{\R^3} e^{-2\delta_0\langle\xi+\zeta\rangle^{1/2}}e^{2\delta_0\langle\alpha\rangle^{1/2}}\big|\widetilde{\,\partial_{v}B'_0\,}(\alpha)\big|\left[\sqrt{\delta}+\frac{C_{\delta}}{\langle ka+\eta-\xi-\zeta-\alpha\rangle^{1/8}}\right]\\
&\hspace{1in} \times W_k(ka+\eta-\xi-\zeta-\alpha)\big|\widetilde{\,\Pi^{\iota}_{k,\epsilon}}(-\zeta-\gamma-\alpha,\eta-\xi-\zeta-\alpha)\big|\,d\alpha d\zeta d\gamma\\
&\lesssim \int_{\R^3}e^{-2\delta_0\langle \alpha,\zeta\rangle^{1/2}}\left[\sqrt{\delta}+\frac{C_{\delta}}{\langle ka+\eta-\zeta-\alpha\rangle^{1/8}}\right]W_k(ka+\eta-\zeta-\alpha)\big|\widetilde{\,\Pi^{\iota}_{k,\epsilon}}(\gamma,\eta-\zeta-\alpha)\big|\,d\alpha d\zeta d\gamma,
 \end{split}
\end{equation*}
from which \eqref{F29} follows.

{\bf{Step 5.}} 
We show now that
\begin{equation}\label{F100}
\begin{split}
&\left\|(|k|+|\partial_v|)W^a_k\,\Pi^{\iota}_{k,\epsilon}\right\|_{L^2_{v,w}}\lesssim \left\|\Psi_0(w)(|k|+|\partial_v|)W^a_k\,\Pi^{\iota}_{k,\epsilon}\right\|_{L^2_{v,w}}+C_{\delta}\left\|(|k|+|\partial_v|)\Pi^{\iota}_{k,\epsilon}\right\|_{L^2_{v,w}}.
\end{split}
\end{equation}
Indeed, for $v,w\in\R$ let
\begin{equation}\label{C32}
\begin{split}
H(v,w):=&(|k|+|\partial_v|)W^a_k\,\Pi^{\iota}_{k,\epsilon}(v,w)-(|k|+|\partial_v|)\Psi_0(w)W^a_k\,\Pi^{\iota}_{k,\epsilon}(v,w)\\
          =&(|k|+|\partial_v|)W^a_k\,(\Psi_0\Pi^{\iota}_{k,\epsilon})(v,w)-(|k|+|\partial_v|)\Psi_0(w)W^a_k\,\Pi^{\iota}_{k,\epsilon}(v,w).
\end{split}
\end{equation}
For the simplicity of notation we suppressed the dependence of $H$ on $\iota,\epsilon,k$ in the above definition. By the support property of $\Psi_0$ and the bounds \eqref{B22.1} we have
\begin{equation*}
\begin{split}
&\left|\widetilde{\,H\,}(\xi,\eta)\right|=(|k|+|\xi|)\left|\int_{\R}\widetilde{\,\Psi_0}(\zeta)\,\widetilde{\,\Pi^{\iota}_{k,\epsilon}}(\xi,\eta-\zeta)\Big[W_k(\eta+ka)-W_k(\eta+ka-\zeta)\Big]\,d\zeta\right|\\
&\lesssim \int_{\R}e^{-2\delta_0\langle\zeta\rangle^{1/2}}\,
\left[\sqrt{\delta}+C_{\delta}\langle ka+\eta-\zeta\rangle^{-1/8}\right](|k|+|\xi|)W_k(ka+\eta-\zeta)\big|\widetilde{\,\Pi^{\iota}_{k,\epsilon}}(\xi,\eta-\zeta)\big|\,d\zeta.
\end{split}
\end{equation*}
Therefore,
\begin{equation}\label{C33}
\big\|\widetilde{\,H\,}(\xi,\eta)\big\|_{L^2_{\xi,\eta}}\lesssim \sqrt{\delta} \left\|(|k|+|\partial_v|)W^a_k\,\Pi^{\iota}_{k,\epsilon}(v,w)\right\|_{L^2_{v,w}}+C_{\delta}\left\|(|k|+|\partial_v|)\Pi^{\iota}_{k,\epsilon}(v,w)\right\|_{L^2_{v,w}}.
\end{equation}
The bounds \eqref{F100} follow from \eqref{C32}--\eqref{C33}, provided that $\delta>0$ is sufficiently small (see also the qualitative bounds \eqref{F37.6} below).

{\bf Step 6.} We now complete the proof of \eqref{S2}. Using the bounds \eqref{F100}, \eqref{F26.1}, \eqref{F28}, and \eqref{F29}, we have
\begin{equation}\label{F37}
\begin{split}
\big\|(|k|&+|\partial_v|)W^a_k\,\Pi^{\iota}_{k,\epsilon}\big\|_{L^2_{v,w}}\lesssim \left\|\Psi_0(w)(|k|+|\partial_v|)W^a_k\,\Pi^{\iota}_{k,\epsilon}\right\|_{L^2_{v,w}}+C_{\delta}\left\|(|k|+|\partial_v|)\Pi^{\iota}_{k,\epsilon}\right\|_{L^2_{v,w}}\\
&\lesssim \left\|(|k|+|\partial_v|)F^{\iota}_{k,\epsilon}\right\|_{L^2_{v,w}}+\left\|(|k|+|\partial_v|)\mathcal{C}^{\iota}_{k,\epsilon}\right\|_{L^2_{v,w}}+C_{\delta}\left\|(|k|+|\partial_v|)\Pi^{\iota}_{k,\epsilon}\right\|_{L^2_{v,w}}\\
&\lesssim C_\delta\big\|W_k(\eta)\widetilde{X_k}(\eta)\big\|_{L^2_\eta}+\delta^{1/2}\big\|(|k|+|\partial_v|)W^a_k\,\Pi^{\iota}_{k,\epsilon}\big\|_{L^2_{v,w}}+C_{\delta}\left\|(|k|+|\partial_v|)\Pi^{\iota}_{k,\epsilon}\right\|_{L^2_{v,w}}.
\end{split}
\end{equation}
We would like to absorb the term $\delta^{1/2}\big\|(|k|+|\partial_v|)W^a_k\,\Pi^{\iota}_{k,\epsilon}\big\|_{L^2_{v,w}}$ into the left-hand side, and use \eqref{F17.3} to conclude the proof of the lemma. For this we need also the qualitative bounds
\begin{equation}\label{F37.6}
\big\|(|k|+|\partial_v|)W^a_k\,\Pi^{\iota}_{k,\epsilon}\big\|_{L^2_{v,w}}<\infty.
\end{equation}
We can arrange this by working first with the weights $W_{k,\rho}(\xi)=W_k(\xi)/(1+\rho W_k(\xi))$, $\rho>0$, which still satisfy the main bounds \eqref{B22.1} uniformly in $\rho$. The qualitative bounds \eqref{F37.6} are satisfied for these weights, due to \eqref{F17.3}. We can therefore prove the desired bounds \eqref{S2} for the weights $W_{k,\rho}$ uniformly in $\rho$, and then let $\rho\to 0$.
\end{proof}

\subsection{Proof of Lemma \ref{Le1}}
We can now complete the proof of Lemma \ref{Le1}. For a suitable constant $C_0$ we define
\begin{equation}\label{D1.1}
\Pi'_k(v,w,a):=\frac{i}{4\pi^2}\lim_{\epsilon_n\to 0+}\big[\Pi^{-}_{k,\epsilon_n}(v,w)-\Pi^{+}_{k,\epsilon_n}(v,w)\big],
\end{equation}
as a weak limit along along a subsequence (in fact the limit above exists in the strong sense, see \cite[Lemma 4.3]{JiaL}, but this is not needed here). The desired bounds \eqref{B5} follow from \eqref{S2}. To prove the representation formula \eqref{B7} we start from the identities \eqref{F3.4}. We make the change of variables $v=b(y), w=b(y_0)$ and use \eqref{F10.4} and \eqref{F17.1} to obtain
\begin{equation}\label{D2}
\begin{split}
\Psi(v)\varphi_k(t,v)&=-\frac{1}{2\pi i}\lim_{\epsilon_n\to0+}\int_{\R}e^{-ikwt}\Psi(v)\big[\phi^{-}_{k,\epsilon_n}(v,w)-\phi^{+}_{k,\epsilon_n}(v,w)\big]\Psi(w)\,dw\\
                &=\frac{i}{2\pi}\int_{\R}e^{-ikwt}(-4i\pi^2)\Pi'_k(v-w,w,a)\,dw.
\end{split}
\end{equation}
Hence 
\begin{equation}\label{D3}
\widetilde{\,\Psi \varphi_k}(t,\xi)=2\pi\widetilde{\,\Pi'_k}(\xi,\xi+kt,a).
\end{equation}
In view of the equation \eqref{B1} and the definition \eqref{D1.1}, we obtain that
\begin{equation}\label{D6}
\partial_t\big[e^{ikvt}g_k(t,v)\big]=ikB_0''(v)\varphi_k(t,v)e^{ikvt},\qquad {\rm for}\,\,v\in[\lv,\uv].
\end{equation}
We notice that $\Psi\equiv1$ on the support of $B''_0$. Therefore,
\begin{equation}\label{D7}
e^{ikvt}g_k(t,v)-g_k(0,v)=ik\int_0^tB_0''(v)\Psi(v)\varphi_k(\tau,v)e^{ikv\tau}\,d\tau.
\end{equation}
Using \eqref{D3}, we obtain
\begin{equation}\label{D8}
\begin{split}
\widetilde{g_k}(t,\xi-kt)-\widetilde{X_k}(\xi+ak)&=ik\int_0^t\int_{\R} \widetilde{B_0''}(\zeta) \widetilde{\,\,\Pi'_k}(\xi-\zeta-k\tau,\xi-\zeta,a)\,d\zeta d\tau,
\end{split}
\end{equation}
which gives \eqref{B7}. This completes the proof of Lemma \ref{Le1}.

\section{Proof of the main theorem}\label{ProofMainThm}

In this section we complete the proof of Theorem \ref{maintheoremINTRO}. We start with a local regularity lemma (see \cite[Lemma 3.1]{IOJI} for a simple proof adapted to our situation, or more general results on the Gevrey regularity of Euler flows in \cite{Foias,Vicol,Levermore}).

\begin{lemma}\label{lm:persistenceofhigherregularity}
Assume that $s\in[1/4,3/4]$, $\lambda_0\in(0,1)$, and that $\mathrm{supp}\,\omega_0\subseteq \T\times [2\vartheta_0,1-2\vartheta_0]$. Assume also that
\begin{equation}\label{ini1}
A:=\left\|\langle \nabla\rangle^3\,\omega_0\right\|_{\mathcal{G}^{\lambda_0,s}}<\infty,\qquad \int_{\mathbb{T}\times[0,1]}\omega_0(x,y)\,dxdy=0.
\end{equation}
 Let $\omega\in C([0,\infty):H^{10})$ denote the unique smooth solution of the system \eqref{Eur1}. Assume that for some $T>0$ and all $t\in[0,T]$,
 \begin{equation}\label{SupA1}
 \mathrm{supp}\,\omega(t)\subseteq \T\times [\vartheta_0,1-\vartheta_0].
 \end{equation}
   Then, for any smooth cutoff function $\Upsilon\in\mathcal{G}^{1,3/4}$ with ${\rm supp}\,\Upsilon\subseteq[\vartheta_0/20, 1-\vartheta_0/20]$ and $t\in[0,T]$ we have
\begin{equation}\label{ini2}
\begin{split}
\left\|\langle\nabla\rangle^5\,(\Upsilon\psi)(t)\right\|_{\mathcal{G}^{\lambda(t),s}}+\left\|\langle\nabla\rangle^3\,\omega(t)\right\|_{\mathcal{G}^{\lambda(t),s}}&\leq \exp{\Big[C_{\ast}\int_0^t(\|\omega(s)\|_{H^6}+1)ds\Big]}\|\langle\nabla\rangle^3\omega_0\|_{\mathcal{G}^{\lambda_0,s}},\\
\end{split}
\end{equation}
if we choose, for some large constant $C_\ast=C_\ast(\vartheta_0)>1$,
\begin{equation}\label{fdLam1}
\lambda(t):=\lambda_0\exp{\Big\{-C_{\ast}A\,t\exp\Big[C_\ast\int_0^t(\|\omega(s)\|_{H^6}+1)ds\Big]-C_{\ast}t\Big\}}.
\end{equation}
\end{lemma}

We note that an important aspect of the regularity theory for Euler equations in Gevrey spaces is the shrinking in time, at a fast rate, of the radius of convergence (the function $\lambda(t)$ in Lemma \ref{lm:persistenceofhigherregularity}). In our case, the support assumption (\ref{SupA1}) on $\omega(t)$ is satisfied if $T=2$, as a consequence of the smallness and the support assumptions on $\omega_0$, and the standard local well-posedness theory in Sobolev spaces of the Euler equation \eqref{Eur1}. In fact, as we show below, it is satisfied as part of the bootstrap argument for all $t\in[0,\infty)$.

\subsection{Proof of Theorem \ref{maintheoremINTRO}} For the purpose of proving continuity in time of the energy functionals $\mathcal{E}_g$ and $\mathcal{B}_g$, we make the {\it{a priori} } assumption that $\omega_0\in\mathcal{G}^{1,2/3}$. The argument is  similar to the argument in \cite[Section 3]{IOJI}, and we will be somewhat brief. We divide the proof in several steps.

{\bf{Step 1.}} Given small data $\omega_0$ satisfying \eqref{Eur0} we apply first Lemma \ref{lm:persistenceofhigherregularity}. Therefore $\omega\in C([0,2]:\G^{\lambda_1,2/3})$, $\lambda_1>0$, satisfies the quantitative estimates
\begin{equation}\label{smallnessofomega}
\sup_{t\in[0,2]}\big\|e^{\beta'_0\langle k,\xi\rangle^{1/2}}\widetilde{\omega}(t,k,\xi)\big\|_{L^2_{k,\xi}}\lesssim \eps,
\end{equation}
for some $\beta'_0=\beta'_0(\beta_0,\vartheta_0)>0$. In addition, letting $\Psi'\in\mathcal{G}^{1,3/4}$ denote a cutoff function supported in $[\vartheta_0/8,1-\vartheta_0/8]$ and equal to $1$ in $[\vartheta_0/4,1-\vartheta_0/4]$, the localized stream function $\Psi'\psi$ satisfies similar bounds,
\begin{equation}\label{smallnessofpsi}
\sup_{t\in[0,2]}\big\|\langle k,\xi\rangle^2e^{\beta'_0\langle k,\xi\rangle^{1/2}}\widetilde{(\Psi'\psi)}(t,k,\xi)\big\|_{L^2_{k,\xi}}\lesssim \eps.
\end{equation}

Recalling the definition \eqref{Gevr2}, and using the formula \eqref{rea15} and Lemmas \ref{lm:Gevrey}--\ref{GPF} it follows that there is a constant $K_1=K_1(\beta_0,\vartheta_0)$ such that 
\begin{equation}\label{proo1}
\|v(t,.)\|_{\widetilde{G}_{K_1}^{1/2}[0,1]}\lesssim 1,\qquad \|\mathcal{Y}(t,.)\|_{\widetilde{G}_{K_1}^{1/2}[b(0),b(1)]}\lesssim 1,
\end{equation}
for any $t\in[0,2]$, where $\mathcal{Y}(t,v)$ denotes the inverse of the function $y\to v(t,y)$. 

We would like to show now that 
\begin{equation}\label{init1}
\sum_{g\in\{F, F^\ast, F-F^\ast, \Theta, \Theta^\ast, B'_\ast, B''_\ast, V'_\ast,\mathcal{H}\}}\sup_{t\in[0,2]}\big\|e^{2\delta_0\langle k,\xi\rangle^{1/2}}\widetilde{g}(t,k,\xi)\big\|_{L^2_{k,\xi}}\lesssim \eps,
\end{equation}
for some constant $\delta_0=\delta_0(\beta_0,\vartheta_0)>0$ sufficiently small. Indeed, this follows using again Lemma \ref{GPF} and Lemma \ref{lm:Gevrey} (i) if $g\in\{F, \Theta, B'_\ast, B''_\ast, V'_\ast,\mathcal{H}\}$. To bound $F^{\ast}, F-F^{\ast}$, and $\Theta^\ast$ we use Green's functions.  Indeed, it follows from \eqref{Ph1} and the identity \eqref{elli12} that
\begin{equation}\label{phi'1}
\Psi(v)\phi'_k(t,v)=-\int_{\mathbb{R}}\mathcal{G}'_k(v,v')\frac{F_k(t,v')}{B'_0(v')}e^{ikt(v-v')}\,dv',
\end{equation}
where $\mathcal{G}'_k(v,w)=\Psi(v)\mathcal{G}_k(v,w)\Psi(w)$ as before. Using \eqref{elli13}--\eqref{elli15}, we obtain
\begin{equation}\label{phi'2}
\begin{split}
\left|\widetilde{(\Psi \phi'_k)}(t,\xi)\right|&=C\left|\int_{\mathbb{R}}K(\xi-kt,kt-\zeta)\widetilde{F_k}(t,\zeta)\,d\zeta\right|\lesssim \int_{\mathbb{R}}\frac{e^{-4\delta_0\langle\xi-\zeta\rangle^{1/2}}}{k^2+(\xi-kt)^2}\left|\widetilde{F_k}(t,\zeta)\right|\,d\zeta.
\end{split}
\end{equation}
for any $t\in[0,2]$. This gives \eqref{init1} for $g=\Theta^\ast$, and then for $g=F^\ast$, using \eqref{reb14}. This completes the proof of the desired bounds \eqref{init1}. In particular, the bounds \eqref{boot1} follow from \eqref{init1} if $\delta_0$ is $\eps_1\approx \eps^{2/3}$, see \eqref{reb11}-\eqref{reb13}.

{\bf{Step 2.}} Assume now that the solution $\omega$ satisfies the bounds in the hypothesis of Proposition \ref{MainBootstrap} on a given interval $[0,T]$, $T\geq 1$. We would like to show that the support of $\omega(t)$ is contained in $\T\times\big[3\vartheta_0/2,1-3\vartheta_0/2\big]$ for any $t\in[0,T]$. Indeed, for this we notice that only transportation in the $y$ direction, given by the term $u^y\,\partial_y\omega$, could enlarge the support of $\omega$ in $y$ outside $[b(2\vartheta_0),b(1-2\vartheta_0)]$. Notice that on $\mathbb{T}\times[\vartheta_0,1-\vartheta_0]$,
\begin{equation}\label{uyM}
u^y(t,x,y)=(\partial_x\psi)(t,x,y)=\partial_zP_{\neq0}\big(\Psi\phi\big)(t,x-tv(t,y),v(t,y)).
\end{equation}
Using the bound on $\mathcal{E}_{\Theta}$ from \eqref{boot2}, we can bound, for all $t\in[0,T]$,
\begin{equation}\label{integuy}
\sup_{(x,y)\in \mathbb{T}\times [\vartheta_0,1-\vartheta_0]}\big|u^y(t,x,y)\big|\lesssim \epsilon_1\langle t\rangle^{-2}.
\end{equation}
Since the support of $\omega(0)$ is contained in $\mathbb{T}\times[2\vartheta_0,1-2\vartheta_0]$, we can conclude that ${\rm supp}\,\omega(t)\subseteq \mathbb{T}\times\big[3\vartheta_0/2,1-3\vartheta_0/2\big]$ for any $t\in[0,T]$, 
as long as $\epsilon_1$ is sufficiently small.  

We can now use Proposition \ref{MainBootstrap} and a simple continuity argument to show that if $\omega_0\in\G^{1,2/3}$ has compact support in $\T\times [2\vartheta_0,1-2\vartheta_0]$ and satisfies the assumptions \eqref{Eur0}, then the solution $\omega$ is in $C([0,\infty):\G^{1,3/5})$, has compact support in $[\vartheta_0,1-\vartheta_0]$ and satisfies $\|\langle\omega\rangle(t)\|_{H^{10}}\lesssim\eps^{2/3}$  for all $t\in[0,\infty)$. Moreover, the variables $F, F^\ast, F-F^\ast, \Theta, \Theta^\ast, B'_\ast, B''_\ast, V'_\ast,\mathcal{H}$ satisfy the improved bounds \eqref{boot3}--\eqref{boot3'}. In particular, since $A_k(t,\xi)\ge e^{1.1\delta_0\langle k,\xi\rangle^{1/2}}$ and $A_R(t,\xi)\geq A_{NR}(t,\xi)\geq e^{1.1\delta_0\langle\xi\rangle^{1/2}}$, for any $t\in[0,\infty)$ we have
\begin{equation}\label{uniformf'}
\big\|e^{\delta_0\langle k,\xi\rangle^{1/2}}\widetilde{F}(t,k,\xi)\big\|_{L^2_{k,\xi}}+\langle t\rangle^2\big\|ke^{\delta_0\langle k,\xi\rangle^{1/2}}\widetilde{(\Psi\phi)}(t,k,\xi)\big\|_{L^2_{k,\xi}}\lesssim_{\delta} \epsilon_1^{3/2},
\end{equation}
and
\begin{equation}\label{uniformf}
\big\|V'_\ast(t)\big\|_{\mathcal{G}^{\delta_0,1/2}}+\big\|B'_\ast(t)\big\|_{\mathcal{G}^{\delta_0,1/2}}+\big\|B''_\ast(t)\big\|_{\mathcal{G}^{\delta_0,1/2}}+\langle t\rangle^{3/4}\big\|\mathcal{H}(t)\big\|_{\mathcal{G}^{\delta_0,1/2}}\lesssim \epsilon_1.
\end{equation}

{\bf{Step 3.}} We show now that for any $t\in[0,\infty)$ 
\begin{equation}\label{conv1}
\langle t\rangle \|\mathcal{H}(t)\|_{\mathcal{G}^{\delta_1,1/2}}+\langle t\rangle^2\|\dot{V}(t)\|_{\mathcal{G}^{\delta_1,1/2}}\lesssim \epsilon_1,
\end{equation}
where $\delta_1=\delta_1(\delta_0)>0$. We use the equation \eqref{rea25}, thus
\begin{equation}\label{conv1.2}
\partial_t(t\mathcal{H})=-t\dot{V}\partial_v\mathcal{H}+tV'\big\{-\big<\partial_vP_{\neq0}\phi\,\partial_zF\big>+\big<\partial_z\phi\,\partial_vF\big>\big\}=-t\dot{V}\partial_v\mathcal{H}+tV'\partial_v\big<\partial_z\phi F\big>.
\end{equation}
Since $V'=V'_\ast+B'_0$ and $V''=(1/2)\partial_v(V')^2$, it follows from \eqref{uniformf} and Lemma \ref{lm:Gevrey} that 
\begin{equation}\label{conv1.3}
\|V'(t)\|_{\widetilde{\mathcal{G}}^{1/2}_{K}[b(0),b(1)]}+\|V''(t)\|_{\widetilde{\mathcal{G}}^{1/2}_{K}[b(0),b(1)]}+\|B'(t)\|_{\widetilde{\mathcal{G}}^{1/2}_{K}[b(0),b(1)]}+\|B''(t)\|_{\widetilde{\mathcal{G}}^{1/2}_{K}[b(0),b(1)]}\lesssim 1
\end{equation}
for any $t\in [0,\infty)$, for some $K=K(\delta_0)>0$. Using also \eqref{rea26}, we have
\begin{equation}\label{conv1.32}
\big<\partial_z\phi F\big>=(V')^2\langle\partial_z\phi\cdot(\partial_v^2\phi-2t\partial_v\partial_z\phi)\rangle+V''\langle\partial_z\phi\cdot(\partial_v\phi-t\partial_z\phi)\rangle.
\end{equation}
In particular, using the bounds on $\phi$ in \eqref{uniformf'}, $\|\big<\partial_z\phi F\big>(t)\|_{\mathcal{G}^{\delta_0/2,1/2}}\lesssim\eps_1^2\langle t\rangle^{-3}$. Using also \eqref{uniformf'}--\eqref{uniformf} and the identity $\mathcal{H}=tV'\partial_v{\dot{V}}$, it follows from \eqref{uniformf} that $\big\|\partial_t(t\mathcal{H})(t)\big\|_{\mathcal{G}^{\delta_0/2,1/2}}\lesssim \epsilon_1\langle t\rangle^{-3/2}$ for any $t\in[0,\infty)$, and the desired bounds \eqref{conv1} follow.  

As a consequence, we also have the bounds
\begin{equation}\label{conv1.35}
\|v(t)\|_{\mathcal{G}^{1/2}_{K_2}[0,1]}+\langle t\rangle^2\|(\partial_tv)(t)\|_{\mathcal{G}^{1/2}_{K_2}[0,1]}\lesssim 1
\end{equation}
for any $t\in [0,\infty)$, for some $K_2=K_2(\delta_0)>0$. Indeed, the bounds on $v$ follow from the identity $\partial_yv(t,y)=V'(t,v(t,y))$, the bounds \eqref{conv1.3}, and Lemma \ref{GPF}. The bounds on $\partial_tv$ then follow using the identity $\partial_tv(t,y)=\dot{V}(t,v(t,y))$, the bounds \eqref{conv1}, and Lemma \ref{GPF}.

{\bf Step 4.} We prove now the conclusions of the theorem. Notice that 
\begin{equation}\label{LfM}
\partial_tF-B''\partial_z(\Psi\phi)-V'\partial_vP_{\neq 0}(\Psi\phi)\,\partial_zF+\dot{V}\,\partial_vF+V'\partial_z(\Psi\phi)\,\partial_vF=0,
\end{equation}
using (\ref{rea23.1}) and ${\rm supp}\,F(t)\subseteq[b(\vartheta_0),b(1-\vartheta_0)]$. Using the bounds \eqref{uniformf'}--\eqref{conv1}, and \eqref{conv1.3}, it follows that $\big\|\partial_tF\big\|_{\mathcal{G}^{\delta_2,1/2}}\lesssim \epsilon_1^{3/2}\langle t\rangle^{-2}$, for some $\delta_2>0$.
Moreover, the definitions \eqref{changeofvariables1}--\eqref{rea1} show that
\begin{equation*}
\omega(t,x+tb(y)+\Phi(t,y),y)=\omega(t,x+tv(t,y),y)=F(t,x,v(t,y)).
\end{equation*}
Using also \eqref{conv1.35}, we have
\begin{equation*}
\Big\|\frac{d}{dt}\big[\omega(t,x+tb(y)+\Phi(t,y),y)\big]\Big\|_{\mathcal{G}^{\delta_3,1/2}}\lesssim \epsilon_1^{3/2}\langle t\rangle^{-2},
\end{equation*}
for some $\delta_3=\delta_3(\delta_0)>0$, and the bounds \eqref{convergence} follow.

Moreover, we notice that
\begin{equation}\label{conv2.1}
\psi(t,x,y)=\phi(t,x-tv(t,y),v(t,y))
\end{equation}
Since $u^y=\partial_x\psi$ and $u^x=-\partial_y\psi$, the bounds \eqref{convergencetomean}--\eqref{convergenceuy} follow from the bounds on $\phi$ in \eqref{uniformf'} and the fact that $\psi(t)$ is harmonic in $\mathbb{T}\times\big\{[0,\vartheta_0]\cup[1-\vartheta_0,1]\big\}$. 

Finally, to prove \eqref{convergenceofux} we start from the formula $\langle u^x\rangle=-\langle\partial_y\psi\rangle$, thus $\partial_y\langle u^x\rangle=-\langle\omega\rangle$. Therefore, using the evolution equation \eqref{IEur1}, 
\begin{equation*}
\partial_t\partial_y\langle u^x\rangle=\langle-\partial_t\omega\rangle=\langle u^x\partial_x\omega+u^y\partial_y\omega\rangle=\langle -\partial_y\psi\partial_x\omega+\partial_x\psi\partial_y\omega\rangle=\partial_y\langle\omega\partial_x\psi\rangle.
\end{equation*}
Moreover, since $\psi(t,x,0)=\psi(t,x,1)=0$, for any $t\in [0,\infty)$ we have
\begin{equation}\label{conv2.2}
\int_{[0,1]}\langle u^x\rangle(t,y)\,dy=-\int_{[0,1]}\partial_y\langle\psi\rangle (t,y)\,dy=0.
\end{equation}
Moreover, $\langle\omega\partial_x\psi\rangle=\langle\partial_y^2\psi\partial_x\psi\rangle=\partial_y\langle\partial_y\psi\partial_x\psi\rangle$. These identities show that 
\begin{equation}\label{conv2.3}
\partial_t\langle u^x\rangle=\langle\omega\partial_x\psi\rangle\qquad\text{ in }\qquad[0,\infty)\times[0,1].
\end{equation}

Using the definitions \eqref{rea1}, we have
\begin{equation*}
(\partial_t\langle u^x\rangle)(t,y)=\frac{1}{2\pi}\int_{\mathbb{T}}\omega(t,x,y)\partial_x\psi(t,x,y)\,dx=\frac{1}{2\pi}\int_{\mathbb{T}}F(t,z,v(t,y))\partial_z\phi(t,z,v(t,y))\,dz.
\end{equation*}
Using now \eqref{conv1.32}, \eqref{conv1.35}, and the bounds on $\phi$ in \eqref{uniformf'} it follows that $\|(\partial_t\langle u^x\rangle)(t)\|_{\mathcal{G}^{\delta_4,1/2}}\lesssim\eps_1^2\langle t\rangle^{-3}$, for some $\delta_4=\delta_4(\delta_0)>0$. Moreover, using \eqref{convergence},
\begin{equation*}
\lim_{t\to\infty}\{\partial_y\langle u^x\rangle(t)+\partial_y^2\psi_\infty\}=\lim_{t\to\infty}\{-\langle\omega\rangle(t)+\langle F_\infty\rangle\}=0. 
\end{equation*}
The desired conclusion \eqref{convergenceofux} follows using also \eqref{conv2.2}.

\end{document}